\documentclass[reqno,12pt]{amsart}
\pdfoutput=1
\usepackage{hbpaper,verbatim,genyoungtabtikz}
\usepackage{enumitem,ytableau}
\usepackage{youngtab}
\usepackage{ytableau}

\allowdisplaybreaks

\newcommand{\Arm}{\operatorname{Arm}}
\newcommand{\Leg}{\operatorname{Leg}}
\newcommand{\Gar}{\operatorname{Gar}}
\def\bi{\text{\boldmath$i$}}

\newcommand\std{\operatorname{Std}} %standard
\newcommand\Std{\operatorname{Std}} %standard
\newcommand\tabupto[2]{#1_{\downarrow#2}}
\newcommand{\Ht}{{\rm ht}} %height
\newtheorem*{theorem*}{Theorem}

\title[Homomorphisms to hooks in type C]{Homomorphisms to Hook Specht Modules over Quiver Hecke algebras of type C}

\author[Mart\'{i}n Forsberg Conde]{Mart\'{i}n Forsberg Conde}
\address{Okinawa Institute of Science and Technology, Okinawa 904-0495, Japan}
\email{martin.forsberg@oist.jp}

\author[Berta Hudak]{Berta Hudak}
\address{National Center for Theoretical Sciences, Taipei, Taiwan}
\email{berta.hudak@ncts.ntu.edu.tw}

\date{}

\begin{document}

\renewcommand\auth{Mart\^{\i}n Forsberg Conde & Berta Hudak}

\let\thefootnote\relax

\begin{abstract}
We investigate homomorphisms between Specht modules for level $1$ cyclotomic quiver Hecke algebras of affine type $C$.
For hook partitions with a `short leg' or a `short arm', we give the full quiver Hecke algebra action and we show that every non-zero homomorphism from an arbitrary Specht module is determined by sending the cyclic generator to a single standard basis element. We use this action to investigate homomorphisms to these Specht modules, giving a full classification in the short leg case.
This provides the first investigation of homomorphisms between Specht modules in affine type $C$ and extends methods previously used in affine type $A$.
\end{abstract}

\maketitle

\section{Introduction} \label{S:Intro}

The study of the structure of Specht modules is the central problem in the modular representation theory of symmetric groups and related algebras. In type A (and over a field $k$), the Brundan--Kleshchev isomorphism~\cite{bkisom} provides good guidance to a mathematician working on quiver Hecke algebras. The representation theory of cyclotomic Hecke algebras (not to mention symmetric groups) is a much older subject, and the isomorphism immediately suggests that any result about Specht modules of cyclotomic Hecke algebras should be provable in the quiver Hecke algebra formalism. Work in this direction helps build intuition about quiver Hecke algebras in general, and reveals the (previously hidden) graded structure of cyclotomic Hecke algebras. 

However, in type C, there is no analogous result to the Brundan--Kleshchev isomorphism. The general approach to studying quiver Hecke algebras of type C is to reason by analogy with type A, and find out the differences through computations. Ariki, Park, and Speyer~\cite{aps} recently  constructed Specht modules in type C, and Evseev and Mathas~\cite{em24} showed that the algebras are graded cellular. There have also been studies on their representation type~\cite{APC,CH23,ahsw23}. 

In this paper, we focus on homomorphisms between Specht modules of type C. This problem has not been studied before, and only special cases can be solved in type A. To have any hope of reaching a result, we should restrict the conditions in some way. A non-trivial yet manageable restriction, motivated by results in type A, is to require that $\ell = 1$ and that the label $\lambda$ of the target Specht module $S^\lambda$ is a hook-shaped partition. In type A, James~\cite{James} gave a complete description of homomorphisms to the trivial module. Much later, Loubert~\cite{loub17} gave a complete description of homomorphisms to Specht modules labelled by a hook-shaped $\lambda$ under the restriction $e > 2$. A partial description of the situation when $e = 2$ was given by Hudak~\cite{Hudak24}. The results by Loubert and Hudak were proven using the new quiver Hecke algebra formalism. We will give a full description of homomorphisms to Specht modules labelled by hook-shaped partitions $\lambda = (a+1, 1^b)$ which have a `short leg' (that is, $b < 2e$), and a partial description of the corresponding homomorphisms when $\lambda$ has a `short arm' (that is, $a < 2e$). Under the short leg/arm restriction, the quiver Hecke algebra action is tractable, which allows us to give a full description of the homomorphisms in most cases. 

Two key features of our description are:
\begin{enumerate}
    \item The homomorphisms we describe map the Specht module generator $v_{\tabt^\mu}$ to a single basis element of $S^\lambda$, which follows from \cref{Cor:YAction}.
    \item  All the homomorphisms that we find for short leg hook-shaped modules in type $C_e^{(1)}$ have a counterpart in type $A_{2e-1}^{(1)}$, according to the Main Theorem of \cite{loub17}, as we describe in \cref{SS:AC-Connection}.
\end{enumerate}

The structure of this text is as follows. 
In \cref{S:Background}, we give the necessary background on quiver Hecke algebras of type $C_e^{(1)}$, their Specht modules, and string diagrams. We review the relevant tableaux combinatorics under the different residue structure that applies in type $C_e^{(1)}$.

In \cref{S:Action}, we describe the action of the quiver Hecke algebra on hook Specht modules with a short arm or short leg. This serves as our main technical tool, and it provides the corollary that all homomorphisms to these Specht modules send the Specht module generator to a single basis element.

In \cref{S:ShortLeg}, we present a full explicit description of the homomorphisms from an arbitrary Specht module $S^\mu$ to a Specht module $S^\la$, where $\la$ has a short leg.

In \cref{S:ShortArm}, we discuss the analogous case where $\la$ has a short arm. The case where the charge of the hook is $0$ or $e$ is especially hard in type $C$ and we were not able to obtain a full characterization, so we limit the discussion to some examples.

\section{Background} \label{S:Background}

\subsection{Lie theory notation}\label{SubS:lie}

We follow \cite{kac} and use standard notation for the root datum.

Let $e\in \{2,3,\dots\}$ and $I= \{ 0,1,2, \ldots, e \} $ and suppose we are given a quiver of affine type $C_e^{(1)}$, as follows.
\[
\begin{tikzpicture}[scale=2.3]
  \dynkin[
    ordering    = Kac,
    labels      = {0,1,2,,e{-}1,e},
    mark        = o,
    edge/.style = {-stealth, shorten >=3pt},
  ] C[1]{}
\end{tikzpicture}
\]
We assign integer labels in $I$ from left to right to each vertex of the quiver. 

Then we have the \emph{simple roots} $\Pi=\{\alpha_i \mid i\in I\}$ and the \emph{fundamental weights} $\{\Lambda_i \mid i\in I \}$ in the \emph{weight lattice} $P$. 

There is a symmetric bilinear form $( - , - )$ on $P$ satisfying the following:
\begin{enumerate}
    \item $( \Lambda_i, \alpha_j ) = d_j \delta_{ij}$; and 
    \item $( \alpha_i, \alpha_j ) = d_i a_{ij}$,
\end{enumerate}
where $d = (2,1,\dots,1,2)$ if $e<\infty$.
We call $Q:=\bigoplus_{i\in I} \bbz \alpha_i$ the \emph{root lattice} and $Q^+ = \bigoplus_{i\in I} \bbz_{\ge 0} \alpha_i$ the \emph{positive cone} of the root lattice.
For $\beta=\sum_{i\in i}a_i\al_i\in Q^+$, let $\Ht(\beta)$ denote the \emph{height} of $\beta$, that is $\Ht(\beta)=\sum_{i\in I}a_i$.

Let $\fkS_n = \langle s_k \mid k=1, \ldots, n-1 \rangle$ denote the symmetric group on $n$ letters.
The group $\fkS_n$ acts on the set $I^n$ by place permutations.
If $\bi=(i_1,\dots,i_n)\in I^n$, then its \emph{weight} is $|\bi|:=\al_{i_1}+\dots+\al_{i_n}\in Q^+$.

For our purposes, the quiver $C_\infty$ can simply be taken to result from the limit $e \rightarrow \infty$ in the above construction.

\subsection{Combinatorics}

In this section, we recall some standard combinatorial notions.

A \emph{partition} of $n\ge0$ is a weakly decreasing sequence of non-negative integers $\mu = (\mu_1, \mu_2, \dots)$ such that $|\mu|=\mu_1+\mu_2+\dots=n$.
If $\mu$ is a partition of $n$, we write $\mu\vdash n$.
We denote the set of partitions of $n$ by $\calp_n$.
For any $\mu\in \calp_n$, its \emph{Young diagram} is defined as the set
\[
[\mu]=\{(r,c) \in \bbz_{\ge 1} \times \bbz_{\ge 1}  \mid c\le \mu_r\}.
\]

For $\mu\vdash n$, we say that $\mu$ is a \emph{partition of hook shape} (or simply a \emph{hook}) if $\mu_i\le 1$ for all $i\ge2$.
In this paper, we are concerned with a special shape of hooks:

\begin{defn}
    A short leg (resp.\ arm) hook is a partition $\mu=(a+1,1^b)$ such that $b<2e$ (resp.\ $a < 2e$).
\end{defn}

We define $\overline{\mkern2mu\cdot\mkern2mu} : \bbz \rightarrow I$ by $k \mapsto |k|$ if $e = \infty$ and, otherwise, by
\begin{equation}\label{res}
\begin{aligned}
&\overline{0 + 2e\bbz} = 0, \quad \overline{e+2e\bbz} = e, \\
 &\overline{k + 2e\bbz} = \overline{2e-k + 2e\bbz} = k \quad  \text{ for } 1 \le k \le e-1.
\end{aligned}
\end{equation}
Note that this function is periodic.

Given a \emph{charge} $\kappa\in \bbz$, we define $\Lambda_\kappa\in P^+$ by $\Lambda_\kappa:= \Lambda_{\overline{\kappa}}.$
If $\mu\vdash n$, we associate to any node $A = (r,c) \in [\mu]$ its \emph{content} defined by 
\[
\cont A = \kappa+c-r
\]
and its \emph{residue} defined by
\[
\res A = \overline{\kappa+c-r}.
\]
An $i$-node is a node with residue $i$.

\begin{eg}
If $\mu=(7,4,4,3,2,2)$, $\kappa=1$ and $e=3$, then $\mu$ has the following residue pattern.
\[
\ColorTableau[]{{1,2,3,2,1,0,1},{0,1,2,3},{1,0,1,2},{2,1,0},{3,2},{2,3}}
\]
\end{eg}

We call a node $A\in[\mu]$ \emph{removable} (resp.~\emph{addable}) if $[\mu]\setminus A$ (resp.~$[\mu]\cup A$) is a valid Young diagram for a partition of $n-1$ (resp.~$n+1$).
For an $i$-node $A\in[\mu]$, we set
\begin{equation*}
d_A(\mu) = d_i \cdot (\#\{\text{addable $i$-nodes of $[\mu]$ below } A\} - \#\{\text{removable $i$-nodes of $[\mu]$ below } A\}).
\end{equation*}
(Recall that $d = (2,1\dots,1,2)$.)

Let $\mu \in \calp_n$.
A \emph{$\mu$-tableau} is a bijection $\tabt :[\mu] \rightarrow \{1,\dots,n\}$.
We depict $\tabt$ by filling each node $(r,c)\in [\mu]$ with $\tabt(r,c)$.
We say that the tableau $\tabt$ is \emph{standard} if its entries increase along rows and down the columns and denote the set of standard tableaux by $\std(\mu)$.

The symmetric group $\fkS_n$ acts on the set of $\mu$-tableaux on the right by its natural action on the entries of the tableaux.
Let $\tabt^\mu$ be the unique $\mu$-tableau in which the numbers $1,2,\dots,n$ appear in order from left to right along the successive rows, working from top to bottom.
We call $\tabt^\mu$ the \emph{(row) initial tableau}. The initial tableau gives an ordering of the nodes of the Young diagram of a partition $\mu$, so that we write $N < M$ if $\tabt^{\mu}(N) < \tabt^\mu(M)$.
Set 
$
\bi^\mu:=\bi^{\tabt^\mu}.
$
For every $\mu$-tableau $\tabt$, we define the permutation $w^\tabt \in \fkS_n$ by
$
  \tabt^\mu w^\tabt=\tabt.
$
We also define, for each $\mu$-tableau $\tabt$, its \emph{content sequence}
\begin{align*}
\bm{c}(\tabt) &= (c_1,c_2,\dots,c_n)\\
&=( \cont{\tabt^{-1}(1)}  ,\   \cont \tabt^{-1}(2) , \dots,  \cont \tabt^{-1}(n) ).
\end{align*}
and its \emph{residue sequence}
\begin{align*}
\bi(\tabt) &= (i_1,i_2,\dots,i_n)\\
&=(  \res { \tabt^{-1}(1) } ,\   \res { \tabt^{-1}(2)} , \dots,  \res { \tabt^{-1}(n)} ).
\end{align*}
We write $\bi^\mu$ for the residue sequence of the initial tableau $\tabt^\mu$. We write $\ell(\resi)$ for the length of the residue sequence $\resi$.

Let $0\le m\le n$.
For a $\mu$-tableau $\tabt$, we denote by $\tabupto\tabt m$ the set of nodes of $[\mu]$ whose entries are less than or equal to $m$.
If $\tabt$ is standard, then $\tabupto\tabt m$ is also a standard tableau for some partition.

For any $\mu\in\calp_n$ and $\tabt\in\std(\mu)$, we define the \emph{degree} $\deg\tabt$ of $\tabt$ inductively.
If $n=0$ then $\tabt$ is the unique empty tableau and we set $\deg\ttt:=0$.
Otherwise, let $A=\tabt^{-1}(n)\in[\mu]$ and suppose $A$ is an $i$-node. We set
\begin{align*} %\label{deg}
\deg\tabt := \deg\tabupto\tabt{n-1} + d_A(\mu).
\end{align*}

\begin{eg}
If $\mu=(3,3,2,1)$, $\kappa=1$ and $e=2$, then $\mu$ has the following residue pattern
\[
\young(121,012,10,2)
\]
and if $\ttt$ is the tableau
\[
\young(134,258,69,7)
\]
then
\[
\deg\ttt=0+0+0+(1+1)+1+0+0-2+0=1.
\]
\end{eg}

\subsection{KLR algebras of type C}\label{subs:klr-alg}

Let us work over a field $k$ of characteristic $p$. For simplicity, let us assume $e < \infty$. We set polynomials $\calq_{i,j}(u,v)\in k[u,v]$ satisfying $\calq_{i,j}(u,v) = \calq_{j,i}(v,u)$, for $i,j\in I$, of the form
\begin{align}
\calq_{i,j}(u,v) = 
\begin{cases} 
    u + v^2 & \text{if } (i,j) = (0, 1), \\
    u + v   & \text{if } i \neq 0, j = i + 1, j \neq \ell, \\
    u^2 + v & \text{if } (i, j) = (\ell - 1, \ell), \\
    1       & \text{otherwise,}
\end{cases}
\end{align}
following \cite[eq (2.1)]{aps}. The algebra $R_n$ associated with the polynomials $\calq_{i,j}(u,v)$ is defined to be the unital associative diagrammatic $k$-algebra with generators 

\begin{equation} \label{E:generatorsC}
\begin{gathered} 
\def\h{3}
\psi_r = \begin{braid}
	\draw (0,0) node[anchor=north]{$1$}-- (0,\h);
	\draw (1,0) node[anchor=north]{$2$}-- (1,\h);
	\draw (1.5+\sep/2, \h/2) node{$\cdots$};
	\draw (2+\sep, 0) node[anchor=north]{$r$}-- (3 + \sep, \h);
	\draw (3 + \sep, 0) node[anchor=north]{$r+1$}-- (2 + \sep, \h);
	\draw (3.5+\sep + \sep/2, \h/2) node{$\cdots$};
	\draw (4 + 2*\sep, 0) node[anchor=north]{$n-1$} -- (4 + 2*\sep,\h);
	\draw (5+2*\sep, 0) node[anchor=north]{$n$}-- (5 + 2*\sep, \h);
\end{braid}, \qquad
y_r = \begin{braid}
	\draw (0,0) node[anchor=north]{$1$}-- (0,\h);
	\draw (1,0) node[anchor=north]{$2$}-- (1,\h);
	\draw (1.5+\sep/2, \h/2) node{$\cdots$};
	\draw (2+\sep, 0) node[anchor=north]{$r$}-- (2 + \sep, \h);
	\greendot(2+\sep, \h/2);
	\draw (2.5+\sep + \sep/2, \h/2) node{$\cdots$};
	\draw (3 + 2*\sep, 0) node[anchor=north]{$n-1$} -- (3 + 2*\sep,\h);
	\draw (4+2*\sep, 0) node[anchor=north]{$n$}-- (4 + 2*\sep, \h);
\end{braid}, \\
e(\bm{i}) = \begin{braid}
	\draw (0,0) node[anchor=north]{$1$}-- (0,\h) node[anchor= south]{$i_{1}$};
	\draw (1,0) node[anchor=north]{$2$}-- (1,\h)node[anchor= south]{$i_{2}$};
	\draw (1.5+\sep/2, \h/2) node{$\cdots$};
	\draw (2 +\sep, 0) node[anchor=north]{$n-1$} -- (2 + \sep,\h)node[anchor= south]{$i_{n-1}$};
	\draw (3+\sep, 0) node[anchor=north]{$n$}-- (3 + \sep, \h)node[anchor= south]{$i_{n}$};
\end{braid}
\end{gathered}
\end{equation}

under the following relations. First, the double crossing:
\begin{equation} \label{E:psisquaredC}
        \begin{braid}
            \draw (0,0) -- (1, \h/2) -- (0, \h) node[anchor=south]{$i$};
            \draw (1,0) -- (0, \h/2) -- (1, \h) node[anchor=south]{$j$};
        \end{braid}
      =  
      \def\h{5/3}
      \begin{cases}
          \eqmakebox[lhs][c]{$0$} &\text{if } i = j,\\[10pt]
          \eqmakebox[lhs][c]{$\begin{braid}
            	\draw (0,0) -- (0, \h) node[anchor=south]{$i$};
            	\draw (1,0) -- (1, \h) node[anchor=south]{$j$};
            	\greendot(1,\h/2);
                \end{braid} +
          \begin{braid}
            	\draw (0,0) -- (0, \h) node[anchor=south]{$i$};
            	\draw (1,0) -- (1, \h) node[anchor=south]{$j$};
            	\greendot(0,\h/2);
                \end{braid}$} &\text{if  }i\rightarrow j \text{ or }i\leftarrow j,\\[20pt]
          \eqmakebox[lhs][c]{$\begin{braid}
                \draw (0,0) -- (0, \h) node[anchor=south]{$i$};
            	\draw (1,0) -- (1, \h) node[anchor=south]{$j$};
            	\greendot(1,2*\h/3);
                \greendot(1,\h/3);
                \end{braid} + 
                \begin{braid}
            	\draw (0,0) -- (0, \h) node[anchor=south]{$i$};
            	\draw (1,0) -- (1, \h) node[anchor=south]{$j$};
            	\greendot(0,\h/2);
          \end{braid}$} &\text{if } \begin{tikzpicture}[inner sep=2pt,outer sep=0pt,baseline=-3pt]
            \node (i) at (1,0) {$i$};
            \node (C) at (1.3,0) {};
            \node (j) at (1.6,0) {$j$};
            \draw[thin,double] (C.west) -- (j);
            \draw[->,thin,double] (i) -- (C.east);
          \end{tikzpicture},\\[20pt]
          \eqmakebox[lhs][c]{$\begin{braid}
                \draw (0,0) -- (0, \h) node[anchor=south]{$i$};
            	\draw (1,0) -- (1, \h) node[anchor=south]{$j$};
            	\greendot(0,2*\h/3);
                \greendot(0,\h/3);
                \end{braid} + 
                \begin{braid}
            	\draw (0,0) -- (0, \h) node[anchor=south]{$i$};
            	\draw (1,0) -- (1, \h) node[anchor=south]{$j$};
            	\greendot(1,\h/2);
          \end{braid}$} &\text{if } \begin{tikzpicture}[inner sep=2pt,outer sep=0pt,baseline=-3pt]
            \node (i) at (1,0) {$i$};
            \node (C) at (1.3,0) {};
            \node (j) at (1.6,0) {$j$};
            \draw[thin,double] (C.west) -- (j);
            \draw[-<,thin,double] (i) -- (C.east);
          \end{tikzpicture},\\[20pt]
          \eqmakebox[lhs][c]{$\begin{braid}
            	\draw (0,0) -- (0, \h) node[anchor=south]{$i$};
            	\draw (1,0) -- (1, \h) node[anchor=south]{$j$};
                \end{braid}$} &\text{otherwise.}\\
        \end{cases}
\end{equation}
Next is the braid relation:
\begin{equation} \label{E:braidC}
    \begin{braid}
    	\draw (0,0) -- (2, \h) node[anchor=south]{$k$};
    	\draw (2,0) -- (0,\h) node[anchor=south]{$i$};
    	\draw (1, 0) -- (0, \h/2) -- (1, \h)node[anchor=south]{$j$};
\end{braid}
  = 
 \def\h{3}
\begin{cases}
        \eqmakebox[lhs][c]{$\begin{braid}
        	\draw (0,0) -- (2, \h) node[anchor=south]{$k$};
        	\draw (2,0) -- (0,\h) node[anchor=south]{$i$};
        	\draw (1, 0) -- (2, \h/2) -- (1, \h)node[anchor=south]{$j$};
        \end{braid} -
        \begin{braid}
        	\draw (0,0) -- (0, \h) node[anchor=south]{$i$};
        	\draw (2,0) -- (2,\h) node[anchor=south]{$k$};
        	\draw (1, 0) -- (1, \h)node[anchor=south]{$j$};
        \end{braid}$} 
        & \text{if }i = k \rightarrow j \text{ or } i = k \leftarrow j\\[20pt]
        \eqmakebox[lhs][c]{$\begin{braid}
        	\draw (0,0) -- (2, \h) node[anchor=south]{$k$};
        	\draw (2,0) -- (0,\h) node[anchor=south]{$i$};
        	\draw (1, 0) -- (2, \h/2) -- (1, \h)node[anchor=south]{$j$};
        \end{braid} - 
        \begin{braid}
        	\draw (0,0) -- (0, \h) node[anchor=south]{$i$};
        	\draw (2,0) -- (2,\h) node[anchor=south]{$k$};
        	\draw (1, 0) -- (1, \h)node[anchor=south]{$j$};
         \greendot(0,\h/2);
        \end{braid} - 
        \begin{braid}
        	\draw (0,0) -- (0, \h) node[anchor=south]{$i$};
        	\draw (2,0) -- (2,\h) node[anchor=south]{$k$};
        	\draw (1, 0) -- (1, \h)node[anchor=south]{$j$};
            \greendot(2,\h/2);
        \end{braid}$} &\text{if } i = \begin{tikzpicture}[inner sep=2pt,outer sep=0pt,baseline=-3pt]
            \node (k) at (1,0) {$k$};
            \node (C) at (1.3,0) {};
            \node (j) at (1.6,0) {$j$};
            \draw[thin,double] (C.west) -- (j);
            \draw[->,thin,double] (k) -- (C.east);
          \end{tikzpicture} \text{ or } i = \begin{tikzpicture}[inner sep=2pt,outer sep=0pt,baseline=-3pt]
            \node (k) at (1,0) {$k$};
            \node (C) at (1.3,0) {};
            \node (j) at (1.6,0) {$j$};
            \draw[thin,double] (C.west) -- (j);
            \draw[-<,thin,double] (k) -- (C.east);
          \end{tikzpicture}\\[20pt]
        \eqmakebox[lhs][c]{$\begin{braid}
        	\draw (0,0) -- (2, \h) node[anchor=south]{$k$};
        	\draw (2,0) -- (0,\h) node[anchor=south]{$i$};
        	\draw (1, 0) -- (2, \h/2) -- (1, \h)node[anchor=south]{$j$};
        \end{braid}$} & \text{otherwise.}
\end{cases}
\end{equation}

Next is the dot slide:
\begin{equation} \label{E:dotslideC}
    \def\h{3}
    \begin{braid}
    	\draw (0,0) -- (2, \h) coordinate[pos = 0.25] (B) node[anchor=south]{$j$};
    	\draw (2,0) -- (0,\h) node[anchor=south]{$i$};
    	\greendot(B);
    \end{braid}
     = 
     \def\h{3}
     \begin{braid}
    	\draw (0,0) -- (2, \h) coordinate[pos = 0.75] (B) node[anchor=south]{$j$};
    	\draw (2,0) -- (0,\h) node[anchor=south]{$i$};
    	\greendot(B);
    \end{braid}
    -\delta_{i,j}
    \begin{braid}
    	\draw (0,0) -- (0, \h) node[anchor=south]{$i$};
    	\draw (1,0) -- (1,\h) node[anchor=south]{$j$};
    \end{braid}
    \qquad \text{and} \qquad
    \def\h{3}
    \begin{braid}
    	\draw (0,0) -- (2, \h)  node[anchor=south]{$j$};
    	\draw (2,0) -- (0,\h) coordinate[pos = 0.25] (B) node[anchor=south]{$i$};
    	\greendot(B);
    \end{braid}
    = 
    \def\h{3}
    \begin{braid}
    	\draw (0,0) -- (2, \h)  node[anchor=south]{$j$};
    	\draw (2,0) -- (0,\h) coordinate[pos = 0.75] (B) node[anchor=south]{$i$};
    	\greendot(B);
    \end{braid}+\delta_{i,j}
    \begin{braid}
    	\draw (0,0) -- (0, \h) node[anchor=south]{$i$};
    	\draw (1,0) -- (1,\h) node[anchor=south]{$j$};
    \end{braid}.
\end{equation}
The summands conserving the crossings in the right-hand side of \cref{E:braidC} and \cref{E:dotslideC} are termed \emph{Slide} terms, while the rest of the terms are called \emph{Error} terms.

Finally, the $e(\resi)$ form a set of mutually orthogonal idempotents.
\begin{equation} \label{E:idempotentrelationsC}
    e(\resi)e(\resj)= \delta_{\resi, \resj} e(\resi), \qquad \sum_{\resi \in I^n} e(\resi) = 1.
\end{equation}

The algebra $R_n$ is $\bbz$-graded by setting
\begin{equation}
    \deg(e(\resi))=0, \qquad \deg(y_r e(\resi)) = (\alpha_{i_r},\alpha_{i_r}) = d_{i_r}, \qquad \deg(\psi_s e(\resi)) = -( \alpha_{i_s}, \alpha_{i_{s+1}})
\end{equation}
for all admissible $r,s$ and $\resi$.

For each element $w\in\fkS_n$, we fix a preferred reduced expression $w=s_{k_1}\dots s_{k_m}$ and define the corresponding element of $R_n$ by
\begin{equation}
    \psi_w=\psi_{s_{k_1}}\dots\psi_{s_{k_m}}.
\end{equation}

Note that unless $w\in\fkS_n$ is fully commutative, $\psi_w$ depends on this choice of reduced expression.

For $\La\in P^+$, we can define the algebra $R^\La_n$ as the quotient of $R_n$ by the additional cyclotomic relations $y_1^{\langle \alpha^\vee_{i_1},\Lambda \rangle} e(\resi)=0$ for $\resi \in I^n$ to obtain the cyclotomic quiver Hecke algebra.

\subsection{Specht modules}

The graded (row) Specht modules over quiver Hecke algebras of type $C$ were described in \cite{aps}. The simple Specht relations are the same as in type A.
\begin{equation}\label{E:specht}
    \def\h{3}
    \begin{braid}
    	\draw (0,0) -- (1, \h);
    	\draw (1,0) -- (0,\h);
        \draw[brown] (-1.7, \h) rectangle (2.7, \h + 1);
    \end{braid}
    = \quad 0 \qquad \text{and} \qquad
    \def\h{3}
    \begin{braid}
    	\draw (0,0) -- (0, \h) coordinate[pos=0.5] (C);
        \draw[brown] (-1.7, \h) rectangle (1.7, \h+1);
        \greendot(C);
    \end{braid}
    = \quad 0.
\end{equation}
The residue relation is also similar.
\begin{equation} \label{E:residuerelation}
v_{\tabt^\bla} e(\bm{i}) = \delta_{\bm{i}, \bm{i}^\bla} v_{\tabt^\bla}.
\end{equation}
The homogeneous Garnir relations in type $C_e^{(1)}$ were conjectured in \cite[Conjecture 5.3]{aps} to be the same as the Garnir relations in type $A_{2e-1}^{(1)}$. A forthcoming publication by Mathas~\cite{MathasGarnir} will confirm this, and we will use it freely.
\begin{equation} \label{E:GarnirC}
    v_{\tabt^\lambda} g^N = 0.
\end{equation}

\begin{eg} Let $e = 2, \kappa = 0, \lambda = (10, 5,3,2)$, and let $N = (1,5)$. The residue pattern is as follows.
\[
\ColorTableau[
  5/1/red!30, 6/1/red!30, 7/1/red!30, 8/1/red!30, 9/1/yellow!30, 10/1/yellow!30,
  1/2/green!30, 2/2/blue!30, 3/2/blue!30, 4/2/blue!30, 5/2/blue!30
]{{0,1,2,1,0,1,2,1,0,1},{1,0,1,2,1},{2,1,0},{1,2}}
\]
The nodes in $C^N,D^N$ and $\bm{B}^N$ are as follows.
\[
\ColorTableau[
  5/1/red!30, 6/1/red!30, 7/1/red!30, 8/1/red!30, 9/1/yellow!30, 10/1/yellow!30,
  1/2/green!30, 2/2/blue!30, 3/2/blue!30, 4/2/blue!30, 5/2/blue!30
]{{ , , , ,B_1,B_1,B_1,B_1,C,C},{D,B_2,B_2,B_2,B_2},{ , , },{ , }}
\]    
Next we give the tableau $\tabt^N$.
\[
\ColorTableau[
  5/1/red!30, 6/1/red!30, 7/1/red!30, 8/1/red!30, 9/1/yellow!30, 10/1/yellow!30,
  1/2/green!30, 2/2/blue!30, 3/2/blue!30, 4/2/blue!30, 5/2/blue!30
]{{1,2,3,4,6,7,8,9,14,15},{5,10,11,12,13},{16,17,18},{19,20}}
\]
If $\sigma_1^N$ is the (fully commutative) permutation exchanging the numbers $\{6,7,8,9\}$ with $\{10,11,12,13\}$ and $\sigma_{\tabt^N}$ is the (fully commutative) permutation such that $\tabt^\lambda \sigma_{\tabt^N} = \tabt^{N}$, then the Garnir relation is as follows.
\[
v_{\tabt^\lambda} \psi_{\tabt^N} (2 + \psi_{\sigma_1^N}) = 0.
\]
\end{eg}

\subsection{Stubborn Strings and Immobilization}
The concept of stubborn strings was introduced by Lyle and Mathas in \cite{lm14}, and they were developed further in \cite{forsberg}. They are a useful computational tool, especially for fully commutative string diagrams. We briefly recall the theory of stubborn strings as introduced in \cite{forsberg}. The setting of both papers is Specht modules of quiver Hecke algebras of type $A^{(1)}_{e-1}$, but the results hold for Specht modules of quiver Hecke algebras of type $C^{(1)}_e$ as well.

Let $v \in S^\la$ be a monomial. The set of \emph{accessible nodes} $A_v(s)$ for a string $s$ in $v$ is defined in \cite{forsberg}. Intuitively, these are the nodes at the top of the diagram that $s$ can reach by applying relations in $R_{n}^\Lambda$ that either preserve crossings or cut them, but never create new crossings. Recall that for two nodes in a Young diagram $[\mu]$, we may write $N < M$ if $\tabt^{\mu}(N) < \tabt^\mu(M)$. In the string diagram, the nodes appear in this order at the top of the diagram.

\begin{defn}
    Let $s$ be a string in a monomial $v \in S^\la$.
    \begin{enumerate}
        \item We say $s$ is \emph{stubborn} if it reaches the smallest node in $A_v(s)$.
        \item We say $s$ is \emph{co-stubborn} if it reaches the greatest node in $A_v(s)$. 
    \end{enumerate}
\end{defn}

We say that a string $s$ of residue $i$ is \emph{immobile} if any relation that resolves an $(i,i)$-crossing of $s$ yields zero. The simplest way for this to happen is if a string has a single accessible node, but it can happen in several ways. Immobile strings behave similarly to strings of a distinct residue: dots and strings may slide past their crossings without generating error terms. Specifically, the following relations hold for an immobile string $s$:
\begin{equation} \label{E:ImmobileRelations}
    \def\h{3}
    \begin{braid}
    	\draw (0,0)  -- (2, \h) node[anchor=south]{$k$};
    	\draw (2,0) node[anchor=north]{$s$} -- (0,\h) node[anchor=south]{$i$};
    	\draw (1, 0) -- (0, \h/2) -- (1, \h)node[anchor=south]{$j$};
    \end{braid}
    =
    \begin{braid}
    	\draw (0,0)  -- (2, \h) node[anchor=south]{$k$};
    	\draw (2,0) node[anchor=north]{$s$} -- (0,\h) node[anchor=south]{$i$};
    	\draw (1, 0) -- (2, \h/2) -- (1, \h)node[anchor=south]{$j$};
    \end{braid}
    \quad \text{and} \quad
    \begin{braid}
    	\draw (0,0)  -- (2, \h) coordinate[pos = 0.25] (B) node[anchor=south]{$j$};
    	\draw (2,0) node[anchor=north]{$s$} -- (0,\h) node[anchor=south]{$i$};
    	\greendot(B);
    \end{braid}
     = 
     \begin{braid}
    	\draw (0,0) -- (2, \h) coordinate[pos = 0.75] (B) node[anchor=south]{$j$};
    	\draw (2,0) node[anchor=north]{$s$} -- (0,\h) node[anchor=south]{$i$};
    	\greendot(B);
    \end{braid}.
\end{equation}

In the remainder of this paper, we will make use of these relations for immobile strings.

\section{The Quiver Hecke Action on Hooks with a Short Arm or Leg} \label{S:Action}

\subsection{$y$-action} \label{SS:YAction}
In this subsection we are concerned with describing the action of the nilpotent elements $y_k$ ($k \in \{1, \dots, n\}$) on a Specht module $S^\lambda$, where $\lambda \vdash n$ is a hook with a short arm or leg. For convenience we will assume that the leg is short, but the arguments work equally well for $\lambda$ with a short arm. It suffices to describe the action of $y_k$ on the basis of $S^\lambda$ given by $\{v_\tabt \mid \tabt \in \Std(\lambda)\}$. 

The $y$-action respects the residue sequence $\resi(\tabt)$, in the sense that $v_\tabt y_k = \sum a_{\tabt'} v_{\tabt'}$, where $a_{\tabt'} \in \bbz$ and the sum is taken over all $\tabt' \in \Std(\lambda)$ such that $\resi(\tabt') = \resi(\tabt)$. For this reason, we will begin the study of the $y$-action on $S^\lambda$ by examining the set $\Std(\lambda,\resi) = \{\tabt \in \Std(\lambda) \mid \resi(\tabt) = \resi \}$ for an arbitrary $\resi \in I^n$. In particular, $\Std(\lambda,\resi)$ is non-empty whenever $e(\resi)S^\lambda \neq 0$.

Let $\resi = (i_1, i_2, \dots, i_n)$ be a sequence of residues. A \emph{substring} of $\resi$ is a subsequence formed by a contiguous block of indices, taking the form $(i_j, i_{j+1}, \dots, i_m)$ for some $1 \leq j \leq m \leq n$. For instance, $(1,2,1,4)$ is a substring of $(1,3,1,2,1,4,5)$, while $(1,3,5)$ is not.

A weight $w \in Q^+$ is said to be a \emph{short weight} if $w = \wt(\resi)$, where
\[
\resi = (\overline{a\vphantom{b}}, \overline{a+1 \vphantom{b}}, \dots, \overline{a+k})
\]
where $a, k \in \mathbb{Z}$ and $k < 2e - 1$. For instance, $\alpha_0 + 2\alpha_1$ is a short weight, while $2\alpha_1$ and $\alpha_0 + 2\alpha_1 + \dots + 2\alpha_{e-1} + \alpha_e$ are not. Note also that a residue sequence $\resi$ like the above can alternatively be defined as a substring of the infinite residue sequence 
\begin{equation}
    \mathcal{I} \coloneq (\dots, 1,0,1, \dots, e-1, e, e-1, \dots, 1,0,1, \dots).
\end{equation}

A weight $w$ is said to be a \emph{double short weight} if $w = 2w'$, where $w'$ is a short weight. A residue sequence $\resi$ is said to be a \emph{(double) short residue sequence} if $\wt(\resi)$ is a (double) short weight. A double short residue sequence $\resi = (i_1, \dots, i_n)$ is said to be \emph{reduced} if there is no strict substring of the form $(i_1, \dots, i_k)$ with $k < n$ (that is, a strict initial segment starting at the exact same position) which is also a double short residue sequence. For instance, the sequence $\resi = (1,0,1,1,0,1)$ is a reduced double short residue sequence, while the sequence $\resi = (1,1,2,2)$ is not reduced because the subsequence $(1,1)$ is also double short.

Next we list some easy facts about short weights.

\begin{lem} \label{L:elementary_short_weights} \leavevmode
\begin{enumerate}
    \item If $w$ is a short weight, then the multiplicity of each $\alpha_j$ is either $0, 1$ or $2$. Furthermore, if $j = 0 \text{ or } e$, then the multiplicity of $\alpha_j$ is either $0$ or $1$.
    \item Let $w$ be a short weight and $j \in I$. If there exists a substring $\resi$ of $\mathcal{I}$ such that $i_1 = j$ and $\wt(\resi) = w$, then the sequence $\resi$ is unique. We write $\resi_j(w)$ for this residue sequence when it exists.
    \item The weight $\wt(\resi)$ of a substring $\resi$ of $\mathcal{I}$ cannot be a double short weight.
    \item There is a distance of at least $2e$ between the starts of any two identical substrings of $\mathcal{I}$ of length at least two.
\end{enumerate}
\end{lem}

We will use the above statements without explicit reference to them. The next lemma is crucial to the description of $\Std(\lambda,\resi)$.
\begin{lem} \label{L:doubleshortweight}
    Let $w$ be a short weight, and let $j \in I$ be any residue. Let $\resi, \resj$ be substrings of $\mathcal{I} = (\dots, 1, 0, 1, \dots, e-1, e, e-1, \dots)$ such that $\wt(\resi) + \wt(\resj) = 2w$ and $i_1 = j_1 = j$. Then $\resi_j(w)$ exists, and furthermore $\resi = \resj = \resi_j(w)$.
\end{lem}

\begin{proof}
    We will show that $\wt(\resi) = \wt(\resj) = w$. Then it will follow by the previous lemma that $\resi = \resj = \resi_j(w)$. There are several ways in which $\wt(\resi)$ might differ from $\wt(\resj)$, and we will deal with each case separately.
    \begin{enumerate}
        \item \label{subitem:case1}First, suppose towards a contradiction that the multiplicity of $\alpha_0$ in $\wt(\resi)$ is $2$ and the multiplicity of $\alpha_0$ in $\wt(\resj)$ is $0$. 
        \begin{enumerate}
            \item If $\resi = (j, j - 1, \dots, 0, 1, 2, \dots, 2, 1, 0, \dots)$, then $\resj = (j, j + 1, \dots, e - 1, e, \dots)$. We see that each residue between $1$ and $e - 1$ appears at least three times in $\wt(\resi) + \wt(\resj)$, while $0$ and $e$ both appear twice, so it follows that $|\wt(\resi) + \wt(\resj)| = 2w \geq 4e$, so that $w$ is not a short weight.
            \item If $\resi = (j, j + 1, \dots, e-1,e,e-1, \dots, 0, 1, 2, \dots, 2, 1, 0, \dots)$, then $\resj = (j, j - 1, \dots,2,1)$, and a similar argument shows that $|\wt(\resi) + \wt(\resj)| \geq 4e$, so that $w$ is not a short weight. 
        \end{enumerate}
        It follows that, whenever the multiplicity of $\alpha_0$ in $w$ is $1$, we must have that the multiplicity of $\alpha_0$ in both $\wt(\resi)$ and $\wt(\resj)$ is $1$. The analogous statement holds for the residue $e$.
        \item Next we let $k \in I \setminus\{0,e\}$ and suppose that the multiplicity of $\alpha_k$ in $\wt(\resi)$ is $2$ and the multiplicity of $\alpha_k$ in $\wt(\resj)$ is $0$. Without loss of generality, we may assume that $\resi = (j, \dots, k, k+1, \dots, e-1,e,e-1, \dots, k+1, k, \dots)$. From the above, we know that $e$ may not appear twice inside $\resi$, so that $e$ must appear once inside $\resj$. We must have $j \leq k$, since otherwise $\resi$ would include more than two $k$'s. Therefore, $\resj$ starts with $j\leq k$, contains no $k$'s, and one $e$. This is a contradiction, and therefore whenever the multiplicity of $\alpha_k$ in $w$ is $1$, we must have that the multiplicity of $\alpha_k$ in both $\wt(\resi)$ and $\wt(\resj)$ is $1$.
        \item Next we let $k \in I \setminus\{0,e\}$ and suppose that the multiplicity of $\alpha_k$ in $\wt(\resi)$ is $4$ and the multiplicity of $\alpha_k$ in $\wt(\resj)$ is $0$. It is clear that $\resi$ contains one of either $0$ or $e$ between each pair of $k$'s. Since $\resi$ contains four $k$'s, it contains at least two of either $0$ or $e$. This contradicts \ref{subitem:case1}. 
        \item Finally, we let $k \in I \setminus\{0,e\}$ and suppose that the multiplicity of $\alpha_k$ in $\wt(\resi)$ is $3$ and the multiplicity of $\alpha_k$ in $\wt(\resj)$ is $1$. Without loss of generality, we may assume that $\resi = (j, j+1, \dots )$. We know that $\resi$ has one of $0$ or $e$ between each pair of $k$'s. Due to \ref{subitem:case1}, we must have $$\resi = (j,j+1,\dots,k,\dots,e,\dots,k,\dots,0,\dots,k,\dots).$$
        Note that $\resi$ includes at least three $j$'s. Therefore, $\resj$ must contain exactly one of each $0$, $k$ and $e$. Moreover, $\resj$ starts with $j$, but it may contain at most one $j$. It is easy to see that this is absurd, and therefore whenever the multiplicity of $\alpha_k$ in $w$ is $2$, we must have that the multiplicity of $\alpha_k$ in both $\wt(\resi)$ and $\wt(\resj)$ is $2$, completing the proof. \qedhere
    \end{enumerate} 
    \end{proof}
Let $\resi \in I^n$ be a residue sequence such that $e(\resi) S^\lambda \neq 0$. For any tableau $\tabt \in \Std(\lambda,\resi)$, it is clear that $\tabt^{-1}(1) = (1, 1)$. Suppose we know $\tabt_{\downarrow m}$ for some $m \geq 1$ and we wish to determine $\tabt^{-1}(m+1)$. The unfilled nodes in $\Arm(\lambda)$ and $\Leg(\lambda)$ each form a contiguous strip starting 
from the current boundary of $\tabt_{\downarrow m}$. We denote the sequences of 
residues of these unfilled nodes, ordered by their distance from the corner 
node $(1,1)$, as:
\[
\resi^{\text{arm}}(\tabt_{\downarrow m}) = (i^{\text{arm}}_1(\tabt_{\downarrow m}), i^{\text{arm}}_2(\tabt_{\downarrow m}), \dots) \quad \text{and} 
\quad \resi^{\text{leg}}(\tabt_{\downarrow m}) = (i^{\text{leg}}_1(\tabt_{\downarrow m}), i^{\text{leg}}_2(\tabt_{\downarrow m}), \dots).
\]

We omit the argument $\tabt_{\downarrow m}$ when it is clear from context. Next we are interested in placing the number $m+1 \leq n$ in $\tabt$. The residue $i_{m+1}$ must equal at least one of $i^{\text{arm}}_1$ or $i^{\text{leg}}_1$, and if $i_{m+1} \neq i^{\text{arm}}_1$, then $\tabt^{-1}(m+1) \in \Leg(\lambda)$. Similarly, if $i_{m+1} \neq i^{\text{leg}}_1$, then $\tabt^{-1}(m+1) \in \Arm(\lambda)$. Suppose instead that $i_{m+1} = i^{\text{arm}}_1 = i^{\text{leg}}_1$. We split our investigation into two possibilities:
\begin{enumerate} 
    \item Suppose that there exists a substring $\resi '$ of $\resi$ beginning at $i'_1 = i_{m+1}$ such that $\wt(\resi ') = 2w$ is a double short weight, and let $l' \coloneq \ell(\resi')$. Then, according to \cref{L:doubleshortweight}, the numbers $m+1, \dots, m+l'$ can be divided into two halves $A$ and $B$ of equal size so that $\tabt^{-1}(a) \in \Arm(\lambda)$ if $a \in A$, and $\tabt^{-1}(b) \in \Leg(\lambda)$ if $b \in B$. Furthermore, $\sum_{a \in A} \alpha_{\res(\tabt^{-1}(a))} = \sum_{b \in B} \alpha_{\res(\tabt^{-1}(b))} = w$. If the numbers in $A$ and $B$ are exchanged in $\tabt \in \Std(\lambda)$, the resulting tableau is also standard.
    \item Suppose, to the contrary, that no substring $\resi'$ of $\resi$ beginning at $i'_1 = i_{m+1}$ exists such that $\wt(\resi ') = 2w$ is a double short weight, and in particular $i_{m+2} \neq i_{m+1}$. We claim that in this case, the node $\tabt^{-1}(m+1)$ is fully determined by the residue sequence $\resi$.
    \begin{enumerate}
        \item Suppose that $i_{m+2} \neq i^\text{arm}_2$. Then it is clear that $\tabt^{-1}(m+1) \in \Leg(\lambda)$. Similarly, if $i_{m+2} \neq i^\text{leg}_2$, then $\tabt^{-1}(m+1) \in \Arm(\lambda)$.
        \item Suppose, to the contrary, that $i_{m+2} = i^\text{arm}_2 = i^\text{leg}_2$.
        \begin{enumerate}
            \item Suppose that $\Arm(\lambda)$ and $\Leg(\lambda)$ share the same number of unfilled nodes. Then the sequence $\resi' = (i_{m+1}, i_{m+2}, \dots, i_n)$ has $\wt(\resi') = 2w$, where $w$ is the short weight corresponding to the unfilled nodes in both the arm and leg. (Recall that in this section $\Leg(\lambda)$ is assumed to be short.)
            \item Suppose that $\Arm(\lambda)$ has more unfilled nodes than $\Leg(\lambda)$ in $\tabt_{\downarrow m}$ and $i_{m+2} = i^{\text{arm}}_2 = i^{\text{leg}}_2$. Suppose, towards a contradiction, that $\tabt^{-1}(m+1) \in \Leg(\lambda)$. Let $m' > m$ be an integer such that $\tabt_{\downarrow m'} \setminus \tabt_{\downarrow m}$ has the same number of filled nodes in the arm as in the leg. Note that $m' > m$ must exist, since $\Arm(\lambda)$ has more unfilled nodes than $\Leg(\lambda)$. Then the sequence $\resi' = (i_{m+1}, \dots, i_{m'})$ has a double short weight $\wt(\resi') = 2w$ (recall that $\Leg(\lambda)$ is short), which contradicts our assumption at the start of (2). Therefore, $\tabt^{-1}(m+1) \in \Arm(\lambda)$. A similar conclusion holds if $\Leg(\lambda)$ has more unfilled nodes than $\Arm(\lambda)$.
        \end{enumerate}
    \end{enumerate}
\end{enumerate}
We conclude that the construction of any $\tabt \in \Std(\lambda,\resi)$ by placing the natural numbers into the Young diagram proceeds deterministically except at specific points where a binary choice between the arm and the leg arises.

\begin{defn} \label{D:DS}
    The set of \emph{designated sequences} $DS(\resi) = \{\bm{s}^{(1)}, \dots, \bm{s}^{(d)} \}$ is the set containing all contiguous integer intervals $\bm{s} = [u, v] \subseteq \{1, \dots, n\}$ satisfying the following two conditions:
\begin{enumerate}
    \item The subword of residues corresponding to the interval, $\resi_{\bm{s}} = (i_u, i_{u+1}, \dots, i_v)$, is a reduced double short residue sequence.
    \item For any $\tabt \in \Std(\lambda,\resi)$, consider the subtableau $\tabt_{\downarrow u-1}$. The next available addable nodes in both $\Arm(\lambda)$ and $\Leg(\lambda)$ exist, and their residues are both equal to $i_u$ (i.e., $i_1^{\text{arm}} = i_1^{\text{leg}} = i_u$).
\end{enumerate}
\end{defn} 

For each $r \in \{1, \dots, d\}$, we denote the indices of the $r$-th interval by $\bm{s}^{(r)} = (s^{(r)}_1, s^{(r)}_2, \dots,\allowbreak s^{(r)}_{|\bm{s}^{(r)}|})$ and its corresponding residue sequence by $\resi^{(r)} \coloneq \resi_{\bm{s}^{(r)}}$. Since a reduced double short residue sequence is uniquely determined by its starting position, the intervals in $DS(\resi)$ have distinct starting points, and we may unambiguously order them according to their starting indices so that ${s}^{(1)}_1 < {s}^{(2)}_1 < \dots < {s}^{(d)}_1$.

It is important to verify that the set $DS(\resi)$ is well-defined, that is, that Condition (2) does not depend on the choice of $\tabt \in \Std(\lambda,\resi)$. Since designated sequences can be nested (that is, an interval $\bm{s}^{(2)}$ may be strictly contained within another interval $\bm{s}^{(1)}$), the precise shape of the subtableau $\tabt_{\downarrow (u_2-1)}$ for $u_2 = {s}^{(2)}_1$ may depend on whether the initial elements of the larger sequence $\bm{s}^{(1)}$ were assigned to $\Arm(\lambda)$ or $\Leg(\lambda)$. In order to deal with this issue, first we describe the subset structure of $DS(\resi)$.

\begin{lem} \label{L:YActionSubsetStructure}
    Let $r_1, r_2 \in \{1, \dots, d\}$ with $r_1 < r_2$. Suppose that $\bm{s}^{(r_1)} \cap \bm{s}^{(r_2)} \neq \varnothing$. Then $\bm{s}^{(r_2)} \subset \bm{s}^{(r_1)}$, $|\bm{s}^{(r_2)}| = 2$ and the repeated residue in $\bm{s}^{(r_2)}$ is not $0$ or $e$.
\end{lem}
\begin{proof} 
    Let $\bm{s}^{(r_1)} =[u_1, v_1]$ and $\bm{s}^{(r_2)} = [u_2, v_2]$. Since $r_1 < r_2$ and the designated sequences are ordered by their starting index, we have $u_1 < u_2$. The assumption $\bm{s}^{(r_1)} \cap \bm{s}^{(r_2)} \neq \varnothing$ implies $u_2 \leq v_1$. We first show that $\bm{s}^{(r_2)} \subset \bm{s}^{(r_1)}$, which is equivalent to showing $v_2 \leq v_1$. Suppose, towards a contradiction, that $\bm{s}^{(r_2)} \not\subset \bm{s}^{(r_1)}$. Let us first check that, under our supposition, both $\bm{s}^{(r_1)}$ and $\bm{s}^{(r_2)}$ have length larger than 2. Indeed, suppose that $|\bm{s}^{(1)}| = 2$. Then, since we have both $\bm{s}^{(r_2)} \not\subset \bm{s}^{(r_1)}$ and $\bm{s}^{(r_1)} \cap \bm{s}^{(r_2)} \neq \varnothing$, it follows that $u_2 = u_1 + 1 = v_1$, which indicates that $\tabt_{\downarrow u_1}$ has two addable nodes with residue $i_{v_1} = i_{u_1}$, after placing $u_1$ in a node with that residue. This is impossible, so it follows that $|\bm{s}^{(r_1)}| > 2$. A similar argument shows that $|\bm{s}^{(r_2)}| > 2$.

    Since $|\bm{s}^{(r_1)}| > 2$, the addable sequences $\resi^{\text{arm}}(\tabt_{\downarrow (u_1-1)})$ and $\resi^{\text{leg}}(\tabt_{\downarrow (u_1-1)})$ coincide until one of them terminates, since their first two residues coincide. The same can be said about $\resi^{\text{arm}}(\tabt_{\downarrow (u_2-1)})$ and $\resi^{\text{leg}}(\tabt_{\downarrow (u_2-1)})$. The amounts of nodes in the arm and leg covered by numbers in $\bm{s^{(r_1)}}$ are equal, and a similar result holds for $\bm{s^{(r_2)}}$. Consequently, the amounts of nodes covered by the numbers $[u_1, u_2 - 1]$ in $\Arm(\lambda)$ and $\Leg(\lambda)$ are equal. Therefore, the sequence $\bm{s}^* = [u_1, u_2 - 1]$ is itself a double short sequence. This contradicts the assumption that the designated sequence $\bm{s}^{(r_1)}$ must be reduced, according to the definition of $DS(\resi)$. This completes the proof that $\bm{s}^{(r_2)} \subset \bm{s}^{r_1}$.
    
    The sequence $\bm{s}^{(r_2)}$ must fill subsegments of equal length within the segments filled by $\bm{s}^{(r_1)}$. Let $d_{\text{arm}}$ and $d_{\text{leg}}$ be the respective number of nodes placed in the arm and in the leg by $\bm{s}^{(r_1)}$ prior to the start of $\bm{s}^{(r_2)}$. We must have $d_{\text{arm}} \neq d_{\text{leg}}$, as otherwise the subsequence $\bm{s}^*$ would be a double short sequence, contradicting the reducedness of $\bm{s}^{(r_1)}$.

    Define $\resi_{\text{arm}}^{(r)}$ to be the residue sequence in the arm of $\tabt$ filled by the numbers in $\bm{s}^{r}$, and define $\resi_{\text{leg}}^{(r)}$ analogously. The sequence $\resi_{\text{arm}}^{(r_2)}$ appears at two distinct starting positions ($d_{\text{arm}}+1$ and $d_{\text{leg}}+1$) inside $\resi_{\text{arm}}^{(r_1)}$. But the starts of two identical substrings of $\mathcal{I}$ of length at least two are separated by at least $2e$ residues, and of course the residues $0$ and $e$ are also placed at intervals of $2e$. Since $\resi_{\text{arm}}^{(r_1)}$ is a short residue sequence, its total length is strictly less than $2e$. It follows that $\resi^{(r_2)}_{\text{arm}}$ cannot have length 2 or greater, and thus must contain a single residue, which cannot be $0$ or $e$. This completes the proof.
\end{proof}

Let us return to our goal of checking that Condition 2 of the definition of $DS(\resi)$ is well-defined. Suppose we have $\bm{s}^{(2)} \subset \bm{s}^{(1)}$. Since $|\bm{s}^{(1)}| > 2$, then the addable residue sequences in the arm and leg will be symmetric at $\tabt_{\downarrow {(u_1-1)}}$, where $u_1 = {s}^{(1)}_1$. After placing $u_1$, the numbers $u_1+1, \dots, u_2 - 1$ are each forced into either $\Arm(\lambda)$ or $\Leg(\lambda)$, and the next available nodes in both $\Arm(\lambda)$ and $\Leg(\lambda)$ at step $u_2$ will share the same residue $i_{u_2}$. We then make an independent choice ($A$ or $L$) for the elements of the short sequence $\bm{s}^{(2)}$. Since $\bm{s}^{(2)}$ consumes exactly one node in the arm and one in the leg, the set of filled nodes does not depend on this choice. By the time the numbers in the larger sequence $\bm{s}^{(1)}$ have all been placed in $\tabt$, the resulting Young diagram shape is independent of these intermediate choices, as a sequence of nodes with weight $w(\resi)$ has been filled in both the arm and leg of $\lambda$. This guarantees that the process is deterministic outside of these specific binary decisions, and that an earlier $(A,L)$-choice does not affect the structure of $DS(\resi)$.

We partition the index set $\{1, \dots, d\}$ into two disjoint subsets $W_2$ and $W_4$, where $r \in W_4$ if $\resi^{(r)} \in \{(0, 0), (e, e)\}$, and $r \in W_2$ otherwise.

The next proposition synthesizes our findings and parameterizes the set $\Std(\lambda,\resi)$ via these independent binary choices.

\begin{prop} \label{P:ResidueHookTableaux}
    Suppose that $\Std(\lambda,\resi) \neq \varnothing$. Let $\{A, L\}^{d}$ be the set of words of length $d$ in the alphabet $\{A, L\}$. Then there is a bijection $\phi_{\resi}^\lambda: \{A, L\}^{d} \rightarrow \Std(\lambda,\resi)$ uniquely determined by the property that for any $r \in \{1, \dots, d\}$, the node $\tabt^{-1}(s^{(r)}_1)$ belongs to $\Arm(\lambda)$ if and only if the $r$-th letter of the word is $A$. Furthermore, if $\tabt = \phi_{\resi}^\lambda(l_1, \dots, l_{d})$, its degree is given by
    \[
        \deg(\tabt) = \deg(\phi_{\resi}^\lambda(A, \dots, A)) - 2\#\{r \in W_2 \mid l_r = L \} - 4\#\{r \in W_4 \mid l_r = L \}.
    \]
    We say that $(l_1, \dots, l_{d})$ is the \emph{binary word} of $\tabt$.
\end{prop}

\begin{proof}
    The existence of the bijection $\phi_{\resi}^\lambda$ follows from the preceding discussion. We must prove the statement about the degrees. By an inductive argument, it is enough to determine $\deg(\tabt') - \deg(\tabt)$ in the case of two tableaux $\tabt, \tabt' \in \Std(\lambda,\resi)$ where 
    \begin{align*}
        \tabt &= \phi_{\resi}^\lambda(l_1, \dots, l_{r-1}, A, l_{r+1}, \dots, l_{d}),\\
        \tabt' &= \phi_{\resi}^\lambda(l_1, \dots, l_{r-1}, L, l_{r+1}, \dots, l_{d}).
    \end{align*}
    Suppose that $\resi^{(r)}$ has length $2q$. We restrict our attention to the row- and column-strips in $\lambda$ which house the numbers in $\bm{s}^{(r)} = \{s^{(r)}_1, \dots, s^{(r)}_{2q}\}$. If we focus our attention on the relevant strings only, then the computation in \cref{L:YActionYkVL} shows that $\deg(v_{\tabt_L}) - \deg(v_{\tabt_A}) = 2$ whenever $r \in W_2$, and the case $r \in W_4$ is trivial.
\end{proof}

\begin{cor} \label{C:minimal_tableau}
    There exists a unique tableau $\tabt_{\mathrm{max}}(\resi) \coloneq \phi_{\resi}^{\lambda}(A, \dots, A) \in \Std(\lambda,\resi)$ of maximal degree inside $\Std(\lambda,\resi)$.
\end{cor}

\begin{eg} \label{eg:binaryword}
Suppose $e = 2$ and $\Lambda = \Lambda_0$. Let $\lambda = (4, 1^3) \vdash 7$ and fix the residue sequence $\resi = (0, 1, 2, 1, 1, 2, 1)$. 
    
    By mapping these residues onto the nodes of the hook $\lambda$, we find two points where an identical residue can be placed in either the arm or the leg, leading to two designated sequences in $DS(\resi)$:
    \begin{itemize}
        \item $\bm{s}^{(1)} = (2, 3, 4, 5, 6, 7)$ with $\resi^{(1)} = (1, 2, 1, 1, 2, 1)$.
        \item $\bm{s}^{(2)} = (4, 5)$ with $\resi^{(2)} = (1, 1)$.
    \end{itemize}
    Note that the double short residue sequence $\bm{s}^* = (3,4,5,6)$ is not a designated sequence in $DS(\resi)$ as it does not satisfy condition 2: at the time of placing the number $3$, the tableau $\tabt_{\downarrow 2}$ can only be one of
    \[\ColorTableau[]{{1,2}} \qquad \text{ or } \qquad \ColorTableau[]{{1},{2}} \,,\]
    and in either case there is only one available node with the correct residue.
    
    Thus, $d=2$. Since neither sequence is $(0,0)$ or $(e,e)$, both belong to $W_2$. Note that $\bm{s}^{(2)} \subset \bm{s}^{(1)}$. 
    The bijection $\phi_{\resi}^\lambda$ gives four standard tableaux in $\Std(\lambda,\resi)$ corresponding to the words in $\{A, L\}^2$:
    \[
    \phi_{\resi}^\lambda(L, L) = 
    \ColorTableau[]{{1,5,6,7},{2},{3},{4}}, \qquad \qquad \quad
    \phi_{\resi}^\lambda(A, L) = \ColorTableau[]{{1,2,3,5},{4},{6},{7}},
    \]
    \[
    \phi_{\resi}^\lambda(L, A) = 
    \ColorTableau[]{{1,4,6,7},{2},{3},{5}}, \, \qquad
    \tabt_{\mathrm{max}}(\resi) = \phi_{\resi}^\lambda(A, A) = \ColorTableau[]{{1,2,3,4},{5},{6},{7}}.
    \]
\end{eg}

Now that we understand the set $\Std(\lambda,\resi)$, we are ready to explore the effect of the $y$-action on a basis vector $v_\tabt \in S^\lambda$ with $\tabt \in \Std(\lambda,\resi)$. That is, we are interested in finding the coefficients $a_{\tabt'} \in \mathbb{Z}$ in $v_\tabt y_k = \sum_{\tabt' \in \Std(\lambda,\resi)}a_{\tabt'} v_{\tabt'}$.

Let us first compare the diagram $v_\tabt = v_{\phi_{\resi}^\lambda(l_1, \dots, l_{r-1}, A, l_{r+1}, \dots, l_d)}$ to the diagram $v_{\tabt'} \coloneq v_{\phi_{\resi}^\lambda(l_1', \dots, l_{r-1}', L, l_{r+1}', \dots, l_d')}$. Recall that for any fully commutative standard monomial, every string in the diagram is either stubborn or co-stubborn. In our specific configuration, a string routed to the arm is stubborn, while a string routed to the leg is co-stubborn. In the diagram for $v_{\tabt}$, the choice $A$ dictates that the string $s^{(r)}_1$ is routed to the arm. Because $s^{(r)}_1$ is stubborn, it cannot reach larger nodes, meaning it is geometrically impossible for it to reach any node located in the leg in any non-zero resolution of crossings. It follows directly that an $A$ choice cannot transform into an $L$ choice through the $y$-action. Consequently, if $l_r = A$ for $\tabt$, then $l_r = A$ for any $\tabt'$ such that $a_{\tabt'} \neq 0$ in the decomposition of $v_{\tabt} y_k$. Together with the degree considerations from \cref{P:ResidueHookTableaux}, we obtain the following lemma.

\begin{lem} \label{L:YActionVanishing}
    Let $k \in \{1, \dots, n\}$, $\tabt = \phi_{\resi}^\lambda(l_1, \dots, l_d) \in \Std(\lambda,\resi)$. 
    \begin{enumerate}
        \item If $\deg(y_k) = 2$, then $v_\tabt y_k = \sum_{\tabt' \in \Std(\lambda,\resi)}a_{\tabt'} v_{\tabt'}$, with $a_{\tabt'} \neq 0$ only if $\tabt' = \phi_{\resi}^\lambda(l_1', \dots, l_d')$, with $l_s = l_s'$ for $s \in \{1, \dots, r-1, r+1, \dots, d\}$, while $l_r = L$, $l_r' = A$, $k \in \bm{s}^{(r)}$ and $r \in W_2$.
        \item If $\deg(y_k) = 4$, then $v_\tabt y_k = a_{\tabt'} v_{\tabt'}$, with $\tabt' = \phi_{\resi}^\lambda(l_1', \dots, l_d')$, with $l_s = l_s'$ for all $s \in \{1, \dots, r-1, r+1, \dots, d\}$, while $l_r = L$, $l_r' = A$, $k \in \bm{s}^{(r)}$ and $r \in W_4$.
    \end{enumerate}
\end{lem}
\begin{proof}
    Indeed, we know that $v_\tabt y_k = \sum_{\tabt'} a_{\tabt'} v_{\tabt'}$, where $a_{\tabt'} \neq 0$ only if 
    \begin{enumerate}
        \item $\tabt' \in \Std(\lambda,\resi)$,
        \item $\deg(\tabt') - \deg(\tabt) = \deg(y_k e(\resi(\tabt)))$ and
        \item The node $(\tabt')^{-1}(k)$ is accessible to string $k$ in the diagram $v_\tabt$. That is, $(\tabt')^{-1}(k) \in A_{v_\tabt}(k)$.
    \end{enumerate}  
    Finally, according to \cite[Lemma 3.6]{forsberg}, after cutting a crossing of string $k$ in $v_\tabt$, the original node $\tabt^{-1}(k)$ that string $k$ reached in $v_\tabt$ is no longer accessible to $k$, so that one of the binary choices involving string $k$ must change.
    The case where $\deg(y_ke(\resi(\tabt))) = 2$ now immediately follows.
    
    Moreover, a residue $0$ or $e$ may be at most inside one reduced residue sequence of double short weight, as shown in \cref{L:YActionSubsetStructure}. It follows that there is at most one summand in the basis decomposition of $v_\tabt y_k$ whenever $\deg(y_k e(\resi(\tabt))) = 4$. \end{proof}
Incidentally, the usual tableau dominance order relates nicely to the function $\phi_{\resi}^\lambda$. To be precise, we have 
$$\tabt = \phi_{\resi}^\lambda(l_1, \dots, l_d) \doms \phi_{\resi}^\lambda(l_1', \dots, l_d') = \tabt'$$
if and only if $(l_r, l_r') \neq (L, A)$ for all $r \in \{1, \dots, d\}$. In particular, $\tabt \doms \tabt'$ implies $\deg(\tabt) > \deg(\tabt')$.

The above lemma provides us with an accurate idea of the terms that vanish under the $y$-action. However, to find the value of $a_{\tabt'}$ in the remaining cases we will compute it via the diagram calculus of quiver Hecke algebras.

Let $s \in \{1, \dots, n\}$ be such that $s \notin \bm{s}^{(r)}$ for any $r$ such that $k \in \bm{s}^{(r)}$. Then $a_{\tabt'} = 0$ whenever $\tabt^{-1}(s) \neq (\tabt')^{-1}(s)$, in concordance with \cref{L:YActionVanishing}. Therefore all such strings are immobile with respect to the action of $y_k$. Since they are immobile strings in a diagram $v_\tabt$ with a fully commutative underlying permutation $\sigma_\tabt$, we may fairly stop drawing them when computing $v_\tabt y_k$. In our computations we will only draw the strings labelled by an integer in $\bm{s}^{(r)}$. Recall that the residues sequence attached to the numbers in $\bm{s}^{(r)} = \{s^{(r)}_1, \dots, s^{(r)}_{2m}\}$ is $i_1^{(r)}, \dots, i_{2m}^{(r)}$. We will use the simple notation $i_1, \dots, i_m$ for the (equal) short residue sequences filled in the arm and leg of $\lambda$ by $\bm{s}^{(r)}$. These appear in a row-strip in the arm and in a column-strip in the leg of $\lambda$.

Next we give a description of the diagram for the strings $\bm{s}^{(r)}$ in $v_\tabt$, assuming that $l_r = L$ for $\tabt$. This condition ensures that $s_1^{(r)} \in \Leg(\lambda)$ in $\tabt$. Moreover, we assume that $\resi^{(r)}$ is reduced, which requires that $s_2^{(r)} \in \Leg(\lambda)$. By the same reasoning, $s^{(r)}_3$ and $s^{(r)}_4$ may not both be in $\Arm(\lambda)$. In general, for any $k < \frac{m}{2}$, at least $k+1$ out of the first $2k$ strings must reach $\Leg(\lambda)$. 

Next we give a description of the diagram for $v_\tabt$. For graphical clarity, we have used the notation $i_m^{-j} \coloneq i_{m-j}$. The column-strip satisfies the same relations as the row-strip, the Garnir relations in the column-strip playing the role of the Specht relations in the row-strip. For this reason, we have chosen to draw the column-strip as a row-strip. 

\begin{lem} \label{L:YActionVTDescription}
    Let $\tabt = \phi_{\resi}^\lambda(l_1, \dots, l_{r-1}, L, l_{r+1}, \dots, l_d)$ be the tableau described above, and draw only the strings in $\bm{s}^{(r)}$. Then
    \begin{equation}
    v_\tabt = 
    \begin{braid}
	\def\h{5};
	\def\sep{2}
	\def\a{7+\sep};
	\def\dashpart{0.8}
	\def\sp{1.5}
	\def\hrect{1.15}
	\def\shiftdown{-3}
	
	\draw(\sp + 0, \shiftdown) -- (\sp + 0, 0) -- (\a, \h) node[anchor=south]{$i_1$} ;
	\draw(\sp + 1, \shiftdown) --(\sp +1, 0) -- (\a+1, \h) node[anchor=south]{$i_2$};
	\draw(\sp + 2, \shiftdown) --(\sp +2, 0) -- (0, \h) node[anchor=south]{$i_1$};
	\draw(\sp + 3, \shiftdown) --(\sp +3, 0) -- (\a+2, \h) node[anchor=south]{$i_3$};
	\draw(\sp + 4, \shiftdown) --(\sp +4, 0) -- (1, \h) node[anchor=south]{$i_2$};
	\draw(\sp + 5, \shiftdown) --(\sp +5, 0) -- (\a+3, \h) node[anchor=south]{$i_4$};
	
	\draw (\sp + 5.5 + \sep/2, -3) node {$\cdots$};
	\draw (1.5 + \sep/2, \h) node[anchor=south]{$\cdots$};
	\draw (\a + 3.5 + \sep/2, \h) node[anchor=south]{$\cdots$};
	
	\draw (\sp + 6+\sep, \shiftdown) --(\sp + 6 + \sep, 0) -- (2  + \sep, \h) node[anchor=south]{$i_m^{-3}$};
	\draw (\sp + 7+\sep, \shiftdown) -- (\sp + 7 + \sep, 0) -- (\a + 4  + \sep, \h) node[anchor=south]{$i_m^{-1}$};
	\draw  (\sp + 8+\sep, \shiftdown) --(\sp +8 + \sep, 0) -- (3  + \sep, \h) node[anchor=south]{$i_m^{-2}$};
	\draw  (\sp + 9+\sep, \shiftdown) --(\sp + 9 + \sep, 0) -- (\a + 5  + \sep, \h) node[anchor=south]{$i_m$};
	\draw  (\sp + 10+\sep, \shiftdown) --(\sp + 10 + \sep, 0) -- (4  + \sep, \h) node[anchor=south]{$i_m^{-1}$};
	\draw  (\sp + 11+\sep, \shiftdown) --(\sp + 11 + \sep, 0) -- (5  + \sep, \h) node[anchor=south]{$i_m$};
	
	\draw[black, sharp corners, fill=white] (1.5 + \sp, \shiftdown + 1) rectangle (9.5 + \sp + \sep , \shiftdown + 1 + \hrect);
	\draw[brown] (-0.5, \h) rectangle (\a - 1.5 , \h + \hrect);
	\draw (\sp + 5.5 + \sep/2, \shiftdown +1 + 0.5*\hrect) node[black, scale=1.3]{$\psi_\sigma$};
	\draw[brown] (\a - 0.5, \h) rectangle (2*\a + 0.5 , \h + \hrect);

\end{braid},
\end{equation} \label{E:GenericVtL}
where the permutation $\sigma$ is restricted in the following two ways:
\begin{enumerate}[(A)]
    \item $\sigma$ may not undo any crossings in the diagram above it. \label{C:perm1}
    \item $v_\tabt$ is the basis element for a standard tableau $\tabt$. \label{C:perm2}
\end{enumerate}
We denote the upper part of Diagram \ref{E:GenericVtL} by $v_L$, so that $v_\tabt = v_L \psi_\sigma$.
\end{lem}

 \begin{proof}
     We work by induction on $m$. The result is clear for $m = 2$. Write $N_1, \dots, N_m$ for the nodes in $\Arm(\lambda)$ reached by strings in $\bm{s}^{(r)}$, and write $N_1', \dots, N_m'$ for the corresponding nodes in $\Leg(\lambda)$. Next, ignore the strings reaching nodes $N_m, N_m'$. In this way we obtain a diagram with $2m-2$ strings having a reduced sequence with double short weight. By the inductive hypothesis, this diagram looks as follows, for some permutation $\sigma'$ satisfying the appropriate conditions.
     
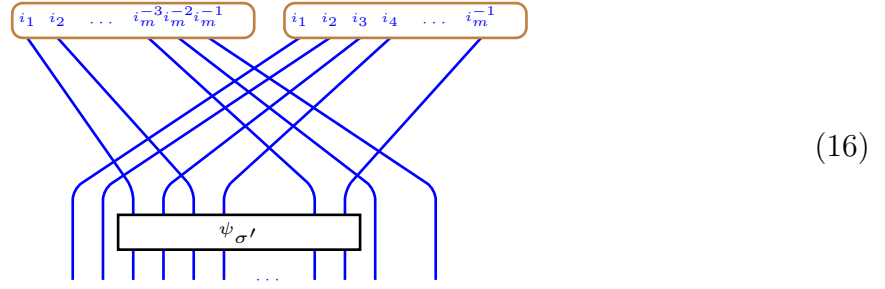
\begin{equation}
          \begin{braid}
	\def\h{5};
	\def\sep{2}
	\def\a{7+\sep};
	\def\dashpart{0.8}
	\def\sp{1.5}
	\def\hrect{1.15}
	\def\shiftdown{-3}
	
	\draw(\sp + 0, \shiftdown) -- (\sp + 0, 0) -- (\a, \h) node[anchor=south]{$i_1$} ;
	\draw(\sp + 1, \shiftdown) --(\sp +1, 0) -- (\a+1, \h) node[anchor=south]{$i_2$};
	\draw(\sp + 2, \shiftdown) --(\sp +2, 0) -- (0, \h) node[anchor=south]{$i_1$};
	\draw(\sp + 3, \shiftdown) --(\sp +3, 0) -- (\a+2, \h) node[anchor=south]{$i_3$};
	\draw(\sp + 4, \shiftdown) --(\sp +4, 0) -- (1, \h) node[anchor=south]{$i_2$};
	\draw(\sp + 5, \shiftdown) --(\sp +5, 0) -- (\a+3, \h) node[anchor=south]{$i_4$};
	
	\draw (\sp + 5.5 + \sep/2, -3) node {$\cdots$};
	\draw (1.5 + \sep/2, \h) node[anchor=south]{$\cdots$};
	\draw (\a + 3.5 + \sep/2, \h) node[anchor=south]{$\cdots$};
	
	\draw (\sp + 6+\sep, \shiftdown) --(\sp + 6 + \sep, 0) -- (2  + \sep, \h) node[anchor=south]{$i_m^{-3}$};
	\draw (\sp + 7+\sep, \shiftdown) -- (\sp + 7 + \sep, 0) -- (\a + 4  + \sep, \h) node[anchor=south]{$i_m^{-1}$};
	\draw  (\sp + 8+\sep, \shiftdown) --(\sp +8 + \sep, 0) -- (3  + \sep, \h) node[anchor=south]{$i_m^{-2}$};
	\draw  (\sp + 10+\sep, \shiftdown) --(\sp + 10 + \sep, 0) -- (4  + \sep, \h) node[anchor=south]{$i_m^{-1}$};
	
	\draw[black, sharp corners, fill=white] (1.5 + \sp, \shiftdown + 1) rectangle (7.5 + \sp + \sep , \shiftdown + 1 + \hrect);
	\draw[brown] (-0.5, \h) rectangle (\a - 1.5 , \h + \hrect);
	\draw (\sp + 4.5 + \sep/2, \shiftdown +1 + 0.5*\hrect) node[black, scale=1.3]{$\psi_{\sigma'}$};
	\draw[brown] (\a - 0.5, \h) rectangle (2*\a + 0.5 , \h + \hrect);

\end{braid}
     \end{equation}
     Next we consider the ways in which the strings reaching nodes $N_m, N_m'$ may appear in the diagram. Standardness requires that node $N_m$ is filled with a number higher than $\tabt(N_{m-1})$. Moreover, the nodes $\{N_m, N_m'\}$ may not be filled with numbers $\{s^{(r)}_{2m-1}, s^{(r)}_{2m}\}$, since $\bm{s}^{(r)}$ must be reduced. Together, these conditions imply that $\tabt(N_m) = s^{(r)}_{2m}$, while the string labeled by $\tabt(N_m')$ crosses at least the strings $s^{(r)}_{2m-1}, s^{(r)}_{2m}$. It may then go on to cross further strings. Therefore, the diagram for $v_\tabt$ looks as follows.
     
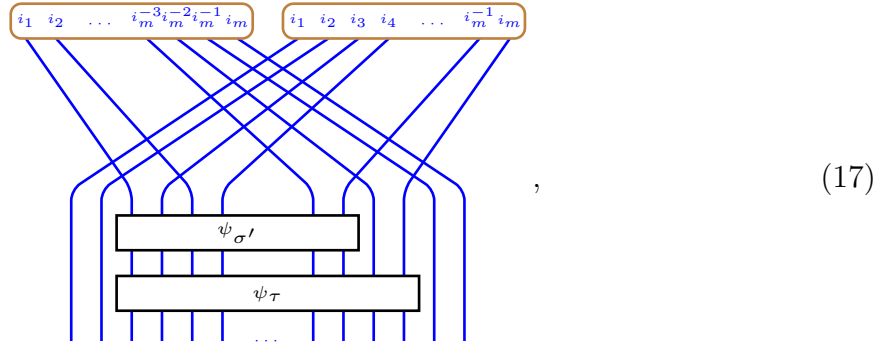
\begin{equation}
          \begin{braid}
	\def\h{5};
	\def\sep{2}
	\def\a{7+\sep};
	\def\dashpart{0.8}
	\def\sp{1.5}
	\def\hrect{1.15}
	\def\shiftdown{-5}
	
	\draw(\sp + 0, \shiftdown) -- (\sp + 0, 0) -- (\a, \h) node[anchor=south]{$i_1$} ;
	\draw(\sp + 1, \shiftdown) --(\sp +1, 0) -- (\a+1, \h) node[anchor=south]{$i_2$};
	\draw(\sp + 2, \shiftdown) --(\sp +2, 0) -- (0, \h) node[anchor=south]{$i_1$};
	\draw(\sp + 3, \shiftdown) --(\sp +3, 0) -- (\a+2, \h) node[anchor=south]{$i_3$};
	\draw(\sp + 4, \shiftdown) --(\sp +4, 0) -- (1, \h) node[anchor=south]{$i_2$};
	\draw(\sp + 5, \shiftdown) --(\sp +5, 0) -- (\a+3, \h) node[anchor=south]{$i_4$};
	
	\draw (\sp + 5.5 + \sep/2, \shiftdown) node {$\cdots$};
	\draw (1.5 + \sep/2, \h) node[anchor=south]{$\cdots$};
	\draw (\a + 3.5 + \sep/2, \h) node[anchor=south]{$\cdots$};
	
	\draw (\sp + 6+\sep, \shiftdown) --(\sp + 6 + \sep, 0) -- (2  + \sep, \h) node[anchor=south]{$i_m^{-3}$};
	\draw (\sp + 7+\sep, \shiftdown) -- (\sp + 7 + \sep, 0) -- (\a + 4  + \sep, \h) node[anchor=south]{$i_m^{-1}$};
	\draw  (\sp + 8+\sep, \shiftdown) --(\sp +8 + \sep, 0) -- (3  + \sep, \h) node[anchor=south]{$i_m^{-2}$};
	\draw  (\sp + 9+\sep, \shiftdown) --(\sp + 9 + \sep, 0) -- (\a + 5  + \sep, \h) node[anchor=south]{$i_m$};
	\draw  (\sp + 10+\sep, \shiftdown) --(\sp + 10 + \sep, 0) -- (4  + \sep, \h) node[anchor=south]{$i_m^{-1}$};
	\draw  (\sp + 11+\sep, \shiftdown) --(\sp + 11 + \sep, 0) -- (5  + \sep, \h) node[anchor=south]{$i_m$};
	\draw[black, sharp corners, fill=white] (1.5 + \sp, \shiftdown + 3) rectangle (7.5 + \sp + \sep , \shiftdown + 3 + \hrect);
	\draw (\sp + 4.5 + \sep/2, \shiftdown +3 + 0.5*\hrect) node[black, scale=1.3]{$\psi_{\sigma'}$};
	
	\draw[black, sharp corners, fill=white] (1.5 + \sp, \shiftdown + 1) rectangle (9.5 + \sp + \sep , \shiftdown + 1 + \hrect);
	\draw (\sp + 5.5 + \sep/2, \shiftdown +1 + 0.5*\hrect) node[black, scale=1.3]{$\psi_{\tau}$};
	\draw[brown] (-0.5, \h) rectangle (\a - 1.5 , \h + \hrect);
	
	\draw[brown] (\a - 0.5, \h) rectangle (2*\a + 0.5 , \h + \hrect);

\end{braid},
     \end{equation}
     where the permutation $\tau$ exchanges only the string labeled by $\tabt(N_m')$ with other strings. Renaming $\psi_\sigma = \psi_{\sigma'} \psi_\tau = \psi_{\sigma'\tau}$, we easily see that $\sigma$ can be any permutation satisfying conditions \ref{C:perm1} and \ref{C:perm2}.
 \end{proof}
 
 Next we wish to compute $v_L y_k$, which will naturally be useful in the course of computing $v_\tabt y_k = v_L \psi_{\sigma} y_k$. In the diagram we will use an operator $\phi_{ab}$  which we define to be a crossing between two strings unless $i_a = i_b$, in which case $\phi_{ab} = 1$. Furthermore, we will use the compact notation $\phi_{-ab} \coloneq \phi_{m-a,m-b}$ to avoid visual clutter.
 
\begin{lem} \label{L:YActionYkVL}
    Let $m \geq 2$, $k \in \bm{s}^{(r)}$ with $\deg(y_k e(\resi(\tabt))) = 2$, and let $v_L \in S^\lambda$ be as in \cref{L:YActionVTDescription}. Then we have
    \begin{equation*}
        v_L y_k = \pm 
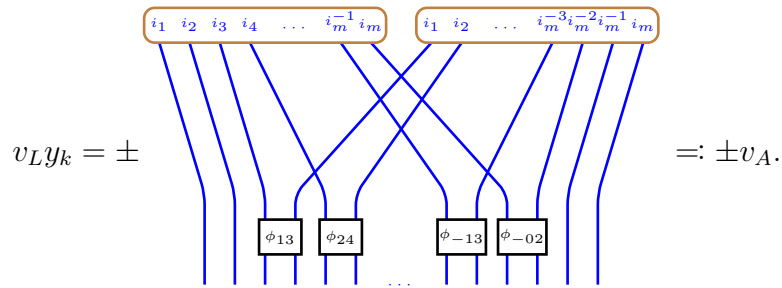
        \begin{braid}
	\def\h{5};
	\def\sep{2}
	\def\a{7+\sep};
	\def\dashpart{0.8}
	\def\sp{1.5}
	\def\hrect{1.15}
	\def\shiftdown{-3}
	
	\draw(\sp + 0, \shiftdown) -- (\sp + 0, 0) -- (0, \h) node[anchor=south]{$i_1$} ;
	\draw(\sp + 1, \shiftdown) --(\sp +1, 0) -- (1, \h) node[anchor=south]{$i_2$};
	\draw(\sp + 2, \shiftdown) --(\sp +2, 0) -- (2, \h) node[anchor=south]{$i_3$};
	\draw(\sp + 3, \shiftdown) --(\sp +3, 0) -- (\a, \h) node[anchor=south]{$i_1$};
	\draw(\sp + 4, \shiftdown) --(\sp +4, 0) -- (3, \h) node[anchor=south]{$i_4$};
	\draw(\sp + 5, \shiftdown) --(\sp +5, 0) -- (\a+1, \h) node[anchor=south]{$i_2$};
	
	\draw (\sp + 5.5 + \sep/2, -3) node {$\cdots$};
	\draw (3.5 + \sep/2, \h) node[anchor=south]{$\cdots$};
	\draw (\a + 1.5 + \sep/2, \h) node[anchor=south]{$\cdots$};
	
	\draw (\sp + 6+\sep, \shiftdown) --(\sp + 6 + \sep, 0) -- (4  + \sep, \h) node[anchor=south]{$i_m^{-1}$};
	\draw (\sp + 7+\sep, \shiftdown) -- (\sp + 7 + \sep, 0) -- (\a + 2  + \sep, \h) node[anchor=south]{$i_m^{-3}$};
	\draw  (\sp + 8+\sep, \shiftdown) --(\sp +8 + \sep, 0) -- (5  + \sep, \h) node[anchor=south]{$i_m$};
	\draw  (\sp + 9+\sep, \shiftdown) --(\sp + 9 + \sep, 0) -- (\a + 3  + \sep, \h) node[anchor=south]{$i_m^{-2}$};
	\draw  (\sp + 10+\sep, \shiftdown) --(\sp + 10 + \sep, 0) -- (\a + 4  + \sep, \h) node[anchor=south]{$i_m^{-1}$};
	\draw  (\sp + 11+\sep, \shiftdown) --(\sp + 11 + \sep, 0) -- (\a + 5  + \sep, \h) node[anchor=south]{$i_m$};
	
	\draw[black, sharp corners, fill=white] (1.8 + \sp, \shiftdown + 1) rectangle (3.2 + \sp , \shiftdown + 1 + \hrect);
	\draw[black, sharp corners, fill=white] (3.8 + \sp, \shiftdown + 1) rectangle (5.2 + \sp , \shiftdown + 1 + \hrect);
	\draw[black, sharp corners, fill=white] (5.7 + \sep + \sp, \shiftdown + 1) rectangle (7.3 + \sp + \sep , \shiftdown + 1 + \hrect);
	\draw[black, sharp corners, fill=white] (7.7 + \sep + \sp, \shiftdown + 1) rectangle (9.3 + \sp + \sep , \shiftdown + 1 + \hrect);
	
	\draw[brown] (-0.5, \h) rectangle (\a - 1.5 , \h + \hrect);
	\draw (\sp + 2.5, \shiftdown +1 + 0.5*\hrect) node[black, scale=1]{$\phi_{13}$};
	\draw (\sp + 4.5, \shiftdown +1 + 0.5*\hrect) node[black, scale=1]{$\phi_{24}$};
	\draw (\sp + 6.5+\sep, \shiftdown +1 + 0.5*\hrect) node[black, scale=1]{$\phi_{-13}$};
	\draw (\sp + 8.5+\sep, \shiftdown +1 + 0.5*\hrect) node[black, scale=1]{$\phi_{-02}$};
	\draw[brown] (\a - 0.5, \h) rectangle (2*\a + 0.5 , \h + \hrect);

\end{braid} \eqcolon \pm v_A.
    \end{equation*} 
\end{lem}

\begin{proof}
    We show only a sketch of the computation for $k=1$. We wish to compute
    \begin{equation}
         \begin{braid}
	\def\h{5};
	\def\sep{2}
	\def\a{7+\sep};
	\def\dashpart{0.8}
	\def\sp{1.5}
	\def\hrect{1.15}
	\def\shiftdown{-5}
	
	\draw (\sp + 0, 0) -- (\a, \h) node[anchor=south]{$i_1$} ;
	\draw(\sp +1, 0) -- (\a+1, \h) node[anchor=south]{$i_2$};
	\draw(\sp +2, 0) -- (0, \h) node[anchor=south]{$i_1$};
	\draw(\sp+3,0) -- (\a+2, \h) node[anchor=south]{$i_3$};
	\draw(\sp +4, 0) -- (1, \h) node[anchor=south]{$i_2$};
	\draw(\sp +5, 0) -- (\a+3, \h) node[anchor=south]{$i_4$};
	
	\draw (\sp + 5.5 + \sep/2, 0) node {$\cdots$};
	\draw (1.5 + \sep/2, \h) node[anchor=south]{$\cdots$};
	\draw (\a + 3.5 + \sep/2, \h) node[anchor=south]{$\cdots$};
	
	\draw(\sp + 6 + \sep, 0) -- (2  + \sep, \h) node[anchor=south]{$i_m^{-3}$};
	\draw(\sp + 7 + \sep, 0) -- (\a + 4  + \sep, \h) node[anchor=south]{$i_m^{-1}$};
	\draw (\sp +8 + \sep, 0) -- (3  + \sep, \h) node[anchor=south]{$i_m^{-2}$};
	\draw (\sp + 9 + \sep, 0) -- (\a + 5  + \sep, \h) node[anchor=south]{$i_m$};
	\draw (\sp + 10 + \sep, 0) -- (4  + \sep, \h) node[anchor=south]{$i_m^{-1}$};
	\draw (\sp + 11 + \sep, 0) -- (5  + \sep, \h) node[anchor=south]{$i_m$};
	\draw[brown] (-0.5, \h) rectangle (\a - 1.5 , \h + \hrect);
	
	\draw[brown] (\a - 0.5, \h) rectangle (2*\a + 0.5 , \h + \hrect);
	\path (\sp + 0, 0) -- (\a, \h) coordinate[pos = 0.1] (A);
    	\greendot(A);

\end{braid}.
    \end{equation}
    We simply slide the dot past crossings until it reaches the top of the diagram and vanishes. We are left with some summands corresponding to the breaking of crossings for the string $s^{(r)}_1$. But breaking a crossing above the first crossing immediately gives zero due to a Specht relation, so that Relation \ref{E:dotslideC} gives simply
    \begin{equation}
         -  \begin{braid}
	\def\h{5};
	\def\sep{2}
	\def\a{7+\sep};
	\def\dashpart{0.8}
	\def\sp{1.5}
	\def\hrect{1.15}
	\def\shiftdown{0}
	
	\draw (\sp + 0, 0) -- (0, \h) node[anchor=south]{$i_1$};
	\draw(\sp +1, 0) -- (\a+1, \h) node[anchor=south]{$i_2$};
	\path(\sp +2, 0) -- (0, \h) coordinate[pos=0.2] (A);
	\draw(\sp+2,0) -- (A) -- (\a, \h) node[anchor=south]{$i_1$};
	\draw(\sp + 3, 0) -- (\a+2, \h) node[anchor=south]{$i_3$};
	\draw(\sp +4, 0) -- (1, \h) node[anchor=south]{$i_2$};
	\draw(\sp +5, 0) -- (\a+3, \h) node[anchor=south]{$i_4$};
	
	\draw (\sp + 5.5 + \sep/2, \shiftdown) node {$\cdots$};
	\draw (1.5 + \sep/2, \h) node[anchor=south]{$\cdots$};
	\draw (\a + 3.5 + \sep/2, \h) node[anchor=south]{$\cdots$};
	
	\draw(\sp + 6 + \sep, 0) -- (2  + \sep, \h) node[anchor=south]{$i_m^{-3}$};
	\draw(\sp + 7 + \sep, 0) -- (\a + 4  + \sep, \h) node[anchor=south]{$i_m^{-1}$};
	\draw (\sp +8 + \sep, 0) -- (3  + \sep, \h) node[anchor=south]{$i_m^{-2}$};
	\draw (\sp + 9 + \sep, 0) -- (\a + 5  + \sep, \h) node[anchor=south]{$i_m$};
	\draw (\sp + 10 + \sep, 0) -- (4  + \sep, \h) node[anchor=south]{$i_m^{-1}$};
	\draw (\sp + 11 + \sep, 0) -- (5  + \sep, \h) node[anchor=south]{$i_m$};
	\draw[brown] (-0.5, \h) rectangle (\a - 1.5 , \h + \hrect);
	
	\draw[brown] (\a - 0.5, \h) rectangle (2*\a + 0.5 , \h + \hrect);
	\path (\sp + 0, 0) -- (\a, \h) coordinate[pos = 0.1] (A);

\end{braid}.
    \end{equation}
    Next we slide the string $s^{(r)}_3$ past the crossings of $s^{(r)}_2$ until it reaches the top of the diagram, where the diagram evaluates to zero due to a Specht relation. Again, the only non-zero term corresponds to the summand where the lowest crossing of $s^{(r)}_2$ is cut, where Relation \ref{E:braidC} gives
    \begin{equation}
         \begin{braid}
	\def\h{5};
	\def\sep{2}
	\def\a{7+\sep};
	\def\dashpart{0.8}
	\def\sp{1.5}
	\def\hrect{1.15}
	\def\shiftdown{0}
	
	\draw (\sp + 0, 0) -- (0, \h) node[anchor=south]{$i_1$};
	\draw(\sp +1, 0) -- (1, \h) node[anchor=south]{$i_2$};
	\path(\sp +2, 0) -- (0, \h) coordinate[pos=0.2] (A);
	\draw(\sp+2,0) -- (A) -- (\a, \h) node[anchor=south]{$i_1$};
	\draw(\sp + 3, 0) -- (\a+2, \h) node[anchor=south]{$i_3$};
	\path(\sp +4, 0) -- (1, \h) coordinate[pos=0.23] (A);
	\draw(\sp +4, 0)  -- (A) -- (\a + 1, \h)node[anchor=south]{$i_2$};
	\draw(\sp +5, 0) -- (\a+3, \h) node[anchor=south]{$i_4$};
	
	\draw (\sp + 5.5 + \sep/2, \shiftdown) node {$\cdots$};
	\draw (1.5 + \sep/2, \h) node[anchor=south]{$\cdots$};
	\draw (\a + 3.5 + \sep/2, \h) node[anchor=south]{$\cdots$};
	
	\draw(\sp + 6 + \sep, 0) -- (2  + \sep, \h) node[anchor=south]{$i_m^{-3}$};
	\draw(\sp + 7 + \sep, 0) -- (\a + 4  + \sep, \h) node[anchor=south]{$i_m^{-1}$};
	\draw (\sp +8 + \sep, 0) -- (3  + \sep, \h) node[anchor=south]{$i_m^{-2}$};
	\draw (\sp + 9 + \sep, 0) -- (\a + 5  + \sep, \h) node[anchor=south]{$i_m$};
	\draw (\sp + 10 + \sep, 0) -- (4  + \sep, \h) node[anchor=south]{$i_m^{-1}$};
	\draw (\sp + 11 + \sep, 0) -- (5  + \sep, \h) node[anchor=south]{$i_m$};
	\draw[brown] (-0.5, \h) rectangle (\a - 1.5 , \h + \hrect);
	
	\draw[brown] (\a - 0.5, \h) rectangle (2*\a + 0.5 , \h + \hrect);
	\path (\sp + 0, 0) -- (\a, \h) coordinate[pos = 0.1] (A);

\end{braid}.
    \end{equation}
    The process can be continued in exactly the same way, with the only caveat that Relation \ref{E:braidC} adds some dots to our diagrams. That is, we reach a sum of diagrams as follows:
    
    \begin{equation} \label{E:YActionYSumDiagram}
        \sum_Y \pm \begin{braid}
	\def\h{5};
	\def\sep{2}
	\def\a{7+\sep};
	\def\dashpart{0.8}
	\def\sp{1.5}
	\def\hrect{1.15}
	\def\shiftdown{-3}
	
	\draw(\sp + 0, 0) -- (0, \h) node[anchor=south]{$i_1$} ;
	\draw(\sp +1, 0) -- (1, \h) node[anchor=south]{$i_2$};
	\draw(\sp +3, 0) -- (2, \h) node[anchor=south]{$i_3$};
	\draw(\sp +2, 0) -- (\a, \h) node[anchor=south]{$i_1$};
	\draw(\sp +5, 0) -- (3, \h) node[anchor=south]{$i_4$};
	\draw(\sp +4, 0) -- (\a+1, \h) node[anchor=south]{$i_2$};
	
	\draw (\sp + 5.5 + \sep/2, 0) node {$\cdots$};
	\draw (3.5 + \sep/2, \h) node[anchor=south]{$\cdots$};
	\draw (\a + 1.5 + \sep/2, \h) node[anchor=south]{$\cdots$};
	
	\draw (\sp + 7 + \sep, 0) -- (4  + \sep, \h) node[anchor=south]{$i_m^{-1}$};
	\draw (\sp + 6 + \sep, 0) -- (\a + 2  + \sep, \h) node[anchor=south]{$i_m^{-3}$};
	\draw  (\sp +9 + \sep, 0) -- (5  + \sep, \h) node[anchor=south]{$i_m$};
	\draw  (\sp + 8 + \sep, 0) -- (\a + 3  + \sep, \h) node[anchor=south]{$i_m^{-2}$};
	\draw  (\sp + 10 + \sep, 0) -- (\a + 4  + \sep, \h) node[anchor=south]{$i_m^{-1}$};
	\draw  (\sp + 11 + \sep, 0) -- (\a + 5  + \sep, \h) node[anchor=south]{$i_m$};
	
	\draw[brown] (-0.5, \h) rectangle (\a - 1.5 , \h + \hrect);
	\draw[brown] (\a - 0.5, \h) rectangle (2*\a + 0.5 , \h + \hrect);

\end{braid} Y \eqcolon \pm \sum_Y v_A' Y
    \end{equation} 
    where the sum is over all collections $Y$ of dots having one dot per pair of strings $s_a, s_a' \in \bm{s}^{(r)}$  which, in $v_L$, reach nodes $N_a, N_a'$ to the right of nodes $N_{a-1},N_{a-1}'$ of residue $0$ or $e$. That is, in $v_L$, string $s_a$ reaches $N_a \in \Arm(\lambda)$, while $s_a'$ reaches $N_a' \in \Leg(\lambda)$. However, the string $s_a'$ does not meet any other strings of its residue in $v_A'$, so that $v_A' y_{s_a'} = 0$. It follows that the only non-zero summand in Diagram \ref{E:YActionYSumDiagram} which does not vanish is given by the $Y$ where all the $s_a$ have a dot and none of the $s_a'$ have a dot. These dots break all the crossings between strings of the same residue, and we reach the diagram in the lemma statement.
\end{proof}

Next we wish to compute $v_\tabt y_k$ in more generality. Let $\sigma$ be an arbitrary permutation that satisfies conditions \ref{C:perm1} and \ref{C:perm2} above. Suppose first that $k \in \bm{s}^{(r+1)} \subset \bm{s}^{(r)}$ and $l_{r+1} = L$ in $\tabt$. According to \cref{L:YActionSubsetStructure}, $\resi^{(r+1)} = (i, i)$, so that string $k$ is adjacent to and crosses a string with the same residue in $v_\tabt$. Suppose for instance that $i_{k+1} = i_k$, so that $v_\tabt = v' \psi_k$. In that case,
\begin{equation}
    v_\tabt y_k = v' \psi_k y_k = (v' y_{k+1}) \psi_k - v',
\end{equation}
according to \cref{E:dotslideC}. It follows that to understand the $y$-action on $v_\tabt$, it is enough to understand the action on $v'$. In the sequel we make the further assumption that in $\psi_\sigma$, two adjacent strings of the same residue do not cross each other.

\begin{lem} \label{L:YActionPsiVa}
    Let $v_\tabt = v_L \psi_\sigma$, where $\sigma$ is a permutation satisfying conditions \ref{C:perm1}, \ref{C:perm2}, and suppose that adjacent strings of equal residue do not cross each other in $\psi_\sigma$. Suppose also that $\deg(y_k e(\resi(\tabt))) = 2$. Then
    \begin{equation*}
        v_\tabt y_k = \pm v_A \psi_\sigma.
    \end{equation*}
\end{lem}
\begin{proof}
    The following is a sketch of the strings involved in the argument as they appear in $v_\tabt$.
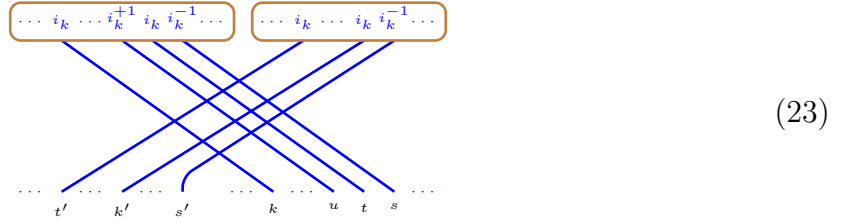
    \begin{equation} \label{E:YActionPsiVASketch}
        \begin{braid}
    	\def\h{5};
    	\def\sep{2}
    	\def\a{6+\sep};
    	\def\dashpart{0.8}
    	\def\sp{1.5}
    	\def\hrect{1.25}
    	\def\shiftdown{-3}
    	
    	\draw(0, 0) node[black, anchor=north]{$t'$}-- (\a, \h) node[anchor=south]{$i_k$} ;
        \draw(-1,0) node{$\cdots$};
        \draw(-1,\h + \hrect*0.5) node{$\cdots$};
        \draw(1,0) node{$\cdots$};
        \draw(1,\h + \hrect*0.5) node{$\cdots$};
        \draw(\a,0) node{$\cdots$};
        \draw(\a+1,\h + \hrect*0.5) node{$\cdots$};
        \draw(\a-2,0) node{$\cdots$};
        \draw(\a-1,\h + \hrect*0.5) node{$\cdots$};
        \draw(3,0) node{$\cdots$};
        \draw(5,\h + \hrect*0.5) node{$\cdots$};
        \draw(\a+4,0) node{$\cdots$};
        \draw(\a+4,\h + \hrect*0.5) node{$\cdots$};
    	\draw(2, 0) node[black, anchor=north]{$k'$} -- (\a + 2, \h) node[anchor=south]{$i_k$};
        \draw(\a-1, 0) node[black, anchor=north]{$k$} -- (0, \h) node[anchor=south]{$i_k$} ;
    	\draw(\a+2, 0) node[black, anchor=north]{$t$} -- (3, \h) node[anchor=south]{$i_k$};
        \draw(4, 0) node[black, anchor=north]{$s'$} -- (4, 0.5) -- (\a+3,\h) coordinate[pos=0.2] node[anchor=south]{$i_k^{-1}$};
        \draw(3+\a, 0) node[black, anchor=north]{$s$} -- (4,\h) node[anchor=south]{$i_k^{-1}$};
        \draw(1+\a, 0) node[black, anchor=north]{$u$}-- (2,\h) node[anchor=south]{$i_k^{+1}$};

    	\draw[brown] (-1.7, \h) rectangle (\a - 2.3 , \h + \hrect);
    	\draw[brown] (\a - 1.7, \h) rectangle (2*\a - 1.3 , \h + \hrect);
    
        \end{braid} 
    \end{equation}
    We are tasked with showing that $v_L \psi_\sigma y_k = v_L y_{\sigma^{-1}(k)} \psi_\sigma$. The lemma will easily follow, as
    \begin{equation}
        v_\tabt y_k = v_L \psi_\sigma y_k = v_L y_{\sigma^{-1}(k)} \psi_\sigma = \pm v_A \psi_\sigma.
    \end{equation}
    The reduced expression for $v_\tabt$ is fully commutative, so that the string labelled by $k$ is either stubborn or co-stubborn. We suppose without loss of generality that it is stubborn. As the dot $y_k$ slides past the crossings of string $k$, it may encounter a crossing with a co-stubborn string of the same residue before reaching $v_L$. We label this string $k'$. The dot slide leaves a summand $\pm v^*_1$ where the crossing between strings $k$ and $k'$ is cut, according to \cref{E:dotslideC}. To prove the lemma it suffices to show that $v^*_1 = 0$.

    First, notice that for strings $k$ and $k'$ to cross below $v_L$ in $v_\tabt$, residue $i_k$ must appear four times inside $\resi^{(r)}$, string $k$ must reach the smallest node of its residue and $k'$ must reach the largest node of its residue.  Moreover, strings $k, k'$ may not be adjacent, according to the condition imposed on $\psi_\sigma$ in the statement of the lemma. Consequently, there is at least one string between $k'$ and $k$, and at least one of the following holds for the tableau corresponding to $v_1^*$.
    \begin{enumerate}
        \item The number in the node immediately to the left of $\tabt_{{v^*_1}}^{-1}(k)$ is smaller than $k$, or
        \item The number in the node immediately to the right of $\tabt_{{v^*_1}}^{-1}(k')$ is larger than $k'$.
    \end{enumerate}
    We suppose that the former condition holds without loss of generality, and we let the number in question be $s'$. It has residue $i_{k-1}$. (Recall that we are using the notation $i_1, \dots, i_m$ for the short residue sequence in the arm and leg covered by $\bm{s}^{(r)}$.)

    String $k$ can slide through all the crossings of string $s'$, leaving summands where one of these crossings is cut. If string $k$ slides past all the crossings of string $s'$ without cutting any, the diagram vanishes due to \cref{E:specht}. In fact, exactly one of the crossings of $s'$ which string $k$ slides past features a co-stubborn string $s$ sharing the residue $i_{k-1} = \res(s')$. This is because, if $s'$ reaches $N \in \Arm(\lambda)$ in $v_\tabt$, then there is a corresponding string $s$ reaching the mirrored node $N' \in \Leg(\lambda)$ in $v_\tabt$, and standardness requires $s > k$. If there were another string $s''$ with residue $i_{k-1}$ reaching a node in $\Leg(\lambda)$, we would have $s'' < k$, as this string must reach a node lower than $\tabt^{-1}(k)$. 

    We call $v^{*}_2$ the diagram obtained from $v^*_1$ by cutting the crossing between the co-stubborn string $s$ and the stubborn string $s'$ via the slide of string $k$ past this crossing, according to \cref{E:braidC}. The $(s,s')$-crossing in $v^*_1$ is part of the middle horizontal of $(i,i)$-crossings in the upper part of the diagram $v^*_1$ described by $v_L$, so that, in diagram $v^{*}_2$, string $s'$ can slide past the crossings of a stubborn string $t$ of residue $i_k$, again vanishing at the top of the diagram due to \cref{E:specht} and leaving only a diagram $v^{*}_3$ where the crossing between string $t$ and a co-stubborn string $t'$ of residue $i_k$ is cut. 
    
    Finally, in diagram $v^{*}_3$, string $t'$ can slide past the crossings of a string $u$ of residue $i_{k+1}$ without meeting any $(i_{k+1},i_{k+1})$-crossings, so that $v^{*}_3 = 0$ due to \cref{E:specht}, and therefore $v^*_1 = 0$, which completes the proof of the lemma. 
\end{proof}

Now we have reached the final problem in this section: we need to compute $v_A \psi_\sigma$. Recall that we're currently making the assumption that adjacent strings of equal residue do not cross.

\begin{lem} \label{L:YActionPsiVaComp}
    Let $\tabt' = \phi_{\resi}^\lambda(l_1, \dots, l_{r-1}, A, l_{r+1}, \dots, l_d)$ and let $v_A$ be as in \cref{L:YActionYkVL}. Then
    \begin{equation*}
        v_A \psi_\sigma = \pm v_{\tabt'}.
    \end{equation*}
\end{lem}
\begin{proof}
    According to \cref{L:YActionVanishing,L:YActionPsiVa}, $v_\tabt y_k = \pm v_A \psi_\sigma = a_{\tabt'} v_{\tabt'}$. We must show that $a_{\tabt'} = \pm 1$. Let $\sigma^*$ be a permutation so that $v_L \psi_{\sigma^*}$ exchanges all stubborn strings with all co-stubborn strings. It is clear that $\sigma^*$ is the longest (in Bruhat order) permutation satisfying conditions \ref{C:perm1} and \ref{C:perm2}, so that for any $\sigma'$ satisfying those two conditions, we have $\psi_{\sigma'} \psi_{\tau_{\sigma'}} = \psi_{\sigma^*}$ for some permutation $\tau_{\sigma'}$. It is enough to show that $v_A \psi_{\sigma^*}$ gives a diagram $v_{\tabt^*}$ in the Specht module basis with coefficient $\pm 1$. Indeed, in that case, $$v_A \psi_{\sigma^*} = v_A \psi_{\sigma} \psi_{\tau_\sigma} = a_{\tabt'} v_{\tabt'} \psi_{\tau_\sigma} = \pm v_{\tabt^*},$$
    so it follows that $a_{\tabt'} = \pm 1$.

    We proceed to the computation of $v_A \psi_{\sigma^*}$. For the purposes of induction, we will perform a slightly more general computation in resolving the computation in \cref{E:YActionV1V2} under the following conditions for $v$, which are easily seen to hold for the center strings of diagram $v_A$.
    \begin{enumerate}
        \item The residue sequences $\resi = (i_1, \dots, i_m)$ and $\resj= (j_1, \dots, j_m)$ are both short residue sequences.
        \item The parity of $i_1$ and $j_1$ coincides, so that: 
        \begin{enumerate}
            \item $i_a = j_b$ only if $a-b$ is even.
            \item $i_a$ is adjacent to $j_b$ only if $a-b$ is odd.
        \end{enumerate}
        \item The following configurations are \emph{not} allowed:
        \begin{enumerate}
            \item $(i_a,i_{a+1}) = (j_b, j_{b+1})$ for any $b \leq a$.
            \item $(i_a, i_{a+1}, j_b, j_{b+1})$ is a substring of $\mathcal{I}$, for any $b \leq a$.
        \end{enumerate}
        
        \item A string of residue $i = 0$ or $e$ only crosses strings of residues $j \unrelated i$ in $v$.
    \end{enumerate}
    
    \begin{equation} \label{E:YActionV1V2}
\begin{braid}
    \def\h{5};
    \def\sep{2}
    \def\a{7+\sep};
    \def\dashpart{0.8}
    \def\sp{1.5}
    \def\hrect{1.15}
    \def\shiftdown{-3}
    \def\sshiftdown{-7}

    \draw [black, sharp corners, decorate,decoration={brace,amplitude=5pt,raise=-1ex}]
  (0,\shiftdown + 0.5) -- (0,\h+\hrect) node[midway, anchor=east]{$v$};
  \draw [black, sharp corners, decorate,decoration={brace,amplitude=5pt,raise=-1ex}]
  (0,\sshiftdown + 0.5) -- (0,\shiftdown-0.5) node[midway, anchor=east]{$\psi_D$};
    
    % Left side
    \draw(\sp + 6 + \sep, \sshiftdown) -- (\sp + 2, \shiftdown) --(\sp +2, 0) -- (2, \h) node[anchor=south]{$i_1$};
    \draw(\sp+1, \sshiftdown) -- (\sp+1, \shiftdown - 1) --  (\sp + 3, \shiftdown) --(\sp +3, 0) -- (\a, \h) node[anchor=south]{$j_1$};
    \draw (\sp + 7+\sep, \sshiftdown) -- (\sp + 7 + \sep, \sshiftdown + 1) -- (\sp + 4, \shiftdown) --(\sp +4, 0) -- (3, \h) node[anchor=south]{$i_2$};
    \draw(\sp+2, \sshiftdown) -- (\sp+2, \shiftdown - 2) -- (\sp + 5, \shiftdown) --(\sp +5, 0) -- (\a+1, \h) node[anchor=south]{$j_2$};
    
    % Middle points
    \draw (\sp + 3, \sshiftdown) node {$\cdots$};
    \draw (\sp + 8 + \sep, \sshiftdown) node {$\cdots$};
    \draw (\sp + 5.5+\sep*0.5, \shiftdown + 1) node {$\cdots$};
    \draw (3.5 + \sep/2, \h) node[anchor=south]{$\cdots$};
    \draw (\a + 1.5 + \sep/2, \h) node[anchor=south]{$\cdots$};
    
    % Right side
    \draw  (\sp+9 + \sep, \sshiftdown) -- (\sp+9 + \sep, \shiftdown - 1) -- (\sp + 6+\sep, \shiftdown) --(\sp + 6 + \sep, 0) -- (4  + \sep, \h) node[anchor=south]{$i_m^{-1}$};
    \draw (\sp + 4, \sshiftdown) -- (\sp + 4, \sshiftdown + 1) --  (\sp + 7+\sep, \shiftdown) -- (\sp + 7+\sep, \shiftdown) -- (\sp + 7 + \sep, 0) -- (\a + 2  + \sep, \h) node[anchor=south]{$j_m^{-1}$};
    \draw  (\sp+10 + \sep, \sshiftdown) -- (\sp+10 + \sep, \shiftdown - 1) -- (\sp + 8+\sep, \shiftdown) --(\sp +8 + \sep, 0) -- (5  + \sep, \h) node[anchor=south]{$i_m$};
    \draw (\sp + 5, \sshiftdown) -- (\sp + 9+\sep, \shiftdown) --(\sp + 9 + \sep, 0) -- (\a + 3  + \sep, \h) node[anchor=south]{$j_m$};

    % Rectangles and labels
    \draw[black, sharp corners, fill=white] (1.8 + \sp, \shiftdown + 1) rectangle (3.2 + \sp , \shiftdown + 1 + \hrect);
    \draw[black, sharp corners, fill=white] (3.8 + \sp, \shiftdown + 1) rectangle (5.2 + \sp , \shiftdown + 1 + \hrect);
    \draw[black, sharp corners, fill=white] (5.7 + \sep + \sp, \shiftdown + 1) rectangle (7.3 + \sp + \sep , \shiftdown + 1 + \hrect);
    \draw[black, sharp corners, fill=white] (7.7 + \sep + \sp, \shiftdown + 1) rectangle (9.3 + \sp + \sep , \shiftdown + 1 + \hrect);
    \draw (\sp + 2.5, \shiftdown +1 + 0.5*\hrect) node[black, scale=1]{$\phi_{11}$};
    \draw (\sp + 4.5, \shiftdown +1 + 0.5*\hrect) node[black, scale=1]{$\phi_{22}$};
    \draw (\sp + 6.5+\sep, \shiftdown +1 + 0.5*\hrect) node[black, scale=1]{$\phi_{-11}$};
    \draw (\sp + 8.5+\sep, \shiftdown +1 + 0.5*\hrect) node[black, scale=1]{$\phi_{-00}$};

    % Top bounding rectangles
    \draw[brown] (1.5, \h) rectangle (\a - 1.5 , \h + \hrect);
    \draw[brown] (\a - 0.5, \h) rectangle (2*\a - 1.5 , \h + \hrect);
\end{braid}
= \pm
\begin{braid}
    \def\h{5};
    \def\sep{2}
    \def\a{7};
    \def\dashpart{0.8}
    \def\sp{1.5}
    \def\hrect{1.15}
    \def\shiftdown{-3}
    \def\sshiftdown{-7}

    \draw(0, 0) -- (0, \h) node[anchor=south]{$i_1$};
    \draw(1, 0) -- (1, \h) node[anchor=south]{$i_2$};
    \draw (1.5 + \sep*0.5, 0) node {$\cdots$};
    \draw (1.5 + \sep*0.5, \h + \hrect*0.5) node {$\cdots$};
    \draw(2 + \sep, 0) -- (2+\sep, \h) node[anchor=south]{$i_m^{-1}$};
    \draw(3 + \sep, 0) -- (3+\sep, \h) node[anchor=south]{$i_m$};

    \draw(\a, 0) -- (\a+0, \h) node[anchor=south]{$j_1$};
    \draw(\a+ 1, 0) -- (\a+1, \h) node[anchor=south]{$j_2$};
    \draw (\a+1.5 + \sep*0.5, 0) node {$\cdots$};
    \draw (\a+1.5 + \sep*0.5, \h + \hrect*0.5) node {$\cdots$};
    \draw(\a+2 + \sep, 0) -- (\a+2+\sep, \h) node[anchor=south]{$j_m^{-1}$};
    \draw(\a+3 + \sep, 0) -- (\a+3+\sep, \h) node[anchor=south]{$j_m$};

    \draw[black, sharp corners, fill=white] (3+\sep - 0.5, \h*0.5 - \hrect*0.5) rectangle (\a + 0.5,\h*0.5 + \hrect*0.5);
    \draw[black] (3*0.5 + \sep*0.5 + \a*0.5, \h*0.5) node{$\phi_{m1}$};
    
    % Top bounding rectangles
    \draw[brown] (-0.5, \h) rectangle (\a - 1.5 , \h + \hrect);
    \draw[brown] (\a - 0.5, \h) rectangle (2*\a - 1.5 , \h + \hrect);
\end{braid}
    \end{equation} 
    For the purposes of this diagram we define $\phi_{ab}$ in close analogy with \cref{L:YActionYkVL}, so that $\phi_{ab}$ represents a crossing if and only if $i_a \neq j_b$. We have used the notation $\phi_{-ab} \coloneq \phi_{m-a, m-b}$ for graphical considerations. Next we perform an induction on $m$ to show that \cref{E:YActionV1V2} holds. The statement can be easily checked for $m = 1, 2$.

    Next suppose that $i_1 = j_1$. Then condition (3.a) implies that $i_2 \unrelated j_2$, so that near the center-left of the diagram we have
\begin{equation}
\begin{braid}
    \def\h{5};
    \def\ha{3};
    \def\hb{1};
    \def\hc{-1};
    \def\sep{2}
    \def\a{7};
    \def\dashpart{0.8}
    \def\sp{1.5}
    \def\hrect{1.25}
    \def\shiftdown{-3}
    \def\sshiftdown{-7}

    \draw(2, \hc) -- (0, \hb) -- (0, \ha) -- (0, \h) node[anchor=south]{$i_1$};
    \draw(0, \hc) -- (0, \hc + 1) -- (1, \hb) --(1, \ha) -- (2, \h) node[anchor=south]{$j_1$};
    \draw(3, \hc) -- (3, \hc + 1) --(2, \hb) --(2,\ha) -- (1,\h) node[anchor=south]{$i_{2}$};
    \draw(1, \hc) -- (3, \hb) -- (3, \ha) -- (3, \h) node[anchor=south]{$j_{2}$};
    \draw[black, sharp corners, fill=white] (-0.3, \ha*0.5 + \hb*0.5 - 0.5*\hrect) rectangle (1.3,\ha*0.5 + \hb*0.5 + \hrect*0.5);
    \draw[black] (0.5, \ha*0.5 + \hb*0.5) node{$\phi_{11}$};
    \draw[black, sharp corners, fill=white] (1.7, \ha*0.5 + \hb*0.5 - 0.5*\hrect) rectangle (3.3,\ha*0.5 + \hb*0.5 + \hrect*0.5);
    \draw[black] (2.5, \ha*0.5 + \hb*0.5) node{$\phi_{22}$};

    \draw[brown] (-0.5, \h) rectangle (1.6 , \h + \hrect);
\end{braid}
=
\begin{braid}
    \def\h{5};
    \def\ha{3};
    \def\hb{1};
    \def\hc{-1};
    \def\sep{2}
    \def\a{7};
    \def\dashpart{0.8}
    \def\sp{1.5}
    \def\hrect{1.25}
    \def\shiftdown{-3}
    \def\sshiftdown{-7}

    \draw(2, \hc) -- (0, \hb) -- (0, \ha)  -- (0, \h) node[anchor=south]{$i_1$};
    \draw(0, \hc) -- (0, \hc + 1) -- (1, \hb) --(1, \ha) -- (2, \h) node[anchor=south]{$j_1$};
    \draw(3, \hc) -- (3, \hc + 1) --(2, \hb) --(3,\ha) -- (3,\h) node[anchor=south]{$j_2$};
    \draw(1, \hc) -- (3, \hb) -- (2, \ha) -- (1, \h) node[anchor=south]{$i_2$};

    \draw[brown] (-0.5, \h) rectangle (1.6 , \h + \hrect);
\end{braid}
=
\begin{braid}
    \def\h{5};
    \def\ha{3};
    \def\hb{1};
    \def\hc{-1};
    \def\sep{2}
    \def\a{7};
    \def\dashpart{0.8}
    \def\sp{1.5}
    \def\hrect{1.25}
    \def\shiftdown{-3}
    \def\sshiftdown{-7}

    \draw(2, \hc) -- (0, \ha)  -- (0, \h) node[anchor=south]{$i_1$};
    \draw(0, \hc) -- (0, \hb) -- (2, \h) node[anchor=south]{$j_1$};
    \draw(3, \hc) -- (3,\h) node[anchor=south]{$j_2$};
    \draw(1, \hc) -- (2, \hb) -- (2, \ha) -- (1, \h) node[anchor=south]{$i_2$};

    \draw[brown] (-0.5, \h) rectangle (1.6 , \h + \hrect);
\end{braid}
=
\begin{braid}
    \def\h{5};
    \def\ha{3};
    \def\hb{1};
    \def\hc{-1};
    \def\sep{2}
    \def\a{7};
    \def\dashpart{0.8}
    \def\sp{1.5}
    \def\hrect{1.25}
    \def\shiftdown{-3}
    \def\sshiftdown{-7}

    \draw(2, \hc)  -- (2, \h) node[anchor=south]{$j_1$};
    \draw(0, \hc) -- (0, \h) node[anchor=south]{$i_1$};
    \draw(3, \hc) -- (3,\h) node[anchor=south]{$j_2$};
    \draw(1, \hc) -- (1, \h) node[anchor=south]{$i_2$};

    \draw[brown] (-0.5, \h) rectangle (1.6 , \h + \hrect);
\end{braid}.
    \end{equation}
    Suppose to the contrary that $i_1 \neq j_1$. Then $i_1 \unrelated j_1$, and it follows that 
    \begin{equation} 
\begin{braid}
    \def\h{5};
    \def\ha{3};
    \def\hb{1};
    \def\hc{-2};
    \def\sep{2}
    \def\a{7};
    \def\dashpart{0.8}
    \def\sp{1.5}
    \def\hrect{1.25}
    \def\shiftdown{-3}
    \def\sshiftdown{-7}
    \draw(0, 0) --  (1, \h*0.5) -- (1, \h) node[anchor=south]{$j_1$};
    \draw(1, 0) -- (0, \h*0.5) -- (0, \h) node[anchor=south]{$i_1$};
    \draw[black, sharp corners, fill=white] (-0.3, \h*0.5 ) rectangle (1.3,\h*0.5 + \hrect);
    \draw[black] (0.5, \h*0.5 + \hrect*0.5) node{$\phi_{11}$};
\end{braid} =
\begin{braid}
    \def\h{5};
    \def\ha{3};
    \def\hb{1};
    \def\hc{-2};
    \def\sep{2}
    \def\a{7};
    \def\dashpart{0.8}
    \def\sp{1.5}
    \def\hrect{1.25}
    \def\shiftdown{-3}
    \def\sshiftdown{-7}
    \draw(0, 0) --  (0, \h) node[anchor=south]{$j_1$};
    \draw(1, 0) -- (1, \h) node[anchor=south]{$i_1$};
\end{braid}
.
    \end{equation}

    Next let $a \in \{1, \dots, m-1\}$, and suppose first that $i_a \neq j_a$ and $i_{a+1} \neq j_{a+1}$. Condition (3.b) implies that $i_{a+1} \unrelated j_a$. At the middle of the diagram where the corresponding strings meet, we have
    \begin{equation}
\begin{braid}
    \def\h{5};
    \def\ha{3};
    \def\hb{1};
    \def\hc{-1};
    \def\sep{2}
    \def\a{7};
    \def\dashpart{0.8}
    \def\sp{1.5}
    \def\hrect{1.15}
    \def\shiftdown{-3}
    \def\sshiftdown{-7}

    \draw(2, \hc) -- (0, \hb) -- (0, \ha) -- (0, \h) node[anchor=south]{$i_a$};
    \draw(0, \hc) -- (0, \hc + 1) -- (1, \hb) --(1, \ha) -- (2, \h) node[anchor=south]{$j_a$};
    \draw(3, \hc) -- (3, \hc + 1) --(2, \hb) --(2,\ha) -- (1,\h) node[anchor=south]{$i_{a}^{+1}$};
    \draw(1, \hc) -- (3, \hb) -- (3, \ha) -- (3, \h) node[anchor=south]{$j_{a}^{+1}$};
    \draw[black, sharp corners, fill=white] (-0.3, \ha*0.5 + \hb*0.5 - 0.5*\hrect) rectangle (1.3,\ha*0.5 + \hb*0.5 + \hrect*0.5);
    \draw[black] (0.5, \ha*0.5 + \hb*0.5) node{$\phi_{aa}$};
    \draw[black, sharp corners, fill=white] (1.7, \ha*0.5 + \hb*0.5 - 0.5*\hrect) rectangle (3.3,\ha*0.5 + \hb*0.5 + \hrect*0.5);
    \draw[black] (2.5, \ha*0.5 + \hb*0.5) node{$\phi_{aa}^{+1}$};
\end{braid}
=
\begin{braid}
    \def\h{5};
    \def\ha{3};
    \def\hb{1};
    \def\hc{-1};
    \def\sep{2}
    \def\a{7};
    \def\dashpart{0.8}
    \def\sp{1.5}
    \def\hrect{1.15}
    \def\shiftdown{-3}
    \def\sshiftdown{-7}

    \draw(2, \hc) -- (0, \hb) -- (1, \ha)  -- (2, \h) node[anchor=south]{$j_a$};
    \draw(0, \hc) -- (0, \hc + 1) -- (1, \hb) --(0, \ha) -- (0, \h) node[anchor=south]{$i_a$};
    \draw(3, \hc) -- (3, \hc + 1) --(2, \hb) --(3,\ha) -- (3,\h) node[anchor=south]{$j_{a}^{+1}$};
    \draw(1, \hc) -- (3, \hb) -- (2, \ha) -- (1, \h) node[anchor=south]{$i_{a}^{+1}$};
\end{braid}
=
\begin{braid}
    \def\h{5};
    \def\ha{3};
    \def\hb{1};
    \def\hc{-1};
    \def\sep{2}
    \def\a{7};
    \def\dashpart{0.8}
    \def\sp{1.5}
    \def\hrect{1.15}
    \def\shiftdown{-3}
    \def\sshiftdown{-7}

    \draw(2, \hc) -- (1, \hb) -- (1, \ha)  -- (2, \h) node[anchor=south]{$j_a$};
    \draw(0, \hc) -- (0, \h) node[anchor=south]{$i_a$};
    \draw(3, \hc) -- (3,\h) node[anchor=south]{$j_{a}^{+1}$};
    \draw(1, \hc) -- (2, \hb) -- (2, \ha) -- (1, \h) node[anchor=south]{$i_{a}^{+1}$};
\end{braid}
=
\begin{braid}
    \def\h{5};
    \def\ha{3};
    \def\hb{1};
    \def\hc{-1};
    \def\sep{2}
    \def\a{7};
    \def\dashpart{0.8}
    \def\sp{1.5}
    \def\hrect{1.15}
    \def\shiftdown{-3}
    \def\sshiftdown{-7}

    \draw(2, \hc)  -- (2, \h) node[anchor=south]{$j_a$};
    \draw(0, \hc) -- (0, \h) node[anchor=south]{$i_a$};
    \draw(3, \hc) -- (3,\h) node[anchor=south]{$j_{a}^{+1}$};
    \draw(1, \hc) -- (1, \h) node[anchor=south]{$i_{a}^{+1}$};
\end{braid}.
    \end{equation}
    Here we have written $\phi_{aa}^{+1} \coloneq \phi_{a+1,a+1}$. 
    Next suppose that $i_a = j_a$, with $a \notin \{1, m\}$. Then condition (3.a) implies that $i_{a+1} \unrelated j_{a+1}$ and $i_{a-1} \unrelated j_{a-1}$, which implies that $i_{a+1} = j_{a-1}$. In this case, near the center of the diagram we find
    \begin{equation}
\begin{braid}
    \def\h{5};
    \def\ha{3};
    \def\hb{1};
    \def\hc{-2};
    \def\sep{2}
    \def\a{7};
    \def\dashpart{0.8}
    \def\sp{1.5}
    \def\hrect{1.25}
    \def\shiftdown{-3}
    \def\sshiftdown{-7}

    \draw(0, \hc) -- (0, \hb - 1.5) -- (1, \hb) -- (0, \ha) -- (0, \h) node[anchor=south]{$i_{a}^{-1}$};
    \draw(1, \hc) -- (1, \hb - 2) -- (3, \hb) -- (3, \ha)  -- (4, \h) node[anchor=south]{$j_{a}$};
    \draw(2, \hc) -- (5, \hb) -- (2, \h) node[anchor=south]{$i_{a}^{+1}$};
    \draw(3, \hc) -- (0, \hb) -- (3, \h) node[anchor=south]{$j_{a}^{-1}$};
    \draw(4,\hc) -- (4, \hb - 2) -- (2, \hb) -- (2, \ha)  -- (1, \h) node[anchor=south]{$i_{a}$};
    \draw(5,\hc) -- (5, \hb - 1.5) -- (4, \hb) -- (5, \ha) -- (5, \h) node[anchor=south]{$j_{a}^{+1}$};
\end{braid}
=
\begin{braid}
    \def\h{5};
    \def\ha{3};
    \def\hb{1};
    \def\hc{-2};
    \def\sep{2}
    \def\a{7};
    \def\dashpart{0.8}
    \def\sp{1.5}
    \def\hrect{1.25}
    \def\shiftdown{-3}
    \def\sshiftdown{-7}

    \draw(0, \hc) --  (0, \h) node[anchor=south]{$i_{a}^{-1}$};
    \draw(1, \hc) -- (4, \h) node[anchor=south]{$j_{a}$};
    \draw(2, \hc) -- (4, \h*0.5 + \hc*0.5) -- (2, \h) node[anchor=south]{$i_{a}^{+1}$};
    \draw(3, \hc) -- (1, \h*0.5 + \hc*0.5) -- (3, \h) node[anchor=south]{$j_{a}^{-1}$};
    \draw(4,\hc) -- (1, \h) node[anchor=south]{$i_{a}$};
    \draw(5,\hc) -- (5, \h) node[anchor=south]{$j_{a}^{+1}$};
\end{braid}
= \pm
\begin{braid}
    \def\h{5};
    \def\ha{3};
    \def\hb{1};
    \def\hc{-2};
    \def\sep{2}
    \def\a{7};
    \def\dashpart{0.8}
    \def\sp{1.5}
    \def\hrect{1.25}
    \def\shiftdown{-3}
    \def\sshiftdown{-7}

    \draw(0, \hc) --  (0, \h) node[anchor=south]{$i_{a}^{-1}$};
    \draw(1, \hc) -- (1, \h) node[anchor=south]{$i_{a}$};
    \draw(2, \hc) -- (3, \h) node[anchor=south]{$j_{a}^{-1}$};
    \draw(3, \hc) -- (2, \h) node[anchor=south]{$i_{a}^{+1}$};
    \draw(4,\hc) -- (4, \h) node[anchor=south]{$j_{a}$};
    \draw(5,\hc) -- (5, \h) node[anchor=south]{$j_{a}^{+1}$};
\end{braid}.
    \end{equation}
    Finally, the case where we set $a = m$ is very similar to the case where $a = 1$, so we will skip the details.
    Applying the last few computations successively to our diagram yields
    \begin{equation}
\begin{braid}
    \def\h{5};
    \def\sep{2}
    \def\a{7+\sep};
    \def\dashpart{0.8}
    \def\sp{1.5}
    \def\hrect{1.15}
    \def\shiftdown{-3}
    \def\sshiftdown{-7}
    
    % Left side
    \draw(\sp + 6 + \sep, \sshiftdown) -- (\sp + 2, \shiftdown) --(\sp +2, 0) -- (2, \h) node[anchor=south]{$i_1$};
    \draw(\sp+1, \sshiftdown) -- (\sp+1, \shiftdown - 1) --  (\sp + 3, \shiftdown) --(\sp +3, 0) -- (\a, \h) node[anchor=south]{$j_1$};
    \draw (\sp + 7+\sep, \sshiftdown) -- (\sp + 7 + \sep, \sshiftdown + 1) -- (\sp + 4, \shiftdown) --(\sp +4, 0) -- (3, \h) node[anchor=south]{$i_2$};
    \draw(\sp+2, \sshiftdown) -- (\sp+2, \shiftdown - 2) -- (\sp + 5, \shiftdown) --(\sp +5, 0) -- (\a+1, \h) node[anchor=south]{$j_2$};
    
    % Middle points
    \draw (\sp + 3, \sshiftdown) node {$\cdots$};
    \draw (\sp + 8 + \sep, \sshiftdown) node {$\cdots$};
    \draw (\sp + 5.5+\sep*0.5, \shiftdown + 1) node {$\cdots$};
    \draw (3.5 + \sep/2, \h) node[anchor=south]{$\cdots$};
    \draw (\a + 1.5 + \sep/2, \h) node[anchor=south]{$\cdots$};
    
    % Right side
    \draw  (\sp+9 + \sep, \sshiftdown) -- (\sp+9 + \sep, \shiftdown - 1) -- (\sp + 6+\sep, \shiftdown) --(\sp + 6 + \sep, 0) -- (4  + \sep, \h) node[anchor=south]{$i_m^{-1}$};
    \draw (\sp + 4, \sshiftdown) -- (\sp + 4, \sshiftdown + 1) --  (\sp + 7+\sep, \shiftdown) -- (\sp + 7+\sep, \shiftdown) -- (\sp + 7 + \sep, 0) -- (\a + 2  + \sep, \h) node[anchor=south]{$j_m^{-1}$};
    \draw  (\sp+10 + \sep, \sshiftdown) -- (\sp+10 + \sep, \shiftdown - 1) -- (\sp + 8+\sep, \shiftdown) --(\sp +8 + \sep, 0) -- (5  + \sep, \h) node[anchor=south]{$i_m$};
    \draw (\sp + 5, \sshiftdown) -- (\sp + 9+\sep, \shiftdown) --(\sp + 9 + \sep, 0) -- (\a + 3  + \sep, \h) node[anchor=south]{$j_m$};

    % Rectangles and labels
    \draw[black, sharp corners, fill=white] (1.8 + \sp, \shiftdown + 1) rectangle (3.2 + \sp , \shiftdown + 1 + \hrect);
    \draw[black, sharp corners, fill=white] (3.8 + \sp, \shiftdown + 1) rectangle (5.2 + \sp , \shiftdown + 1 + \hrect);
    \draw[black, sharp corners, fill=white] (5.7 + \sep + \sp, \shiftdown + 1) rectangle (7.3 + \sp + \sep , \shiftdown + 1 + \hrect);
    \draw[black, sharp corners, fill=white] (7.7 + \sep + \sp, \shiftdown + 1) rectangle (9.3 + \sp + \sep , \shiftdown + 1 + \hrect);
    \draw (\sp + 2.5, \shiftdown +1 + 0.5*\hrect) node[black, scale=1]{$\phi_{11}$};
    \draw (\sp + 4.5, \shiftdown +1 + 0.5*\hrect) node[black, scale=1]{$\phi_{22}$};
    \draw (\sp + 6.5+\sep, \shiftdown +1 + 0.5*\hrect) node[black, scale=1]{$\phi_{-11}$};
    \draw (\sp + 8.5+\sep, \shiftdown +1 + 0.5*\hrect) node[black, scale=1]{$\phi_{-00}$};

    % Top bounding rectangles
    \draw[brown] (1.5, \h) rectangle (\a - 1.5 , \h + \hrect);
    \draw[brown] (\a - 0.5, \h) rectangle (2*\a - 1.5 , \h + \hrect);
\end{braid}
= \pm
\begin{braid}
    \def\h{5};
    \def\sep{2}
    \def\a{7+\sep};
    \def\dashpart{0.8}
    \def\sp{1.5}
    \def\hrect{1.15}
    \def\shiftdown{-3}
    \def\sshiftdown{-7}
    
    % Left side
    
    \draw (0, \sshiftdown) -- (0, \h) node[anchor=south]{$i_1$};
    \draw (1, \sshiftdown) -- (1, \h) node[anchor=south]{$i_2$};
    \draw(\sp + 6 + \sep, \sshiftdown) -- (\sp + 2, \shiftdown) --(\sp +2, 0) -- (2, \h) node[anchor=south]{$i_3$};
    \draw(\sp+1, \sshiftdown) -- (\sp+1, \shiftdown - 1) --  (\sp + 3, \shiftdown) --(\sp +3, 0) -- (\a, \h) node[anchor=south]{$j_1$};
    \draw (\sp + 7+\sep, \sshiftdown) -- (\sp + 7 + \sep, \sshiftdown + 1) -- (\sp + 4, \shiftdown) --(\sp +4, 0) -- (3, \h) node[anchor=south]{$i_4$};
    \draw(\sp+2, \sshiftdown) -- (\sp+2, \shiftdown - 2) -- (\sp + 5, \shiftdown) --(\sp +5, 0) -- (\a+1, \h) node[anchor=south]{$j_2$};
    
    % Middle points
    \draw (\sp + 3, \sshiftdown) node {$\cdots$};
    \draw (\sp + 8 + \sep, \sshiftdown) node {$\cdots$};
    \draw (\sp + 5.5+\sep*0.5, \shiftdown + 1) node {$\cdots$};
    \draw (3.5 + \sep/2, \h) node[anchor=south]{$\cdots$};
    \draw (\a + 1.5 + \sep/2, \h) node[anchor=south]{$\cdots$};
    
    % Right side
    \draw  (\sp+9 + \sep, \sshiftdown) -- (\sp+9 + \sep, \shiftdown - 1) -- (\sp + 6+\sep, \shiftdown) --(\sp + 6 + \sep, 0) -- (4  + \sep, \h) node[anchor=south]{$i_m^{-1}$};
    \draw (\sp + 4, \sshiftdown) -- (\sp + 4, \sshiftdown + 1) --  (\sp + 7+\sep, \shiftdown) -- (\sp + 7+\sep, \shiftdown) -- (\sp + 7 + \sep, 0) -- (\a + 2  + \sep, \h) node[anchor=south]{$j_m^{-3}$};
    \draw  (\sp+10 + \sep, \sshiftdown) -- (\sp+10 + \sep, \shiftdown - 1) -- (\sp + 8+\sep, \shiftdown) --(\sp +8 + \sep, 0) -- (5  + \sep, \h) node[anchor=south]{$i_m$};
    \draw (\sp + 5, \sshiftdown) -- (\sp + 9+\sep, \shiftdown) --(\sp + 9 + \sep, 0) -- (\a + 3  + \sep, \h) node[anchor=south]{$j_m^{-2}$};
    \draw (\a + 4 + \sep, \sshiftdown) -- (\a + 4 + \sep, \h) node[anchor=south]{$j_m^{-1}$};;
    \draw (\a + 5 + \sep, \sshiftdown) -- (\a + 5 + \sep, \h) node[anchor=south]{$j_m$};;

    % Rectangles and labels
    \draw[black, sharp corners, fill=white] (1.8 + \sp, \shiftdown + 1) rectangle (3.2 + \sp , \shiftdown + 1 + \hrect);
    \draw[black, sharp corners, fill=white] (3.8 + \sp, \shiftdown + 1) rectangle (5.2 + \sp , \shiftdown + 1 + \hrect);
    \draw[black, sharp corners, fill=white] (5.7 + \sep + \sp, \shiftdown + 1) rectangle (7.3 + \sp + \sep , \shiftdown + 1 + \hrect);
    \draw[black, sharp corners, fill=white] (7.7 + \sep + \sp, \shiftdown + 1) rectangle (9.3 + \sp + \sep , \shiftdown + 1 + \hrect);
    \draw (\sp + 2.5, \shiftdown +1 + 0.5*\hrect) node[black, scale=1]{$\phi_{31}$};
    \draw (\sp + 4.5, \shiftdown +1 + 0.5*\hrect) node[black, scale=1]{$\phi_{42}$};
    \draw (\sp + 6.5+\sep, \shiftdown +1 + 0.5*\hrect) node[black, scale=1]{$\phi_{-13}$};
    \draw (\sp + 8.5+\sep, \shiftdown +1 + 0.5*\hrect) node[black, scale=1]{$\phi_{-02}$};

    % Top bounding rectangles
    \draw[brown] (-0.5, \h) rectangle (\a - 1.5 , \h + \hrect);
    \draw[brown] (\a - 0.5, \h) rectangle (2*\a + 0.5 , \h + \hrect);
\end{braid}.
    \end{equation}
    Resolving this diagram corresponds exactly to solving a version of \cref{E:YActionV1V2} for $m' = m - 2$, and it is easy to see that the conditions above that equation also hold for the new diagram, so that, by induction, \cref{E:YActionV1V2} holds. This completes the proof of the lemma.
\end{proof}

The results of  this section come together nicely to give a complete description (up to sign) of the type C $y$-action on a Specht module with a short leg. This is the main theorem of this section.

\begin{thm} \label{T:YAction}
    Let $k \in \{1, \dots, n\}$, $\tabt = \phi_{\resi}^\lambda(l_1, \dots, l_d) \in \Std(\lambda,\resi)$. 
    \begin{enumerate}
        \item If $\deg(y_k e(\resi(\tabt))) = 2$, then $v_\tabt y_k = \sum_{\tabt' \in \Std(\lambda,\resi)} \pm v_{\tabt'}$, where the sum is taken over $\tabt'$ such that $\tabt' = \phi_{\resi}^\lambda(l_1', \dots, l_d')$, with $l_s = l_s'$ for $s \in \{1, \dots, r-1, r+1, \dots, d\}$, while $l_r = L$, $l_r' = A$, $k \in \bm{s}^{(r)}$ and $r \in W_2$.
        \item If $\deg(y_k e(\resi(\tabt))) = 4$, then $v_\tabt y_k = \pm v_{\tabt'}$, with $\tabt' = \phi_{\resi}^\lambda(l_1', \dots, l_d')$, with $l_s = l_s'$ for all $s \in \{1, \dots, r-1, r+1, \dots, d\}$, while $l_r = L$, $l_r' = A$, $k \in \bm{s}^{(r)}$ and $r \in W_4$.
    \end{enumerate}
\end{thm}

\begin{eg} \label{eg:YActionExample}
    We continue \cref{eg:binaryword}. Let us evaluate the $y$-action on the basis vector $v_{\phi(L,L)}$ according to \cref{T:YAction}:
    \begin{itemize}
        \item For $k \in \{2, 7\}$, we have $k \in \bm{s}^{(1)}$ but $k \notin \bm{s}^{(2)}$. The action flips $l_1$ from $L$ to $A$, yielding exactly one term: 
        \[ v_{\phi(L,L)} y_2 = v_{\phi(L,L)} y_7 = \pm v_{\phi(A,L)}. \]
        \item For $k \in \{4, 5\}$, we have $k \in \bm{s}^{(1)} \cap \bm{s}^{(2)}$. Summing over $r \in \{1, 2\}$ flips $l_1 \to A$ and $l_2 \to A$, yielding a linear combination of two terms:
        \[ v_{\phi(L,L)} y_4 = \pm v_{\phi(A,L)} \pm v_{\phi(L,A)}, \]
        \[ v_{\phi(L,L)} y_5 = \pm v_{\phi(A,L)} \pm v_{\phi(L,A)}. \]
        We make no claims regarding the signs.
        \item For $k \in \{3, 6\}$, we have $\deg(y_k e(\resi(\tabt))) = 4$ and no subsequences in $W_4$, so that $v_{\phi(L,L)} y_3 = v_{\phi(L,L)} y_6 = 0$.
    \end{itemize}
\end{eg}

The following corollary is the type C analogue of \cite[Corollary 4.4]{loub17}. 
\begin{cor} \label{Cor:YAction}
    Suppose that $e(\resi) S^\lambda \neq 0$. Then 
    $$
    \{v \in e(\resi)S^\lambda \mid v y_k = 0 \text{ for all } k \in \{1, \dots, n\}\} = \mathcal{O}v_{\tabt_{\mathrm{max}}(\resi)}.
    $$
\end{cor}

\begin{proof}
    Let $v = \sum_\tabt a_\tabt v_\tabt \in e(\resi) S^\lambda$ be such that $v y_k = 0$ for all $k \in \{1, \dots, n\}$. Let $\tabt^* = \phi_{\resi}^\lambda(l_1, \dots, l_d) \in \Std(\lambda,\resi)$ be minimal in dominance order among those tableaux $\tabt$ with $a_\tabt \neq 0$. Let $r_1, r_2, \dots, r_A \in \{1, \dots, d\}$ be those integers such that $l_{r_a} = L$. For each $a \in \{1, \dots, A\}$, let $k_a \in \bm{s}^{(r_a)}$ be such that $\deg(y_{k_a} e(\resi(\tabt))) = 2$ whenever $r_a \in W_2$ and $\deg(y_{k_a} e(\resi(\tabt))) = 4$ if $r_a \in W_4$. Moreover, if $\bm{s}^{(r_{a})} \subset \bm{s}^{(r_b)}$ for some $a, b \in \{1, \dots, A\}$, then we further require that $k_b \notin \bm{s}^{(r_a)}$. According to \cref{T:YAction}, we have
    \begin{equation}
        v y_{k_1} y_{k_2} \dots y_{k_A} = \pm v_{\tabt_{\mathrm{max}}(\resi)} \neq 0. 
    \end{equation}
    It follows that the product of $y$'s must be empty and we have $v = a_{\tabt_{\mathrm{max}}(\resi)} v_{\tabt_{\mathrm{max}}(\resi)}$, as we claimed.
\end{proof}
It is particularly striking that the above result holds regardless of the characteristic of the ring $\mathcal{O}$.

\subsection{The \texorpdfstring{$\psi$}{psi}-action} \label{SS:PsiAction}

In this section, we determine the right action of the quiver Hecke algebra generators $\psi_r$ ($r \in \{1, \dots, n-1\}$) on the Specht module basis given by $\{ v_\tabt \mid \tabt \in \Std(\lambda) \}$. The computation of $v_\tabt \psi_r$ relies heavily on the relative positions of $r$ and $r+1$ in $\tabt$, the standardness of the swapped tableau $\tabt s_r$, and the diagrammatic relations of the quiver Hecke algebra.

Our analysis naturally splits into distinct cases.

\begin{thm} \label{T:PsiAction}
    Let $\tabt \in \Std(\lambda)$ and $r \in \{1, \dots, n-1\}$. Let $i_r = \res(\tabt^{-1}(r))$ and $i_{r+1} = \res(\tabt^{-1}(r+1))$.
    \begin{enumerate}
        \item If $\tabt^{-1}(r) \in \Arm(\lambda)$ and $\tabt^{-1}(r+1) \in \Leg(\lambda)$, then
        \begin{equation}
            v_\tabt \psi_r = v_{\tabt s_r}.
        \end{equation}
        \item If $\tabt^{-1}(r) \in \Leg(\lambda)$, $\tabt^{-1}(r+1) \in \Arm(\lambda)$, then
        \begin{equation}
            v_\tabt \psi_r = v_{\tabt s_r} \mathcal{Q}_{i_r, i_{r+1}}(y_r, y_{r+1}).
        \end{equation}
        
        \item Suppose $\tabt^{-1}(r)$ and $\tabt^{-1}(r+1)$ lie in the same component of $\lambda$ (i.e., both in $\Leg(\lambda)$ or both in $\Arm(\lambda)$), so that $\tabt s_r \notin \Std(\lambda)$. Let $k = r+1, k' = r$ if the nodes are in $\Leg(\lambda)$, and $k = r, k' = r+1$ if they are in $\Arm(\lambda)$. 
        \begin{enumerate}
            \item Suppose that there exists some $m \in \{1,\dots,d\}$ such that $k \in \bm{s}^{(m)}$ and $k \notin \bm{s}^{(m)}$, and furthermore $l_m = L$. Let $\tabt'$ be the tableau having the same residue sequence as $\tabt$ and the same binary word, with the exception that $l_m = A$. The tableau $\tabt' s_r$ is standard.
            \begin{enumerate}
                \item  If $i_{k'} \notin \{0, e\}$, then
                    \begin{equation}
                        v_\tabt \psi_r = \pm v_{\tabt' s_r}.
                    \end{equation}
                \item Suppose instead that $i_{k'} \in \{0, e\}$. Let $s$ be the string of residue $i_k$ crossing string $k+1$ lowest down in the string diagram of $v_\tabt$. Then
                \begin{equation}
                    v_\tabt \psi_r = \pm v_{\tabt' s_r} (y_{k'} + y_s).
                \end{equation}
            \end{enumerate}
            \item Otherwise,
            \begin{equation}
                v_\tabt \psi_r = 0.
            \end{equation}
        \end{enumerate}
    \end{enumerate}
\end{thm}

\begin{proof}
    Case 1 is immediate. Since $r$ and $r+1$ lie in different components of the hook $\lambda$, the tableau $\tabt s_r$ where the two numbers are swapped remains standard.
    
    Case 2 is also clear. It can be evaluated further by using the description of the $y$-action in the previous section.

    In Case 3, $\tabt s_r \notin \Std(\lambda)$ because $r$ and $r+1$ are adjacent in the same row or column. We will only cover the case where both nodes are in $\Leg(\lambda)$, as the other case is similar. When we multiply $v_\tabt$ by $\psi_r$, we add a crossing between strings $r$ and $r+1$ at the bottom of the diagram. The $\psi_r$ crossing can slide upwards past the crossings of string $r$. Note that string $r$ has residue $i_{r+1}$, since $\psi_r$ exchanges the residues of strings $r$ and $r+1$. If $\psi_r$ reaches the top of the diagram without breaking any $(i_{r+1},i_{r+1})$-crossings, it evaluates to zero by the Specht relations, since $\tabt s_r \notin \Std(\lambda)$. 

    In order for the slide to yield a non-zero result, it must trigger the error term of the braid relations \eqref{E:braidC}. This occurs exactly when the $\psi_r$-crossing slides past a string $s$ such that the residues form the sequence $(i_{r+1}, i_r, i_{r+1})$. Diagrammatically, this condition matches exactly the condition for a dot $y_{r+1}$ to break a crossing during the $y$-action on $v_\tabt$. Consequently, there is a natural bijection between the broken diagrams produced by the $\psi_r$-slide and those produced by the $y_{r+1}$-slide. For this reason, we obtain $0$ if $l_m = A$.

    The key difference is that $\psi_r$ exchanges strings $r$ and $r+1$. In the $y$-action evaluation $v_\tabt y_{r+1} = \sum c_{\tabt'} v_{\tabt'}$, the string $r+1$ is moved from the leg to the arm in $\tabt'$. The $\psi_r$-action applies the same break with the difference that strings $r$ and $r+1$ now reach the arm and leg respectively, and that any term where both $r$ and $r+1$ are in the same double short residue sequence $\bm{s}^{(m)}$ gives zero, as we see next.
    
    As the $\psi_r$ crossing between strings $r$ and $r+1$ (with residues $i_r$ and $i_{r+1}$) slides upwards, it must pass through the other pair of strings in the double short residue sequence $\bm{s}^{(m)}$ with residues $(i_r,i_{r+1})$.

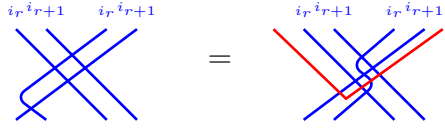
\begin{equation*}
    \begin{braid}
        \def\h{3}
        % Left strings (going right)
        \draw (0,\h) node[anchor=south]{$i_r$} -- (3,0);
        \draw (1,\h) node[anchor=south]{$i_{r+1}$} -- (4,0);
        
        % Right strings (going left, crossing at the bottom)
        \draw (3,\h) node[anchor=south]{$i_r$} -- (0,0.75) -- (1,0);
        \draw (4,\h) node[anchor=south]{$i_{r+1}$} -- (1,0.75) -- (0,0);
    \end{braid}
    \quad = \quad
    \begin{braid}
        \def\h{3}
        % Left strings (going right)
        \draw (0,\h) node[anchor=south]{$i_r$} -- (3,0);
        \draw (1,\h) node[anchor=south]{$i_{r+1}$} -- (4,0);
        
        % Right strings (crossing slid up)
        \draw (3,\h) node[anchor=south]{$i_r$} -- (1.6,1.8) -- (2.4,1.2) -- (1,0);
        \draw (4,\h) node[anchor=south]{$i_{r+1}$} -- (0,0);

        \draw[red,sharp corners] (-1,\h) -- (1.4,0.7) -- (4.6,\h);
    \end{braid}
\end{equation*}
If we look only at the part of the diagram above the red line, we obtain
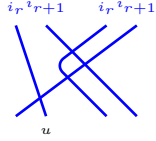
\begin{equation*}
    \begin{braid}
        \def\h{3}
        % Left strings (going right)
        \draw (0,\h) node[anchor=south]{$i_r$} -- (1,0) node[anchor=north, black]{$u$};
        \draw (1,\h) node[anchor=south]{$i_{r+1}$} -- (4,0);
        
        % Right string i_{r+1} (straight, isolating the i_{r+1}, i_{r+1} crossing)
        \draw (4,\h) node[anchor=south]{$i_{r+1}$} -- (0,0);
        
        % Right string i_r curved out of the way to clarify the diagram
        \draw (3,\h) node[anchor=south]{$i_r$} -- (1.3,1.7) -- (3,0);
    \end{braid}
\end{equation*}

Since $u$ is immobile, the error term that resolves the $(i_{r+1}, i_{r+1})$-crossing evaluates to zero. Therefore, we consider only the case where $r+1 \in \bm{s}^{(m)}, r \notin \bm{s}^{(m)}$ for some $m$, as in the statement. Note that in this case $r$ is at the start of such a sequence, and that furthermore there can only be one such sequence. In conclusion, if $i_r \notin \{0,e\}$, then
\[
v_\tabt \psi_r = \pm v_{\tabt' s_r}.
\]
Otherwise, if $i_r \in \{0,e\}$, then the braid relation produces extra dots.
\[
v_\tabt \psi_r = \pm v_{\tabt' s_r} (y_r + y_s).
\]
The action of the dots can be evaluated using the results in the previous section.
\end{proof}

\begin{eg} \label{eg:PsiActionExample}
    We continue \cref{eg:YActionExample}. Consider $\psi_3$ acting on $\phi(L,L)$. The integers $3$ and $4$ are adjacent in the leg. Here $k=4$ and $k'=3$. Notice that $4 \in \bm{s}^{(2)}$ but $3 \notin \bm{s}^{(2)}$, so $m=2$ and $l_2 = L$. Because $i_3 = 2 = e$, this triggers the $(y_{k'} + y_s) = (y_3 + y_5)$ relation. Flipping $l_2 \to A$ gives $\tabt' = \phi(L,A)$. The action results in $v_{\phi(L,L)} \psi_3 = \pm v_{\tabt' s_3}(y_3 + y_5)$, producing the sum of two standard tableaux:

    \[
    v_{\phi(L,L)} \psi_3 = \pm \;
    \ColorTableau[]{{1,2,6,7},{3},{4},{5}}
    \pm 
    \ColorTableau[]{{1,3,4,5},{2},{6},{7}}
    \]
    The effect of applying $\psi_5$ is similar:
    \[
    v_{\phi(L,L)} \psi_5 = \pm \;
    \ColorTableau[]{{1,2,3,7},{4},{5},{6}}
    \pm 
    \ColorTableau[]{{1,4,5,6},{2},{3},{7}}
    \]
\end{eg}

Our results on the action of the quiver Hecke algebra generating elements on Specht modules labelled by a partition having a short arm or a short leg are somewhat less wieldy than the corresponding lemmas in \cite{loub17}, as is to be expected from type C computations. They can nevertheless be used to calculate the action (up to sign) on a basis element, and they will be useful as we begin our description of the homomorphisms into the Specht modules labelled by partitions of this kind.

\section{Homomorphisms when \texorpdfstring{$\lambda$}{lambda} has a Short Leg}  \label{S:ShortLeg} \noindent Let $\la=(a+1,1^b)$ with $b < 2e$ (short leg), and fix $\mu=(\mu_1,\dots,\mu_m)$ a partition of $n$. We say that a node $N \in [\mu]$ is an {\em arm node} (resp.\ {\em leg node}) with respect to $\tabt \in \Std(\lambda)$ if $\tabt^{-1}(\tabt^\mu(N)) \in \Arm(\lambda)$ (resp.\ in $\Leg(\lambda)$). If we refer to arm or leg nodes of $\mu$ without designating a specific tableau $\tabt \in \Std(\lambda)$, then $\tabt = \tabt_{\mathrm{max}}(\resi^\mu)$ is to be assumed.

\begin{lem}\label{L:firstcol}
    If there exists a non-zero $\varphi \in \Hom(S^\mu, S^\la)$, then $v \coloneq \varphi(v_{\tabt^\mu})$ satisfies $v\in \mathcal{O}v_{\tabt_{\mathrm{max}}(\resi^\mu)}$, and moreover, $\mu$ can be split into a partition $[\mu_L] \subset [\mu]$ and a skew partition $[\mu_A] = [\mu] \setminus [\mu_L]$ so that $[\mu_L]$ contains all leg nodes of $[\mu]$ (and also the node $(1,1)$), and $[\mu_A]$ contains all arm nodes of $[\mu]$.
\end{lem}

\begin{proof}
    By \cref{Cor:YAction}, we see that $v$ satisfies \cref{E:residuerelation} and the $y$-relation in \cref{E:specht} if and only if $v=v_{\tabt_{\mathrm{max}}(\resi^\mu)}$. 
    
    Next, let $N = (c,d)$ be an arm node, and suppose that $N' = (c, d+1)$ is a leg node, with $r \coloneq (\tabt^{\mu})^{-1}(N)$ and $r+1 =(\tabt^{\mu})^{-1}(N')$. Then $v_{\tabt_{\mathrm{max}}(\resi^\mu)} \psi_r = v_{\tabt_{\mathrm{max}}(\resi^\mu) \sigma_r} \neq 0$. It follows that each row $[\mu_j]$ can be split into a set $[\mu_j]_L$ composed of leg nodes, followed by a set $[\mu_j]_A$ composed of arm nodes.

    Next, let $N=(a,b)$ be an arm node, and suppose that $N' = (a+1,b)$ is a leg node. Due to the previous paragraph, each node in row $a$ to the east of $N$ is an arm node, and each node in row $b$ to the west of $N'$ is a leg node. Therefore, the strings reaching the Garnir belt of $N$ do not cross each other in $v_{\tabt_{\mathrm{max}}(\resi^\mu)}$. The Garnir relation in $S^\mu$ involves a linear combination of the terms in $\Gar^N$, which is taken to be zero. However, since the strings in the Garnir belt of $N$ do not cross each other in $v_{\tabt_{\mathrm{max}}(\resi^\mu)}$, each of the tableaux $\tabt \in \Gar^N$ is mapped to a basis element $v_{\tabt_{\mathrm{max}}(\resi^\mu)} \psi_\tabt = v_{\tabt_{\mathrm{max}}(\resi^\mu) \sigma_\tabt}$, since $\tabt_{\mathrm{max}}(\resi^\mu) \sigma_\tabt$ is standard. A linear combination of basis elements cannot yield zero, so that there cannot be a homomorphism sending $v_{\tabt^\mu}$ to $v_{\tabt_{\mathrm{max}}(\resi^\mu)}$, which contradicts our assumptions. 
    
    It follows that each column $[\mu_j']$ can be split into a set $[\mu_j']_L$ composed of leg nodes, followed by a set $[\mu_j']_A$ composed of arm nodes.

    The lemma statement follows from the conclusions of the last two paragraphs.
\end{proof}

Without loss of generality, we may assume that $v \coloneq \varphi(v_{\tabt^\mu}) = v_{\tabt_{\mathrm{max}}(\resi^\mu)}$. (That is, the constant $a$ equals $1$ in $\varphi(v_{\tabt^\mu}) = v_{\tabt_{\mathrm{max}}(\resi^\mu)}$.) We consider two separate cases, according to whether the node $(1,2)$ of $[\mu]$ is an arm node or a leg node. The case where $(1,2) \notin [\mu]$ fits without issue inside the case where $(1,2)$ is an arm node of $[\mu]$, so we blend the two together into the case where $(1,2)$ is \emph{not a leg node} of $[\mu]$.

\subsection{The node \texorpdfstring{$(1,2)$}{(1,2)} is not a leg node of \texorpdfstring{$[\mu]$}{[mu]}}

\begin{lem}\label{L:legnodes}
    Suppose that $(1,2)$ is not a leg node of $[\mu]$. Then the leg nodes of $\mu$ are precisely $(2,1)$, $(3,1)$, $\dots$, $(b+1,1)$.
\end{lem}

\begin{proof}
    This is a trivial consequence of \cref{L:firstcol}.
\end{proof}

For a partition $\nu \vdash b + 1$ such that $[\nu] \subseteq [\mu]$, we define $\tabt_\mu^\lambda(\nu) \in \Std(\lambda)$ to be the unique standard tableau obtained as follows:
\begin{enumerate}
    \item For each node $N \in [\mu] \setminus [\nu]$, place the number $\tabt^{\mu}(N)$ into $\Arm(\lambda)$.
    \item For each node $N \in [\nu]$, place the number $\tabt^{\mu}(N)$ into $(1,1) \cup \Leg(\lambda)$.
\end{enumerate} 
According to \cref{L:legnodes}, if there is a homomorphism $\varphi$ from $S^\mu$ to $S^\lambda$ where $(1,2)$ is not a leg node of $[\mu]$ and $\lambda$ has a short leg, then (up to a constant) $\varphi(v_{\tabt^\mu}) = v_{\tabt_\mu^\lambda(\nu)}$, where $\nu =(1^{b+1})$. Let us write simply $\tabt_\mu^\lambda$ if $\nu = (1^{b+1})$. On the other hand, we also have (up to a constant) $\varphi(v_{\tabt^\mu}) = v_{\tabt_{\mathrm{max}}(\resi^\mu)}$ due to \cref{L:firstcol}. Next we give a condition under which ${\tabt_{\mathrm{max}}(\resi^\mu)}$ and $\tabt_\mu^\lambda$ match. 

\begin{lem} \label{L:y_rels_hom}
    Let $\mu$ be a partition such that $\resi^{\tabt^\lambda_\mu} = \resi^\mu$. Then $\tabt_{\mathrm{max}}(\resi^\mu)$ exists and 
    \[\tabt_\mu^\lambda = \tabt_{\mathrm{max}}(\resi^\mu).\]
\end{lem}
\begin{proof} 
    We show that $\tabt_\mu^\lambda = \tabt_{\mathrm{max}}(\resi^\mu) = \phi_{\resi^\mu}^\lambda(A, \dots, A)$, which by \cref{C:minimal_tableau} immediately implies $v_{\tabt_\mu^\lambda} y_k = 0$ for all $k$.
    Suppose, towards a contradiction, that $\tabt_\mu^\lambda = \phi_{\resi^\mu}^\lambda(l_1, \dots, l_d)$ corresponds to a word with $l_m = L$ for some $m \in \{1, \dots, d\}$. Let $\bm{s}^{(m)} \in DS(\resi^\mu)$ be the associated interval of indices.
    
    The choice $L$ means that the integer $u = s^{(m)}_1$ is placed in $\Leg(\lambda)$. By \cref{L:legnodes}, the node $N \coloneq (\tabt^{\mu})^{-1}(u)$ lies in the first column of $[\mu]$, so $N = (r,1)$ for some $r \ge 2$.
    Because $\bm{s}^{(m)}$ defines a double short residue sequence, the contiguous subset of integers $\bm{s}^{(m)}$ must be partitioned equally into an Arm strip $S_A$ and a Leg strip $S_L$ within $\lambda$ that generate identical residue sequences.
    The nodes of $[\mu]$ filled by the integers in $\bm{s}^{(m)}$ consist of consecutive nodes starting in row $r$.
    
    The leg nodes $S_L$ are located in the first column of $\mu$, while the arm nodes $S_A$ are located in arbitrary columns. The first arm node in $(\tabt^{\mu})^{-1}(\bm{s}^{(m)})$ is either in column $1$ or $2$.
    
    If it is in the second column, then it must be $N' = (r, 2)$, which is horizontally adjacent to $N = (r, 1)$. For $S_A$ and $S_L$ to generate identical residue sequences, their starting nodes must have the same residue, but horizontally adjacent nodes in $\mu$ have contents that differ by exactly 1, so their residues are adjacent in the type C quiver and can never be equal. 

    Assume instead that $N'$ is in the first column, so that $N' = (r', 1)$. In that case all of $(\tabt^{\mu})^{-1}(\bm{s}^{(m)})$ is in the first column of $[\mu]$, and we have a double short residue sequence which is a subsequence of $\mathcal{I}$. This is also impossible.
    
    This contradiction establishes that $\tabt_\mu^\lambda$ is constructed using only $A$ choices, and therefore $\tabt_\mu^\lambda = \tabt_{\mathrm{max}}(\resi^\mu)$.
\end{proof}

We set up some additional simple notation. Given $c_1 \in \mathbb{Z} \setminus \{ce \mid c \in \mathbb{Z}\}$, we may write $\overline{c_1}^+$ for $\overline{c_1}$ if $\overline{c_1 + 1} = \overline{c_1} + 1$
Similarly, we may write $\overline{c_1}^-$ for $\overline{c_1}$ if $\overline{c_1+1}=\overline{c_1}-1$. The following is easy to check.

\begin{lem}\label{lem:residues}
    Let $c_1, c_2 \in \mathbb{Z} \setminus \{ce \mid c \in \mathbb{Z}\}$. Then
    \begin{itemize}
        \item[(1)] $\overline{c_1}^\pm=\overline{c_2}^\pm$ if and only if $c_2\equiv +c_1 \mod 2e$, and
        \item[(2)] $\overline{c_1}^\pm=\overline{c_2}^\mp$ if and only if $c_2\equiv -c_1 \mod 2e$.
    \end{itemize}
\end{lem}

The next lemma determines the shapes $\mu$ for which $\resi^{\tabt_\mu^\lambda} = \resi^\mu$, which is a necessary condition for the existence of a homomorphism $\varphi \in \Hom_{R_n}(S^\mu,S^\lambda)$, as we have seen. The lemma is separated into a recursive and a non-recursive part.

\begin{lem}\label{lem:shapes}
    Let $\mu=(\mu_1,\dots,\mu_m)$ be a partition of $n$ such that $m\ge b+1$ with charge $\kappa$ and residue sequence $(i_1,i_2,\dots,i_n)=\resi^\mu$. Recall that $\lambda = (a+1, 1^b)$.
    \begin{enumerate}
        \item \begin{itemize} 
            \item[(a)] Suppose that $b > 0$ and $\mu_1 = c(2e)$ for some $c \in \mathbb{Z}_{>0}$. Then $\mu = (\mu_1, \overline{\mu})$ satisfies $\resi^{\tabt_\mu^\lambda} =\resi^\mu$ if and only if $\overline{\mu}$ satisfies $\resi^{\tabt_{\overline{\mu}}^{\overline{\lambda}}} =\resi^{\overline{\mu}}$, where $\overline{\lambda} = (a + 1 - (\mu_1 - 1), 1^{b-1})$ with charge $\kappa - 1$. 
            \item[(b)] Suppose that $b = 0$ and $\mu_1 = c(2e) - 1$ for some $c \in \mathbb{Z}_{>0}$. Then $\mu = (\mu_1, \overline{\mu})$ satisfies $\resi^{\tabt_\mu^\lambda} =\resi^\mu$ if and only if $\overline{\mu}$ satisfies $\resi^{\tabt_{\overline{\mu}}^{\overline{\lambda}}} =\resi^{\overline{\mu}}$, where $\overline{\lambda} = (a + 1 - \mu_1)$ with charge $\kappa - 1$. 
            \end{itemize} \label{itemlist:1}
        \item Suppose that $\mu_1$ is not as in the previous cases.
        Then $\resi^{\tabt_\mu^\lambda} =\resi^\mu$ if and only if $\mu$ has one of the following shapes.
        \begin{enumerate}
            \item $\mu= \lambda = (a+1, 1^b)$
            \item $\mu=(a,1^{b+1})$ and $a \equiv -b-1 \mod 2e$
            \item $\mu=(\mu_1,1^{n-\mu_1})$ and $b - \mu_1 \equiv 2\kappa - 1 \mod 2e$.
            \item $\mu = (c(2e) - 2\kappa, 2^y, 1^{b-y})$ for some $c \in \mathbb{Z}_{\ge 0}$ and $y \in \{1, 2, \dots, b\}$.
            \item $\mu = (c(2e) - 2\kappa, 2^y, 1^{b-y + 1})$ for some $c \in \mathbb{Z}_{\ge 0}$, $b - y \equiv 2\kappa - 1 \mod 2e$ and $y \in \{1, 2,\dots,b\}.$
        \end{enumerate} \label{itemlist:2}
    \end{enumerate}
\end{lem}

\begin{proof}
    Part \ref{itemlist:1} is a straightforward consequence of the fact that the set of contents $\{\kappa, \kappa+1, \dots, \kappa+c(2e)-1\}$ maps to exactly $c$ full cycles of the residues $\{0, 1, \dots, e, e-1, \dots, 1\}$. For instance, in (a), the residues of the nodes in the second row starting at node $(2,2)$ fit also in the first row starting at $(1, c(2e) + 1)$, and it follows that, if $\resi^{\tabt_{\overline{\mu}}^{\overline{\lambda}}} =\resi^{\overline{\mu}}$ (the definitions of $\overline{\lambda}$ and $\overline{\mu}$ are given in the statement), then a row with $c(2e)$ nodes can be added above both $\overline{\lambda}$ and $\overline{\mu}$, and further the nodes $(1,2), (1,3), \dots, (1,\overline{\lambda}_1)$ can be translated to the first row that was just added, to obtain $\lambda$ and $\mu$, while preserving the correspondence of nodes $N \leftrightarrow (\tabt^\mu)^{-1}(\tabt_{\overline{\mu}}^{\overline{\lambda}}(N))$ between $\overline{\lambda}$ and $\overline{\mu}$ in the new shapes $\lambda$ and $\mu$, so that $\resi^{\tabt_\mu^\lambda} =\resi^\mu$ also. The converse and the argument for (b) are both similar, and we skip the details.

    For Part \ref{itemlist:2}, let the ordered nodes of $\mu \setminus \lambda$ be $(M_1, \dots, M_p)$ and those of $\lambda \setminus \mu$ be $(L_1, \dots, L_p)$. Suppose first that $\mu$ is a hook partition. If $\mu = \lambda$, the result holds trivially. Otherwise, it is necessary that $\overline{\cont{M_1}} = \overline{\cont{L_1}}$. In the case where $\overline{\cont{M_1}}^\pm = \overline{\cont{L_1}}^\pm$, it is easy to see that there cannot exist nodes $M_2, L_2$, as otherwise 
    \[\overline{\cont{M_2}} = \overline{\cont{M_1} - 1} \neq \overline{\cont{L_1} + 1} = \overline{\cont{L_2}}.\] The condition on the residues of $M_1$ and $L_1$ reduces to $a \equiv -b-1 \mod 2e$. On the other hand, if $\overline{\cont{M_1}}^\pm = \overline{\cont{L_1}}^\mp$, then there may exist several nodes $(M_1, \dots, M_p)$ and $(L_1, \dots, L_p)$, and the condition on the residues reduces to $b-\mu_1 \equiv 2\kappa - 1 \mod 2e$.

    Next, suppose instead that $\mu$ is not a hook partition, so that $(2,2) = M_1 \in [\mu]$. Since $\mu_1$ is not a multiple of $2e$, we cannot have $\overline{\cont(M_1)}^{\pm} = \overline{\cont(L_1)}^{\pm}$, so it follows that $\overline{\cont(M_1)}^{\pm} = \overline{\cont(L_1)}^{\mp}$, and therefore, the residue of $(2, 3)$ is different from $\res(L_2)$, if such a node exists. It follows that $(2,3) \notin [\mu]$, which shows that $[\mu] \setminus [\mu_1]$ has exactly two columns. Moreover, the condition $\overline{\cont(M_1)}^{\pm} = \overline{\cont(L_1)}^{\mp}$ reduces to $\mu_1 \equiv -2\kappa \mod 2e$. The case (2.d) is now covered.
    
    Under the conditions in (2.d), let $M_y$ be the last node in the second column of $\mu$. If there is a node $M_{y+1}$ below $(b+1,1)$ in $[\mu]$, then we must have $\overline{\cont{M_{y+1}}} = \overline{\cont{M_y} - 1}$. The condition that the leg of $\lambda$ is short negates the possibility that $\overline{\cont{M_{y+1}}}^{\pm} = \overline{\cont{M_y} - 1}^\pm$, so it follows that $\overline{\cont{M_{y+1}}}^\pm = \overline{\cont{M_y} - 1}^\mp$, and this condition reduces to $b - y = 2\kappa - 1 \mod 2e$. We have exhausted all possibilities for matching residues.
\end{proof}

\begin{eg} \label{eg:leg_residues}
    Let $\lambda = (13, 1^3), \kappa = 0, e = 2$. We give a list of all partitions $\mu$ satisfying $\resi^{\tabt_\mu^\lambda} =\resi^\mu$.

    \begin{itemize}
        \item $\mu = (13, 1^3)$, corresponding to \emph{(2.a)}.
        \item Set $\mu_1 = 12$, according to \emph{(1.a)}. Then we have $\lambda' = (2,1,1), \kappa' = -1$.
            \begin{itemize}
                \item $\mu' = \lambda' = (2,1,1)$, corresponding to \emph{(2.a)}, gives $\mu = (12,2,1,1)$.
                \item $\mu' = (1^4)$, corresponding to \emph{(2.b)}, gives $\mu = (12,1^4)$.
            \end{itemize}
        \item Set $\mu_1 = 8$, according to \emph{(1.a)}. Then we have $\lambda' = (6,1,1), \kappa' = -1$.
            \begin{itemize}
                \item $\mu' = \lambda' = (6,1,1)$, corresponding to \emph{(2.a)}, gives $\mu = (8,6,1,1)$.
                \item $\mu' = (5,1,1,1)$, corresponding to \emph{(2.b)}, gives $\mu = (8,5,1,1,1)$.
                \item Set $\mu_2 = 4$, according to \emph{(1.a)}. Then we have $\lambda'' = (3,1), \kappa = -2$.
                    \begin{itemize}
                        \item $\mu'' = \lambda'' = (3,1)$, according to \emph{(2.a)}, gives $\mu = (8,4,3,1)$.
                        \item $\mu'' = (2,1,1)$, according to \emph{(2.b)}, gives $\mu = (8,4,2,1,1)$.
                    \end{itemize}
                \item $\mu' = (1^8)$, corresponding to \emph{(2.c)}, gives $\mu = (8,1^8)$.
            \end{itemize}
        \item Set $\mu_1 = 4$, according to \emph{(1.a)}. Then we have $\lambda' = (10,1,1), \kappa' = -1$. The restriction that $\mu_2 \leq \mu_1 = 4$ is now relevant.
            \begin{itemize}
                \item Set $\mu_2 = 4$, according to \emph{(1.a)}. Then we have $\lambda'' = (7,1), \kappa'' = -2$. Further iterations lead to the partitions:
                    \begin{itemize}
                        \item  $\mu = (4,4,4,4)$
                        \item  $\mu = (4,4,4,3,1)$
                        \item  $\mu = (4,4,2,1^6)$
                    \end{itemize}
                \item $\mu' = (1^{12})$, corresponding to \emph{(2.c)}, gives $\mu = (4,1^{12})$.
            \end{itemize}
        \end{itemize}
\end{eg}

\cref{lem:shapes} gives a collection of partitions $\mu$ for which the residue relation in \cref{E:residuerelation} and the $y$-relation in \cref{E:specht} allow for the existence of a homomorphism from $S^\mu$ to $S^\lambda$. It is left to check the $\psi$- and Garnir relations.

\begin{lem} \label{L:simple_spechts_hom}
    Let $\mu$ be a partition such that $\resi^{\tabt^\lambda_\mu} = \resi^\mu$. Then $v_{\tabt_\mu^\lambda}$ satisfies the $\psi$-relations of the Specht generator $v_{\tabt^\mu}$ given in \cref{E:specht}.
\end{lem}

\begin{proof}
    Recall from \cref{L:y_rels_hom} that $\tabt_\mu^\lambda = \tabt_{\mathrm{max}}(\resi^\mu)$. For any two horizontally adjacent nodes $N, N' \in [\mu]$ filled with consecutive integers $k$ and $k+1$ in $\tabt^\mu$, we must show that $v_{\tabt_\mu^\lambda} \psi_k = 0$.
    We split into two cases based on the placement of $N$ and $N'$ in $\mu$. 

    \begin{enumerate}
        \item \textbf{Both $N$ and $N'$ are arm nodes of $[\mu]$:} In this case, $\tabt_\mu^\lambda$ places both $k$ and $k+1$ in $\Arm(\lambda)$.
        We apply \cref{T:PsiAction} (Case 3). Since we have established that $\tabt_\mu^\lambda = \phi_{\resi^\mu}^\lambda(A, \dots, A)$, we simply have $v_{\tabt_\mu^\lambda} \psi_k = 0$.
        \item \textbf{$N$ is a Leg node and $N'$ is an arm node of $[\mu]$:} In this case, $(\tabt_\mu^\lambda)^{-1}(k+1) < (\tabt_\mu^\lambda)^{-1}(k)$, so that $v_{\tabt_\mu^\lambda} \psi_k = v_{\tabt_\mu^\lambda s_k} (y_k^{x_1} + y_{k+1}^{x_2})$, where $x_1, x_2$ are $1$ or $2$ depending on the residues of $N$ and $N'$, according to \cref{T:PsiAction}. We have to check that the expression $v_{\tabt_\mu^\lambda s_k} (y_k^{x_1} + y_{k+1}^{x_2})$ evaluates to zero. 

    Let $\resi$ be the residue sequence of $\tabt^\mu$, let $\resi' = \resi s_k$, and let $\tabt' = \tabt_\mu^\lambda s_k$. Through the rest of the proof, we refer to arm and leg nodes of $[\mu]$ with respect to their positions in $\tabt'$ instead of $\tabt_\mu^\lambda$, which goes against our usual convention. Suppose, towards a contradiction, that $k$ is at the end of a double short residue subsequence $\bm{s}' \in DS(\resi')$. Let $(\tabt')^{-1}(\bm{s}')$ be the subset of nodes inside $[\la]$ filled with the numbers in $\bm{s}'$ in tableau $\tabt'$, and let $(\tabt^\mu)^{-1}(\bm{s}')$ be the set of corresponding nodes in tableau $\tabt^\mu$. Since $\resi'_{\bm{s}'}$ is a double short sequence of residues, it follows that for every arm node there must be a corresponding leg node with the same residue. Therefore, a leg node with the same residue as $N' \in (\tabt^\mu)^{-1}(\bm{s}')$ must also be in $(\tabt^\mu)^{-1}(\bm{s}')$. Since leg nodes are limited to the first column of $[\mu]$, this requires that $N^1 \in (\tabt^\mu)^{-1}(\bm{s}')$, where $N^1$ is the node immediately northwest of $N'$. Next, the node $(N^1)'$ immediately to the right of $N^1$ and immediately above $N'$ satisfies $\tabt^\mu(N^1) < \tabt^\mu((N^1)') < \tabt^\mu(N')$, so it follows that $(N^1)' \in (\tabt^\mu)^{-1}(\bm{s}')$. This in turn requires that $N^2 \in (\tabt^\mu)^{-1}(\bm{s}')$, where $N^2$ is the node immediately to the northwest of $(N^1)'$, and this reasoning can be continued by induction. Since there are only finitely many rows in $[\mu]$, it follows that $k$ cannot in fact be at the end of a double short residue subsequence $\bm{s}'\in DS(\resi')$. 
    
    Next, suppose that $k+1$ is at the start of a double short residue subsequence $\bm{s}' \in DS(\resi')$. Let $(\tabt')^{-1}(\bm{s}') \subset [\la]$ be as in the previous paragraph. Since $i'_{k+1} = \res(N)$, the first arm node in $(\tabt^\mu)^{-1}(\bm{s})$, which we call $N''$, has the same residue in $(\tabt^\mu)^{-1}(\bm{s}')$. We reason differently according to where $N''$ is located.
    \begin{enumerate}
        \item If the row starting with $N,N'$ has more nodes in it, then the node $N''$ is immediately to the right of $N'$. Since $\res(N) = \res(N'')$, it follows that the node $N'$ has residue $0$ or $e$, and therefore $x_2 = 2$, so the double crossing after $\psi_k$ leads to a factor of $y_{k+1}^2$, and we have $v_{\tabt'} y_{k+1}^2 = 0$, since $k+1$ can be inside, at most, a single double short sequence $\bm{s}'$ with corresponding choice $L$.
        \item Otherwise, $N''$ is to the southeast of $N$, and the node $N_1$ immediately to the south of $N$ satisfies $\tabt^\mu(N) < \tabt^\mu(N_1) < \tabt^\mu(N'')$, so that $N_1$ must also be in $(\tabt^\mu)^{-1}(\bm{s}')$, and then an induction similar to that in the previous paragraph may be pursued.
    \end{enumerate}
    \end{enumerate}
    
    It follows from the previous two paragraphs that $v_{\tabt_\mu^\lambda} \psi_k = v_{\tabt_\mu^\lambda s_k} (y_k^{x_1} + y_{k+1}^{x_2}) = 0$, finishing the proof.
\end{proof}

We have proved that $v_{\tabt_\mu^\lambda}$ satisfies the simple Specht relations given in \cref{E:specht} for all the shapes in \cref{lem:shapes}. Next we will find which of these shapes satisfy the Garnir relations given in \cref{E:GarnirC}. The following computation will be useful to us.
\begin{lem} \label{lem:garnir-computation}
    \[
    \begin{braid}
	\def\h{5};
	\def\sep{2}
	\def\a{7+\sep};
	\def\dashpart{0.9}
	\def\hrect{1.15}
	
	\path (0, 0) -- (0, \h) coordinate[pos=0.5] (A) ;
	\draw (0, 0) -- (A) -- (3 + 2*\sep, \h) node[anchor=south]{$i_1$};
	\draw (1, 0) -- (1, \h) node[anchor=south]{$i_2$};
	\draw(\sep/2 + 1.5, 0) node{$\cdots$};
	\draw(\sep/2 + 1.5, \h+\hrect/2) node{$\cdots$};
	\draw (2 + \sep, 0) -- (2 + \sep, \h) node[anchor=south]{$i_{2e}$};
	\draw (3 + 2*\sep, 0) -- (A) -- (0, \h) node[anchor=south]{$i_1$};
	\draw[brown] (-0.5, \h) rectangle (2.7 + \sep, \h + \hrect);
\end{braid}
    = - 
    \begin{braid}
	\def\h{5};
	\def\sep{2}
	\def\a{7+\sep};
	\def\dashpart{0.9}
	\def\hrect{1.15}
	
	\path (0, 0) -- (0, \h) coordinate[pos=0.5] (A) ;
	\draw (0, 0) -- (A) -- (0, \h) node[anchor=south]{$i_1$};
	\draw (1, 0) -- (1, \h) node[anchor=south]{$i_2$};
	\draw(\sep/2 + 1.5, 0) node{$\cdots$};
	\draw(\sep/2 + 1.5, \h+\hrect/2) node{$\cdots$};
	\draw (2 + \sep, 0) -- (2 + \sep, \h) node[anchor=south]{$i_{2e}$};
	\draw (3 + 2*\sep, 0) -- (3+2*\sep, \h) node[anchor=south]{$i_1$} coordinate[pos = 0.5] (B);
	\draw[brown] (-0.5, \h) rectangle (2.7 + \sep, \h + \hrect);
	\greendot(B);
\end{braid}
    \]
\end{lem}
\begin{proof}
    Suppose first that $i_1 \neq 0,1,e-1$ or $e$ and let $k \in \{2, \dots, 2e\}$ be characterised by the condition that $i_1 = i_k$. Then we see that
    \[
    \begin{braid}
	\def\h{5};
	\def\sep{2}
	\def\a{6+\sep};
	\def\dashpart{0.9}
	\def\hrect{1.15}
	
	\path (0, 0) -- (0, \h) coordinate[pos=0.5] (A) ;
	\draw (0, 0) -- (A) -- (3 + 2*\sep, \h) node[anchor=south]{$i_1$};
	\draw (1, 0) -- (1, \h) node[anchor=south]{$i_2$};
	\draw(\sep/2 + 1.5, 0) node{$\cdots$};
	\draw(\sep/2 + 1.5, \h+\hrect/2) node{$\cdots$};
	\draw (2 + \sep, 0) -- (2 + \sep, \h) node[anchor=south]{$i_{2e}$};
	\draw (3 + 2*\sep, 0) -- (A) -- (0, \h) node[anchor=south]{$i_1$};
	\draw[brown] (-0.5, \h) rectangle (2.7 + \sep, \h + \hrect);
    \end{braid}
    =
    \begin{braid}
	\def\h{5};
	\def\sep{2}
	\def\a{7+\sep};
	\def\dashpart{0.9}
	\def\hrect{1.15}
	
	\path (1, 0) -- (1, \h) coordinate[pos=0.5] (A) ;
	\draw (0, 0) -- (0, \h) node[anchor=south]{$i_1$};
	\draw (1, 0) -- (1, \h) node[anchor=south]{$i_2$};
	\draw (2, 0) -- (2, \h) node[anchor=south]{$i_3$};
	\draw(\sep/2 + 2.5, 0) node{$\cdots$};
	\draw(\sep/2 + 2.5, \h+\hrect/2) node{$\cdots$};
	\draw (3 + \sep, 0) -- (3 + \sep, \h) node[anchor=south]{$i_{2e}$};
	\draw (4 + 2*\sep, 0) -- (A) -- (4 + 2*\sep, \h) node[anchor=south]{$i_1$};
	\draw[brown] (-0.5, \h) rectangle (3.7 + \sep, \h + \hrect);
\end{braid}
    = 
    \begin{braid}
	\def\h{5};
	\def\sep{2}
	\def\a{6+\sep};
	\def\dashpart{0.9}
	\def\hrect{1.15}

	\draw (0, 0) -- (0, \h) node[anchor=south]{$i_1$};
	\draw(\sep/2 + 0.5, 0) node{$\cdots$};
	\draw(\sep/2 + 0.5, \h+\hrect/2) node{$\cdots$};
	\path (\sep, 0) -- (\sep, \h) coordinate[pos=0.5] (A) ;
	\draw(1 + \sep, 0) -- (1+\sep, \h) node[anchor=south]{$i_{k-1}$};
	\draw(2 + \sep, 0) -- (2+\sep, \h) node[anchor=south]{$i_{k}$};
	\draw (3 + \sep, 0) -- (3 + \sep, \h) node[anchor=south]{$i_{k+1}$};
	\draw (3.5 + 3*\sep/2, 0) node{$\cdots$};
	\draw (3.5 + 3*\sep/2, \h + \hrect/2) node{$\cdots$};
	\draw (4 + 2*\sep, 0) -- (4 + 2*\sep, \h) node[anchor=south]{$i_{2e}$};
	\draw (5 + 3*\sep, 0) -- (A) -- (5+3*\sep, \h) node[anchor=south]{$i_1$};
	\draw[brown] (-0.5, \h) rectangle (4.7 + 2*\sep, \h + \hrect);
\end{braid}
    \]
    \[
    =
    \begin{braid}
	\def\h{5};
	\def\sep{2}
	\def\a{6+\sep};
	\def\dashpart{0.9}
	\def\hrect{1.15}

	\draw (0, 0) -- (0, \h) node[anchor=south]{$i_1$};
	\draw(\sep/2 + 0.5, 0) node{$\cdots$};
	\draw(\sep/2 + 0.5, \h+\hrect/2) node{$\cdots$};
	\path (1+\sep, 0) -- (1+\sep, \h) coordinate[pos=0.5] (A) ;
	\draw(1 + \sep, 0) -- (1+\sep, \h) node[anchor=south]{$i_{k-1}$};
	\draw(2 + \sep, 0) -- (2+\sep, \h) node[anchor=south]{$i_{k}$};
	\draw (3 + \sep, 0) -- (3 + \sep, \h) node[anchor=south]{$i_{k+1}$};
	\draw (3.5 + 3*\sep/2, 0) node{$\cdots$};
	\draw (3.5 + 3*\sep/2, \h + \hrect/2) node{$\cdots$};
	\draw (4 + 2*\sep, 0) -- (4 + 2*\sep, \h) node[anchor=south]{$i_{2e}$};
	\draw (5 + 3*\sep, 0) -- (A) -- (5+3*\sep, \h) node[anchor=south]{$i_1$};
	\draw[brown] (-0.5, \h) rectangle (4.7 + 2*\sep, \h + \hrect);
	\greendot($(A) + (0.3,0)$);
\end{braid}
    =
    -
    \begin{braid}
	\def\h{5};
	\def\sep{2}
	\def\a{6+\sep};
	\def\dashpart{0.9}
	\def\hrect{1.15}
	
	\draw (0, 0) -- (0, \h) node[anchor=south]{$i_1$};
	\draw(\sep/2 + 0.5, 0) node{$\cdots$};
	\draw(\sep/2 + 0.5, \h+\hrect/2) node{$\cdots$};
	\path (2+\sep, 0) -- (2+\sep, \h) coordinate[pos=0.5] (A) ;
	\draw(1 + \sep, 0) -- (1+\sep, \h) node[anchor=south]{$i_{k-1}$};
	\draw(2 + \sep, 0) -- (A) -- (5+3*\sep, \h) node[anchor=south]{$i_{k}$};
	\draw (3 + \sep, 0) -- (3 + \sep, \h) node[anchor=south]{$i_{k+1}$};
	\draw (3.5 + 3*\sep/2, 0) node{$\cdots$};
	\draw (3.5 + 3*\sep/2, \h + \hrect/2) node{$\cdots$};
	\draw (4 + 2*\sep, 0) -- (4 + 2*\sep, \h) node[anchor=south]{$i_{2e}$};
	\draw (5 + 3*\sep, 0) -- (A) -- (2+\sep, \h) node[anchor=south]{$i_1$};
	\draw[brown] (-0.5, \h) rectangle (4.7 + 2*\sep, \h + \hrect);
\end{braid}
    \]
    \[= -
    \begin{braid}
	\def\h{5};
	\def\sep{1.7}
	\def\a{5+\sep};
	\def\dashpart{0.9}
	\def\hrect{1.15}

	\draw (0, 0) -- (0, \h) node[anchor=south]{$i_1$};
	\draw(\sep/2 + 0.5, 0) node{$\cdots$};
	\draw(\sep/2 + 0.5, \h+\hrect/2) node{$\cdots$};
	\path (3+\sep, 0) -- (3+\sep, \h) coordinate[pos=0.5] (A) ;
	\draw(1 + \sep, 0) -- (1+\sep, \h) node[anchor=south]{$i_{k-1}$};
	\draw(2 + \sep, 0) -- (2+\sep, \h) node[anchor=south]{$i_{k}$};
	\draw (3 + \sep, 0) -- (3 + \sep, \h) node[anchor=south]{$i_{k+1}$};
	\draw (3.5 + 3*\sep/2, 0) node{$\cdots$};
	\draw (3.5 + 3*\sep/2, \h + \hrect/2) node{$\cdots$};
	\draw (4 + 2*\sep, 0) -- (4 + 2*\sep, \h) node[anchor=south]{$i_{2e}$};
	\draw (5 + 3*\sep, 0) -- (A) -- (5+3*\sep, \h) node[anchor=south]{$i_1$};
	\draw[brown] (-0.5, \h) rectangle (4.7 + 2*\sep, \h + \hrect);
\end{braid}
= -
    \begin{braid}
	\def\h{5};
	\def\sep{1.7}
	\def\a{5+\sep};
	\def\dashpart{0.9}
	\def\hrect{1.15}

	\draw (0, 0) -- (0, \h) node[anchor=south]{$i_1$};
	\draw(\sep/2 + 0.5, 0) node{$\cdots$};
	\draw(\sep/2 + 0.5, \h+\hrect/2) node{$\cdots$};
	\path (3+2*\sep, 0) -- (3+2*\sep, \h) coordinate[pos=0.5] (A) ;
	\draw(1 + \sep, 0) -- (1+\sep, \h) node[anchor=south]{$i_{k-1}$};
	\draw(2 + \sep, 0) -- (2+\sep, \h) node[anchor=south]{$i_{k}$};
	\draw (3 + \sep, 0) -- (3 + \sep, \h) node[anchor=south]{$i_{k+1}$};
	\draw (3.5 + 3*\sep/2, 0) node{$\cdots$};
	\draw (3.5 + 3*\sep/2, \h + \hrect/2) node{$\cdots$};
	\draw (4 + 2*\sep, 0) -- (4 + 2*\sep, \h) node[anchor=south]{$i_{2e}$};
	\draw (5 + 3*\sep, 0) -- (A) -- (5+3*\sep, \h) node[anchor=south]{$i_1$};
	\draw[brown] (-0.5, \h) rectangle (4.7 + 2*\sep, \h + \hrect);
\end{braid}
    =   - \begin{braid}
	\def\h{5};
	\def\sep{1.7}
	\def\a{5+\sep};
	\def\dashpart{0.9}
	\def\hrect{1.15}
	
	\path (0, 0) -- (0, \h) coordinate[pos=0.5] (A) ;
	\draw (0, 0) -- (A) -- (0, \h) node[anchor=south]{$i_1$};
	\draw (1, 0) -- (1, \h) node[anchor=south]{$i_2$};
	\draw(\sep/2 + 1.5, 0) node{$\cdots$};
	\draw(\sep/2 + 1.5, \h+\hrect/2) node{$\cdots$};
	\draw (2 + \sep, 0) -- (2 + \sep, \h) node[anchor=south]{$i_{2e}$};
	\draw (3 + 2*\sep, 0) -- (3+2*\sep, \h) node[anchor=south]{$i_1$} coordinate[pos = 0.5] (B);
	\draw[brown] (-0.5, \h) rectangle (2.7 + \sep, \h + \hrect);
	\greendot(B);
\end{braid}
    \]
    Similar straightforward computations give the result in the case where $i_1 = \text{$0$ or $e$}$ and the case where $i_1 = \text{$1$ or $e-1$}$.
\end{proof}

The next lemma in relation to Garnir relations is analogous to \cite[Lemma 5.7]{loub17}. The similarity of the statement is striking, though the computations are more involved in our case. The notation $(a_1, \dots, a_k) \in \fkS_n$ simply denotes a cycle of the symmetric group.
\begin{lem} \label{lem:garnir}
    Suppose that $\mu$ has one of the shapes in \cref{lem:shapes} and let $A = (x, y) \in \mu$ be a Garnir node. Define $r, s, t$ to be the values $\tabt^\mu(x, y), \tabt^\mu(x+1, 1)$ and $\tabt^\mu(x + 1, y)$, respectively. Then
    \[
    v_{\tabt_\mu^\lambda} \psi^{\tabt^A} = \begin{cases}
        v_{\tabt_\mu^\lambda (t+1,t, \dots, s)} &\text{if $x \le b$, $y \equiv 0 \mod  2e$ and $y < \mu_{x+1}$},\\
        v_{\tabt_\mu^\lambda (r, r+1, \dots, s)} &\text{if $x \le b$, $y \equiv 1 \mod 2e$ and $y > 1$},\\
        v_{\tabt_\mu^\lambda} &\text{if $x > b$ and $y \equiv 0 \mod 2e$},\\
        0 &\text{otherwise.}
    \end{cases}
    \]
\end{lem}
\begin{proof}
    We take $j$ to be the content of node $A = (x,y)$. 
    
    First we assume that $x \le b$ and $y \equiv 0 \mod 2e$. Since $A$ is a Garnir node it follows that $\mu_x \ge \mu_{x+1} \ge 2e$, and according to \cref{lem:shapes} we must have $\mu_x = c(2e)$ for some $c \in \mathbb{Z}_{\ge 1}$. Consequently, $C^A$ consists of a single node while $D^A$ is empty. From this we may infer the structure of $\psi^{\tabt^A}$. As in \cite{loub17}, we obtain
    \begin{equation}
        v_{\tabt_\mu^\lambda} \psi^{\tabt^A} =  \begin{braid}
	\def\h{5};
	\def\sep{2}
	\def\a{7+\sep};
	\def\dashpart{0.9}
	\def\hrect{1.15}
	
	\draw (0, 0) node{$\cdots$};
	\draw (0, \h + \hrect/2) node{$\cdots$};
	\draw(1, 0) -- (1, \h);
	\path(2, 0) -- (2, \h) coordinate[pos=0.5] (A);
	\draw (2,0) -- (A) -- (12,\h) node[anchor=south]{$\overline{j}$};;
	\draw (3 , 0) -- (3, \h);
	\draw (4+\sep , 0) -- (4+\sep, \h);
	\draw (3.5 + \sep/2, 0) node{$\cdots$};
	\draw (3.5 + \sep/2, \h + \hrect/2) node{$\cdots$};
	\draw (5 + \sep, 0) node[anchor=north]{$t$} -- (A) -- (2, \h) node[anchor=south]{$\overline{j}$};
	\draw (5.5+3*\sep/2, 0) node{$\cdots$};
	\draw (5.5+3*\sep/2, \h + \hrect/2) node{$\cdots$};
	\draw[brown] (-0.7, \h) rectangle (5.2 + 2*\sep, \h + \hrect);
\end{braid}.
    \end{equation}
    Note that by construction, any co-stubborn string before $t$ is followed by several stubborn strings. It easily follows that any co-stubborn string before $t$ is immobile, and with this in mind we ignore them in the computations.
    After applying \cref{lem:garnir-computation} several times, we obtain
    \begin{equation}
        v_{\tabt_\mu^\lambda} \psi^{\tabt^A} = 
        \pm
    \begin{braid}
	\def\h{5};
	\def\sep{2}
	\def\a{7+\sep};
	\def\dashpart{0.9}
	\def\hrect{1.15}
	
	\draw (0, 0) node{$\cdots$};
	\draw (0, \h + \hrect/2) node{$\cdots$};
	\draw(1, 0) -- (1, \h);
	\draw(2, 0) -- (2, \h) node[anchor=south]{$\overline{j}$};
	\draw (3 , 0) -- (3, \h);
	\draw (4+\sep , 0) -- (4+\sep, \h);
	\draw (3.5 + \sep/2, 0) node{$\cdots$};
	\draw (3.5 + \sep/2, \h + \hrect/2) node{$\cdots$};
	\draw (5 + \sep, 0) node[anchor=north]{$t$}  -- (12, \h) coordinate[pos=0.2] (A) node[anchor=south]{$\overline{j}$};
	\greendot(A);
	\draw (5.5+3*\sep/2, 0) node{$\cdots$};
	\draw (5.5+3*\sep/2, \h + \hrect/2) node{$\cdots$};
	\draw[brown] (-0.7, \h) rectangle (5.2 + 2*\sep, \h + \hrect);
\end{braid}
    \coloneq \pm v_\tabs y_t.
    \end{equation}
    Here $\tabs = \tabt_\mu^\lambda (t, t-1, \dots, s)$ is a standard tableau. Next, notice that the residue sequence of unfilled nodes of $\Arm(\lambda)$ in $\tabs_{\downarrow (r-1)}$ continues $(\overline{j}, \overline{j+1}, \overline{j+2}, \dots)$, whereas the residue sequence of unfilled nodes of $\Leg(\lambda)$ continues $(\overline{j}, \overline{j-1}, \overline{j-2}, \dots)$. This is a consequence of the assumption that $y \equiv 0 \mod 2e$.
    After we have placed the $c(2e)$ numbers $r,r+1, \dots, t-1$ into $\Arm(\lambda)$ in $\tabs$, we again have that the residue sequence of unfilled nodes of $\Arm(\lambda)$ in $\tabs_{\downarrow (t-1)}$ continues $(\overline{j}, \overline{j+1}, \overline{j+2}, \dots)$, whereas the residue sequence of unfilled nodes of $\Leg(\lambda)$ continues $(\overline{j}, \overline{j-1}, \overline{j-2}, \dots)$. 
    
    It follows that, if string $t$ starts a sequence $\bm{s}^{(r)} \in DS(\resi(\tabs))$, then $\resi^{(r)} = (\overline{j}, \overline{j})$. But $i_{t+1}(\bm{s}) = \overline{j}$ and $\tabs^{-1}(t+1) \in \Arm(\lambda)$ happen simultaneously if and only if $y < \mu_{x+1}$. In that case, $y_t$ simply breaks the crossing between $t$ and $t+1$ as described in \cref{T:YAction} and we obtain $v_\tabs y_t = v_{\tabt_{\mu}^\lambda (t+1, t, \dots, s)}$. Otherwise, if $y = \mu_{x+1}$, then $v_\tabs y_t = 0$.
    
    Next we let $x \le b$ and $y \equiv 1 \mod 2e$. We first assume that $x \le b$ and $y = 1$. If $x=1$, then $v_{\tabt_\mu^\lambda} \psi^{\tabt^A} = 0$ due to the Garnir relation for the node $(1,1)$. Otherwise both $A = (x, 1)$ and $A' \coloneq (x+1, 1)$ are leg nodes, since $x \le b$. 
    
    If $\mu_x = 1$, then $\mu_{x+1} = 1$ also. The element $\psi^{\tabt^A}$ in this case is simply $\psi_r$, and we have $v_{\tabt_\mu^\lambda} \psi_r = 0$ according to case 3 of \cref{T:PsiAction}, since $\tabt_\mu^\lambda = \tabt_{\mathrm{max}}(\resi^\mu)$, that is, the binary word of $\tabt_\mu^\lambda$ is composed only of $A$'s.
    
    Suppose instead that $\mu_x > 1$. In that case, the node $A'' = (x,2)$ with residue $\overline{j+1}$ is an arm node. Following Loubert, we have
    \begin{equation}
         v_{\tabt_\mu^\lambda} \psi^{\tabt^A} = 
        \begin{braid}
    	\def\h{5};
    	\def\sep{2}
    	\def\a{8};
    	\def\dashpart{0.9}
    	\def\hrect{1.15}
    	
    	\draw (1, 0) node[anchor=north]{$r$}-- (\a+1, \h) node[anchor=south]{$\overline{j_{-1}}$};
    	\draw (2, 0)  -- (1, \h/8) -- (\a, \h) node[anchor=south]{$\overline{j}$};
    	\draw (3, 0) -- (1, \h) node[anchor=south]{$\overline{j_1}$};
    	\draw (3.5 + \sep/2, 0) node{$\cdots$};
    	\draw (4 + \sep, 0) node[anchor=north]{$s$} -- (2 + \sep, \h);
    	\draw (4.5 + 1.5*\sep, 0) node{$\cdots$};
    	\draw (2.5 + 1.5*\sep, \h+\hrect/2) node{$\cdots$};
    	\draw (1.5 + 0.5*\sep, \h+\hrect/2) node{$\cdots$};

        \draw[brown] (0.3, \h) rectangle (3.3+1.5*\sep, \h + \hrect);
        \end{braid}.
    \end{equation}
    
    We have used the notation $\overline{j_k} = \overline{j + k}$ for graphical considerations. 
    Let $\tabt'$ be the (standard) tableau corresponding to the part of this diagram above the $\psi_r$ crossing at the very bottom. The diagram above may only be non-zero if $r+1$ starts a designated sequence, which may only happen if $\overline{j-1} = \overline{j+1}$, so that $\overline{j} = 0$ or $e$, and in that case, the Slide term of string $r+1$ past the crossing $\chi_{r,r+2}$ yields zero, according to \cref{T:PsiAction}. Let us take $j = 0$ without loss of generality. Focusing on a fragment on the left side of the diagram, we find
    \[
        v_{\tabt_\mu^\lambda} \psi^{\tabt^A} = 
        \begin{braid}
        \def\h{5};
        \def\sep{2}
        \def\a{6};
        \def\dashpart{0.9}
        \def\hrect{1.15}
        
        \draw (0, 0) node[anchor=north]{$r$} -- (2 + \sep, \h) node[anchor=south]{$1$};
        \draw (2+\sep, 0) node[anchor=north]{$r+2$}-- (0, \h) node[anchor=south]{$1$};
        \draw (1, 0) -- (0,\h/4) -- (1 + \sep, \h) node[anchor=south]{$0$};
        \draw[brown] (-0.5, \h) rectangle (0.5, \h + \hrect);
        
        \end{braid}
        = -
        \begin{braid}
        \def\h{5};
        \def\sep{2}
        \def\a{6};
        \def\dashpart{0.9}
        \def\hrect{1.15}
        
        \draw (0, 0) node[anchor=north]{$r$} -- (0, \h) node[anchor=south]{$1$};
        \draw (1+\sep, 0) -- (1+\sep, \h) node[anchor=south]{$0$};
        \draw (2+\sep, 0) node[anchor=north]{$r+2$}-- (2 + \sep, \h) coordinate[pos=0.5] (A) node[anchor=south]{$1$};
        \greendot(A);
        \draw[brown] (-0.5, \h) rectangle (0.5, \h + \hrect);
        \end{braid}.
    \]
    
    Let $\tabt''$ be the (standard) tableau obtained from the diagram above by substituting the bottom braid relation by the Error term (and ignoring the dots). Then the number $r+2$ cannot start a designated sequence in $\tabt''$, because when $r+2$ is placed, the addable residue in the arm is either $0$ or $2$, while the residue of $r+2$ is $1$. It follows that $v_{\tabt''} y_{r+2} = 0$, so that $v_{\tabt_\mu^\lambda} \psi^{\tabt^A} = 0$.

        Next suppose that $y > 1$ such that $y \equiv 1 \mod 2e$. Following Loubert, we reach
        \begin{equation}
        v_{\tabt_\mu^\lambda} \psi^{\tabt^A}
        =
        \begin{braid}
    	\def\h{5};
    	\def\sep{2}
    	\def\a{8};
    	\def\dashpart{0.9}
    	\def\hrect{1.15}
    	
    	\draw (0, 0) -- (0, \h) node[anchor=south]{$\overline{j}$};
    	\draw (1, 0) node[anchor=north]{$r$}-- (\a+1, \h) node[anchor=south]{$\overline{j}$};
    	\draw (2, 0)  -- (1, \h) ;
    	\draw (2.5 + \sep/2, 0) node{$\cdots$};
    	\draw (3 + \sep, 0) node[anchor=north]{$s$} -- (2 + \sep, \h);
    	\draw (3.5 + 1.5*\sep, 0) node{$\cdots$};
    	\draw (\a + 1.5 + \sep/2, \h) node{$\cdots$};
    	\draw (2.5 + 1.5*\sep, \h+\hrect/2) node{$\cdots$};
    	\draw (1.5 + 0.5*\sep, \h+\hrect/2) node{$\cdots$};
    	\draw[brown] (-0.5, \h) rectangle (3.3+1.5*\sep, \h + \hrect);
        \end{braid},
        \end{equation}
        which is $v_{\tabt_\mu^\lambda (r, r+1, \dots, s)}$.
        
        Next we assume that $x \le b$ and $y = 2$. First, assume that $\mu_x = 2$. In that case, we have
        \[
        v_{\tabt_\mu^\lambda} \psi^{\tabt^A}
        =
        \begin{braid}
        	\def\h{5};
        	\def\sep{2}
        	\def\a{6};
        	\def\dashpart{0.9}
        	\def\hrect{1.15}
        	
        	\draw (0, 0) node[anchor=north]{$r$} -- (2, \h/4) -- (1, \h) node[anchor=south]{$\overline{j_{-1}}$};
        	\draw (1, 0) node[anchor=north]{$s$} -- (0, \h) node[anchor=south]{$\overline{j}$};
        	\draw (2, 0) node[anchor=north]{$t$} -- (1, \h/4) -- (\a, \h) node[anchor=south]{$\overline{j_{-2}}$};
        	\draw (1.5 + \sep/2, \h + \hrect/2) node{$\cdots$};
        	\draw (2 + \sep/2, 0) node{$\cdots$};
        	\draw (\a + 0.5 + \sep/2, \h + \hrect/2) node{$\cdots$};
        	\draw[brown] (-0.5, \h) rectangle (.3+1.5*\sep, \h + \hrect);
        \end{braid}
        = 0.
        \]
        The diagram evaluates to zero after the resolution of the double crossing between strings $r$ and $t$ due to the crossing between strings $r$ and $s$, according to \cref{E:specht}.
        
        Let us cover the last cases such that $x \le b$. We may assume $y \not\equiv 0, 1 \mod 2e$, and furthermore, if $y = 2$, then $\mu_x > 2$. Then, according to \cref{lem:shapes}, we have $\mu_x = c(2e)$ for some $c \in \mathbb{Z}_{>0}$. It follows that both $C^A$ and $D^A$ have at least two nodes and, as in \cite{loub17}, we reach a diagram which has a subdiagram at the top as follows.
        \[
        v_{\tabt_\mu^\lambda} \psi^{\tabt^A}
        =
        \begin{braid}
        	\def\h{6.5} 
        	\def\sep{2.6} 
        	\def\a{7.8}
        	\def\dashpart{0.9} 
        	\def\hrect{1.3} 
        	
        	\draw (0, 0) -- (2.6, \h/2) -- (5.2 + \sep, \h) node[anchor=south]{$\overline{j_{-1}}$};
        	\draw (1.3, 0) -- (3.9, \h/2) -- (2.6, \h) node[anchor=south]{$\overline{j}$};
        	\draw (3.9, 0) -- (0, \h/2) -- (0, \h) node[anchor=south]{$\overline{j_{-2}}$};
            \draw (5.2, 0) -- (1.3, \h/2) -- (1.3, \h) node[anchor=south]{$\overline{j_{-1}}$};
        	\draw (3.9, \h + \hrect/2) node{$\cdots$};
        	\draw (6.5, 0) node{$\cdots$};
            \draw (2.6, 0) node{$\cdots$};
        	\draw (\a + 0.65 + \sep/2, \h + \hrect/2) node{$\cdots$};
        	\draw[brown] (-0.9, \h) rectangle (4.6, \h + \hrect);
        \end{braid}
        \]
        This diagram evaluates to zero after computing the braid relation in the center due to a $\psi$- or $y$-relation, according to \cref{E:specht}.
        
        Next we check the special case $A = (b+1, 1)$. If $\mu_{b+1} = 1$, then the Garnir relation involves a single $\psi$, which gives a double crossing when applied to $v_{\tabt_\mu^\lambda}$. This gives zero due to arguments like those in the proof of \cref{L:simple_spechts_hom}. Next, note that we cannot have $\mu_{b+1} = 2$, as nodes $N',N'' \in [\mu]$ respectively to the immediate east and south of node $N$ do not have adjacent residues, and so they cannot both be consecutive arm nodes. Suppose instead that $\mu_{b+1} \geq 3$. Following Loubert, we obtain
        \[
        v_{\tabt_\mu^\lambda} \psi^{\tabt^A}
        =
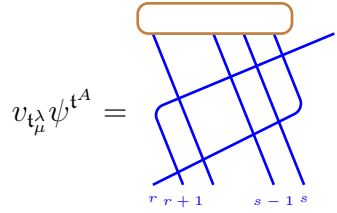
        \begin{braid}
        	\def\h{5} 
        	\def\sep{2} 
        	\def\a{6}
        	\def\dashpart{0.9} 
        	\def\hrect{1} 
        	
        	\draw (1,0) node[anchor=north]{$r+1$}-- (0, \h/2) -- (\a, \h);
            \draw (2,0) -- (1, \h/2) -- (0, \h);
            \draw (2 + \sep,0) node[anchor=north]{$s-1$} -- (1 + \sep, \h/2) -- (0 + \sep, \h);
            \draw (3 + \sep,0) node[anchor=north]{$s$} -- (2 + \sep, \h/2) -- (1 + \sep, \h);
            \draw (0,0)  node[anchor=north]{$r$} -- (3 + \sep, \h/2) -- (2 + \sep, \h);
            
        	\draw[brown] (-0.5, \h) rectangle (4.6, \h + \hrect);
        \end{braid}.
        \]
        This diagram gives zero: the crossing $\chi_{r,s}$ between strings $r$ and $s$ can slide up past the crossing $\chi_{r+1, s}$. The Slide term (according to \cref{E:braidC}) gives zero, and the Error term (if it exists) gives zero due to the leftover crossing $\chi_{r+1,s-1}$, which triggers a $\psi$-relation.

        Next we check the case where $x > b$ and $y \not\equiv 0 \mod 2e$, while also $y \neq 1$. In this case, every node $A'$ such that $A' > A$ is an arm node, and $C^A \neq \varnothing$, so that $\psi^{\tabt^A}$ is a non-trivial element. Recall that all leg nodes before $r$ are immobile, as they are followed by an arm node which immobilizes them. Therefore, any $\psi_m$ for $m \geq r$ annihilates $v_{\tabt_{\mu}^\lambda}$, so that $v_{\tabt_\mu^\lambda} \psi^{\tabt^A} = 0$.
        
        Finally, if $x > b$ and $y \equiv 0 \mod 2e$, then according to \cref{lem:shapes}, row $x$ has $c(2e) - 1$ nodes, so that $C^A = D^A = \varnothing$. It follows that $\tabt^A = \tabt^\mu$, so that $v_{\tabt_\mu^\lambda} \psi^{\tabt^A} = v_{\tabt_\mu^\lambda}$.
\end{proof}

If $\mu_k$ is a row of $\mu$, we write $\mu_k^A$ for the number of Arm nodes in $\mu_k$, which is either $\mu_k$ or $\mu_k - 1$. The \emph{ceiling $2e$-quotient} ${\bar{\mu}}$ of a partition $\mu$ is given by
\[{\bar{\mu}}_k = \bigceil*{\frac{\mu_k^A+1}{2e}}\]
The next lemma follows in the same way as \cite[Lemma 5.9]{loub17}. 

\begin{lem} Under the same conditions as the previous lemma, and for any $1 \leq m \leq \bar{\mu}_x$, the following holds.
\[
v_{\tabt_\mu^\lambda} g^A = \begin{cases}
    \binom{\bar{\mu}_x}{m} v_{\tabt_\mu^\lambda (t+1,t, \dots, s)} &\text{if $x \le b$, $A = (x,m(2e))$ and $\mu_{x+1} > m(2e),$}\\
    \binom{\bar{\mu}_x}{m} v_{\tabt_\mu^\lambda (r, r+1, \dots, s)} &\text{if $x \le b$, $A = (x,m(2e)+1)$ and $m > 0,$}\\
    \binom{\bar{\mu}_x}{m} v_{\tabt_\mu^\lambda} &\text{if $x > b$ and $A = (x,m(2e))$, and}\\
    0 &\text{otherwise.}
\end{cases}
\]
\end{lem}

Now the main result of the section follows just as it does in \cite{loub17}. 
We borrow the following piece of notation. Given a partition $\nu$, we define the \emph{Garnir content} of $\nu$ as follows.
\[
\Gc(\nu) = \Gc(\nu_1, \dots, \nu_m)=\gcd\left\{\binom{\nu_i}{k} \mid 1 \leq k \leq \nu_{i+1}-1, \ 1 \leq i \leq m-1\right\}.
\]
We use the convention that $\gcd(\varnothing) = 0$.

\begin{thm} \label{T:leg}
    Let $\la, \mu$ be partitions of $n$ such that $\la$ is a hook with a short leg, and consider only homomorphisms where $(1,2)$ is not a leg node of $[\mu]$. Then there exists a non-zero $\varphi \in \Hom_{R_n^\Lambda}(S^\mu, S^\la)$ if and only if $\mu$ is one of the shapes in \cref{lem:shapes} with ceiling $2e$-quotient $\bar{\mu}$ satisfying $\Gc(\bar{\mu}) = 0$.
\end{thm}

\begin{eg}
   Most partitions in \cref{eg:leg_residues} lead to homomorphisms. The exceptions are $\mu^1 = (8,6,1,1)$ and $\mu^2 = (8,5,1,1,1)$. We have $\bar{\mu}^1 = (2,2,1,1)$ and $\bar{\mu}^2 = (2,2,1,1,1)$. Both cases lead to $\Gc(\bar{\mu}) = 2$, so that the corresponding homomorphisms only exist in characteristic $2$.
\end{eg}

\subsection{The node \texorpdfstring{$(1,2)$}{(1,2)} is a leg node of \texorpdfstring{$[\mu]$}{[mu]}} \label{SS:12legnode}

Suppose that there is a non-zero homomorphism $\varphi \in \Hom(S^\mu, S^\lambda)$, and let $\tabt \coloneq \tabt_\mathrm{max}(\resi^\mu)$ so that $\varphi(v_{\tabt^\mu}) = v_{\tabt}$. In this subsection we assume that $(1,2)$ is a leg node of $[\mu]$. In particular, we must have $\overline{\kappa-1} = \overline{\kappa+1}$, so that the charge $\kappa$ must be $0$ or $e$, and without loss of generality we work in $    \Lambda = \Lambda_0$. 

\begin{lem}
    The partition $[\mu]$ must be a hook.
\end{lem}
\begin{proof}
    Suppose that $[\mu]$ is not a hook. Then $[\mu]$ contains the node $N = (2,2)$.
    This must be an arm node of $[\mu]$, because $N$ has residue zero and $[\lambda]$ has a short leg and charge zero.
    
    Now, let $r \coloneq \tabt^\mu (N)$ and $N' \coloneq \tabt^{-1}(r) \in [\lambda]$.
    Then $N'$ is separated from the corner node $(1,1)$ by at least $2e-1$ nodes, since they both have residue zero, and in particular the two nodes $A,A' \in [\lambda]$ immediately to the right of $(1,1)$ in $[\lambda]$ have residues $1$ and $2$. Let $\bar{A} = (\tabt^\mu)^{-1}(\tabt(A))$ and define $\bar{A}'$ similarly. Then $\bar{A}$ and $\bar{A}'$ are the first two arm nodes in $[\mu]$.
    
    It follows that the node $B \in [\mu]$ immediately to the left of $\bar{A}$, with residue $0$, is a leg node. But $(1,2)$ is also a leg node, so $A \neq (1,2)$, and therefore $B \neq (1,1)$. We have reached a contradiction, as both $(1,1)$ and $B$ are leg nodes of $[\mu]$ with residue $0$, despite the fact that $[\lambda]$ has a short leg.
\end{proof}

\begin{lem}\label{L:leg-length}
    The length of $\Leg(\la)$ must be strictly larger than the length of $\Arm(\la)$.
\end{lem}

\begin{proof}
    Suppose that the length of $\Arm(\la)$ is at least the same as the length of $\Leg(\lambda)$. In that case, there must exist $k\in \mathbb{Z}_{\ge 2}$ such that $\tabt_{\downarrow k}$ has the same amount of nodes in the arm as in the leg, as $(1,2)$ is a leg node. This implies that the sequence $(i_2, \dots, i_k)$ is a double short residue sequence. Since the number $2$ goes to the leg, it follows from \cref{Cor:YAction} that there must be a $k'$ such that $v y_{k'} \neq 0$, which contradicts our assumption that $v_\tabt = \varphi(v_{\tabt^\mu})$.
\end{proof}

\begin{thm} \label{T:short-arm-leg}
    Let $\la, \mu$ be partitions of $n$ such that $\la = (a+1, 1^b)$ is a hook with a short leg, and consider only homomorphisms where $(1,2)$ is a leg node of $[\mu]$, so that $\Lambda = \Lambda_0$ or $\Lambda_e$. Then there exists a non-zero $\varphi \in \Hom_{R_n^\Lambda}(S^\mu, S^\la)$ if and only if one of the following holds.
    \begin{enumerate}
        \item $\mu = (b+1,1^a)$, where $b>a$.
        \item $\mu = (2e)$ and $\lambda = (2,1^{2e-2})$.
        \item $\mu = (2e-1, 1)$ and $\lambda = (1^{2e})$.
    \end{enumerate}
\end{thm}

\begin{proof}
    First we establish the shape of $\mu$ according to the possible residue sequences. Recall that we have the condition $(1,2) \in [\mu_L]$. Let us split our investigation according to whether $(2,1) \in [\mu_L]$ or $(2,1) \notin [\mu_L]$.
    \begin{enumerate}
        \item Suppose first that $(2,1) \in [\mu_L]$. Since we have charge $0$, it follows that $\Leg(\lambda)$ must have the residue sequence $(1,2,\dots,1)$ with length $2e-1$, so that $b = 2e-1$. In this case, the partition $\mu_L$ must have $\resi^{\mu_L} = (0,1,2,\dots,1)$, so that $\mu_L = (2e-1, 1)$. In this case, the shape $[\mu_A]$ (if it is non-empty) must have a single node in $\Arm(\mu)$ and can continue in $\Leg(\mu)$. In this case we have $\mu = (b+1, 1^a)$.
        \item Suppose that $(2,1) \notin [\mu_L]$. In this case we simply have $\mu_L = (b+1)$. 
        \begin{enumerate}
            \item If $b \neq 2e-2$, then the first residue $1$ in the arm of $\lambda$ differs from the next addable node in the arm of $\mu_L$, so that $[\mu_A]$ must be confined to the leg of $\mu$. In this case we have $\mu = (b+1, 1^a)$.
            \item If $b = 2e-2$, then $[\mu_A]$ may confined to the leg as before, but it is also possible that $[\mu_A]$ is a single node in the arm of $[\mu]$. In this case, $\mu$ is the trivial partition $(2e)$.
        \end{enumerate}
    \end{enumerate}
    In the case $\mu = (b+1,1^a)$ with $b = 2e-1$ and $a \geq 1$, there are two candidates for $\mu_L$, that is, we may have $\mu_L = (2e-1,1)$ or $\mu_L = (2e)$. In fact, we must have $\mu_L = (2e-1,1)$. Indeed, if we had $\mu_L = 2e$ and $a \neq 0$, then the strings $2e$ and $2e+1$ cross in $v_{\tabt_{\mu}^\la}$ and they share the same residue, so that $v_{\tabt_{\mu}^\la} y_{2e} \neq 0$.
    
    In each case, the Garnir relation for the node $(1,1)$ is easily seen to hold, as all tableaux in $\Std(\lambda)$ have residue sequence starting with 0. The $y$-relations also hold, since the tableaux $\tabt_\mu^\la$ obtained from the above $[\mu_L]$ and $[\mu_A]$ can be checked to be maximal for their respective residue sequences. The $\psi$-relations (and the Garnir relations for nodes in $\Leg(\mu)$, which also involve a single $\psi$-element) hold in each case due to cases (2) and (3) in \cref{T:PsiAction}, since the $y$-action is zero. 
\end{proof}

\subsection{The trivial module} \label{SS:Trivial}

For reference, we give a version of \cref{T:leg} in the special case of the trivial module $S^\lambda = S^{(n)}$.

\begin{lem}\label{lem:shapes-t-module}
    Let $\mu=(\mu_1,\dots,\mu_m)$ be a partition of $n$ such that $m\ge 1$ with charge $\kappa$ and residue sequence $(i_1,i_2,\dots,i_n)=\resi^\mu$. Let $\lambda=(n)$.
    \begin{enumerate}
        \item Suppose that $\mu_1 = c(2e) - 1$ for some $c \in \mathbb{Z}_{>0}$. Then $\mu = (\mu_1, \overline{\mu})$ satisfies $\resi^{\tabt_\mu^\lambda} =\resi^\mu$ if and only if $\overline{\mu}$ satisfies $\resi^{\tabt_{\overline{\mu}}^{\overline{\lambda}}} =\resi^{\overline{\mu}}$, where $\overline{\lambda} = (n - \mu_1)$ with charge $\kappa - 1$.
        \item Suppose that $\mu_1$ is not as in the previous case.
        Then $\resi^{\tabt_\mu^\lambda} =\resi^\mu$ if and only if $\mu$ has one of the following shapes.
        \begin{enumerate}
            \item $\mu = \lambda = (n)$
            \item $\mu=(\mu_1,1^{n-\mu_1})$ and $-\mu_1 \equiv 2\kappa - 1 \mod 2e$. 
        \end{enumerate}
    \end{enumerate}
\end{lem}

\begin{thm}
    Let $\la, \mu$ be partitions of $n$ such that $\la=(n)$. Then there exists a non-zero $\varphi \in \Hom_{R_n^\Lambda}(S^\mu, S^{(n)})$ if and only if $\mu$ is one of the shapes in \cref{lem:shapes-t-module} with ceiling $2e$-quotient $\bar{\mu}$ satisfying $\Gc(\bar{\mu}) = 0$.
\end{thm}

\begin{eg}
    Let $\lambda = (12), \kappa = 0$.
    The list of partitions $\mu$ satisfying the condition in \cref{lem:shapes-t-module} can be computed as in \cref{eg:leg_residues}. Here we simply give the full list of partitions. These are $(12)$, $(11,1)$, $(9,1,1,1)$,
        $ (7,5)$,
        $(7,3,2)$,
        $(7,3,1,1)$,
         $(5,1^7)$,
         $(3,3,3,3)$,
        $ (3,3,1^6)$ and
        $(1^{12})$.
    Of these, all of them have $Gc(\bar{\mu}) = 0$ except for $\mu^* = (7, 5)$, for which $\Gc(\bar{\mu}^*) = \Gc(2,2)$ = 2. Therefore, all of the corresponding homomorphisms exist in any characteristic, except for the partition $\mu^*$. In that case $\Hom_{R_n^{\Lambda_0}}(S^{\mu^*}, S^{\la}) \neq 0$ only over a ring of characteristic $2$.
\end{eg}

\subsection{Connection between types A and C} \label{SS:AC-Connection}

In \cite[Section 4]{forsberg} we found that the homomorphism spaces between Specht modules  over different quiver Hecke algebras of type A (varying the quantum characteristic) are sometimes related.

Our results on homomorphisms between Specht modules of type $C_e^{(1)}$ suggest a further connection. Let us check our main result, \cref{T:leg}, against the results for homomorphisms to hooks in type $A_{2e-1}^{(1)}$. It is striking that the shapes given in \cite[Main Theorem]{loub17} are precisely the shapes that appear in \cref{T:leg} if we restrict our attention to the partitions that arise from (2.a) and (2.b) in \cref{lem:shapes}. Let us investigate an example.

\begin{eg}
Let $\lambda=(4,1^3)$, and let us look at the partitions $\mu$ such that the homomorphism space is non-zero under quiver Hecke algebras $R_n({\Gamma})$ for different quivers ${\Gamma}$.

\begin{enumerate}
    \item If $\Gamma = A_{3}^{(1)}$, then the only homomorphism is the identity, with $\mu = (4,1^3)$.
    \item If $\Gamma = C_2^{(1)}$, then we can split our search depending on the charge:
    \begin{enumerate}
        \item If the charge is $0$ or $2$, then again the non-zero homomorphism is the identity, and $\mu = (4,1^3)$.
        \item If the charge is $1$, then the following yield homomorphisms.
        \begin{itemize}
            \item $\mu = (4,1^3)$
            \item $\mu = (2^3,1)$
            \item $\mu = (2,1^6)$
        \end{itemize}
    \end{enumerate}
    \item If $\Gamma = A_{1}^{(1)}$, then several partitions yield homomorphisms, as follows.
    \begin{itemize}
        \item $\mu = (6,1)$
        \item $\mu = (4,3)$
        \item $\mu = (4,1^3)$
        \item $\mu = (2^3,1)$
        \item $\mu = (2,1^6)$
    \end{itemize}
    The partition in 1 appears in both (2.a) and (2.b), and each partition in (2.a) or (2.b) appears in 3. The above hold for an arbitrary field $k$, but there exist other examples where the characteristic of $k$ is relevant.
\end{enumerate}

We extrapolate our findings to all partitions $\mu, \lambda$.

\begin{conj} \label{Conj:Unfurling}
    Let $R_n(\Gamma,p)$ be the quiver Hecke algebra corresponding to quiver $\Gamma$ over a field of characteristic $p$. Let $\mu, \lambda \vdash n$ be arbitrary partitions.
    \begin{enumerate}
    \item Let $e \in \mathbb{Z}_{\geq 2}$. Whenever 
    \[\Hom_{R_n(A_{2e-1}^{(1)},p)}(S^\mu, S^\lambda) \neq \{0\},\] we also find that:
    \[ \Hom_{R_n(C_e^{(1)},p)}(S^\mu, S^\lambda) \neq \{0\} \quad \text{(for any charge $\kappa$).} \]
    \item Whenever \[ \Hom_{R_n(C_{2}^{(1)},2)}(S^\mu, S^\lambda) \neq \{0\} \quad \text{(for any charge $\kappa$)}, \] we also find that \[\Hom_{R_n(A_{1}^{(1)},2)}(S^\mu, S^\lambda) \neq \{0\}.\]
    \end{enumerate}
\end{conj}

This conjecture suggests a structure of quiver specialization, where a generic par\-a\-me\-trized quiver specializes to a quiver with fewer vertices under certain values of the parameter. This structure was suggested to the first author by Ben Webster. It is likely that the results in \cite[Section 4]{forsberg}, and perhaps the first part of \cref{Conj:Unfurling}, follow from \cite[Theorem 6.4]{Web16}, while the second part may require further investigation.
\end{eg}

\section{Homomorphisms when \texorpdfstring{$\lambda$}{lambda} has a Short Arm}  \label{S:ShortArm}

Let $\mu=(\mu_1,\dots,\mu_m)$ and $\la=(a+1,1^b)$ with $a< 2e$ be two partitions of $n$.
In this section, as in the previous one, we investigate homomorphisms from $S^\mu$ to $S^\la$, but now we restrict the amount of nodes in $\Arm(\lambda)$ instead of $\Leg(\lambda)$. The homomorphisms to $S^\la$ under the condition that $\Arm(\lambda)$ be short (that is, it has less than $2e$ nodes) also send the Specht module generator to a single basis element, as \cref{Cor:YAction} can also be applied in this case.

In the short arm case, the case where the node $(1,2) \in [\mu]$ is not a leg node becomes simpler, while the case where $(1,2)$ is a leg node becomes significantly harder. We quickly cover the simpler case, and we discuss the difficulties in the harder case by going over two examples.

\subsection{The node \texorpdfstring{$(1,2)$}{(1,2)} is not a leg node of \texorpdfstring{$[\mu]$}{[mu]}}

Recall from \cref{SS:12legnode} that this case covers every homomorphism whenever $\Lambda \neq \Lambda_0$ or $\Lambda_e$. Applying \cref{L:firstcol}, we see that all of $\mu$'s leg nodes appear in its first column. Define $\tabt_\mu^\lambda$ as in the previous section (before \cref{L:y_rels_hom}).

\begin{lem}\label{L:shortarm-shapes-other}
    We have that $\resi^\mu=\resi^{\tabt^\la_\mu}$ if and only if $\mu$ is one of the following partitions:
    \begin{enumerate}
            \item $\mu = \lambda = (a+1, 1^b)$
            \item $\mu = (a,1^{b+1})$ and $a \equiv -b-1 \mod 2e$
            \item $\mu=(\mu_1,1^{n-\mu_1})$ and $b - \mu_1 \equiv 2\kappa - 1 \mod 2e$.
            \item $\mu = (c(2e) - 2\kappa, 2^y, 1^{b-y})$ for some $c \in \mathbb{Z}_{\ge 0}$ and $y \in \{1, 2, \dots, b\}$.
            \item $\mu = (c(2e) - 2\kappa, 2^y, 1^{b-y + 1})$ for some $c \in \mathbb{Z}_{\ge 0}$, $b - y \equiv 2\kappa - 1 \mod 2e$ and $y \in \{1, 2, \dots, b\}$.
    \end{enumerate}
\end{lem}

\begin{proof}
    Similar to the proof of \cref{lem:shapes}, with the single simplification that the arm of $\lambda$ is short, which eliminates the possibility of $\mu_1$ having a multiple of $2e$ nodes, and also the possibility of $\mu_1$ having $c(2e)-1$ nodes, except in the case $c = 1$ and $\mu = \lambda$, which we already cover.
\end{proof}

It is left to check which of the above $\tabt_\mu^\lambda$ satisfy the Specht module relations of $S^\mu$. In fact, they are satisfied \emph{for all} the shapes above, as one can see by retracing the proofs of \cref{L:simple_spechts_hom,lem:garnir}, which are slightly simpler in the current case, as the possible shapes of $\mu$ are comparatively restricted.

\begin{thm}
    Let $\lambda = (a+1,1^b)$ be a partition with a short arm. Then there is a homomorphism from $S^\mu$ to $S^\lambda$ such that $(1,2) \notin [\mu_L]$ if and only if the partition $\mu$ has one of the shapes appearing in \cref{L:shortarm-shapes-other}.
\end{thm}

\subsection{The node \texorpdfstring{$(1,2)$}{(1,2)} is a leg node of \texorpdfstring{$[\mu]$}{[mu]}}

As in \cref{SS:12legnode}, this case can only present if $\Lambda = \Lambda_0$ or $\Lambda_e$. It is significantly more difficult to analyze, and we do not give a full classification. According to \cref{Cor:YAction}, the nodes in $[\mu]$ can be separated into $[\mu_L]$ and $[\mu_A]$, but in this case the subset $[\mu_L]$ is not restricted to the first column of $[\mu]$. This greatly increases the possibilities for partitions $\mu$ with potential Specht module homomorphisms to $S^\lambda$. 

\begin{eg}
Let $e = 3, \kappa = (0), \lambda = (4,1^5)$. Then the partitions $\mu$ such that there is a homomorphism from $S^\mu$ to $S^\lambda$ are as follows:
\begin{itemize}
    \item $\mu = (2,1^7)$ gives a homomorphism where $(1,2) \in [\mu]$ is an arm node.
    \item $\mu = (4,1^5)$ gives the identity homomorphism.
    \item $\mu = (4,3,1,1)$
    \item $\mu = (4,3,2)$
    \item $\mu = (6,1,1,1)$
    \item $\mu = (8,1)$ 
\end{itemize}
Let us focus on the case $\mu = (4,3,2)$. The homomorphism sends the initial tableau 
\[
\tabt^\mu = 
\ColorTableau[
  1/1/blue!30, 2/1/blue!30, 3/1/blue!30,
  1/2/blue!30, 2/2/blue!30, 3/2/blue!30
]{{1,2,3,4},{5,6,7},{8,9}}
\]
where $[\mu_L]$ has been shaded in blue, to the tableau
\[
\tabt^\lambda_\mu([\mu_L]) =
\ColorTableau[
  1/1/blue!30, 1/2/blue!30, 1/3/blue!30, 1/4/blue!30, 1/5/blue!30, 1/6/blue!30,
]{{1,4,8,9},{2},{3},{5},{6},{7}}.
\]
We see that, since we no longer have an analog of \cref{L:legnodes}, we lack control on $\mu_L$, and as a result, in general we obtain many more candidates for the partition $\mu$.
\end{eg}    

\begin{eg}
Let $e = 3, \kappa = (0), \lambda = (2,1^8)$. One candidate for a homomorphism is $\mu=(5,3,2)$, since the choice $\mu_L = (5,3,1)$ gives the correct residue sequence. We fill in the nodes with residues and shade the nodes in our candidate $[\mu_L]$ (which we call $\nu$) in blue.
\[
[\mu] = 
\ColorTableau[
  1/1/blue!30, 2/1/blue!30, 3/1/blue!30, 4/1/blue!30, 5/1/blue!30,
  1/2/blue!30, 2/2/blue!30, 3/2/blue!30,
  1/3/blue!30
]{{0,1,2,3,2},{1,0,1},{2,1}}
\]
This matches the residues of $[\lambda]$.
\[
[\lambda] = 
\ColorTableau[
  1/1/blue!30, 1/2/blue!30, 1/3/blue!30, 1/4/blue!30, 1/5/blue!30, 1/6/blue!30, 1/7/blue!30, 1/8/blue!30, 1/9/blue!30, 
]{{0,1},{1},{2},{3},{2},{1},{0},{1},{2}}
\]
However, the Specht relation $v_{\tabt^\mu} \psi_{10} = 0$ is not satisfied for $v_{\tabt_\mu^\lambda(\nu)}$, as we see next.
\[
 \begin{braid}
	\def\h{5};
	\def\sep{0.7};
	\def\a{15+\sep};
	\def\dashpart{0.9};
	
	% 1. Draw all standard blue lines
	\draw[blue] (0, 0) -- (0, \h) node[anchor = south]{$0$};
	\draw[blue] (1, 0) -- (2 + \sep, \h) node[anchor = south]{$1$};
	\draw[blue] (2, 0) -- (3 + \sep, \h) node[anchor = south]{$2$};
	\draw[blue] (3, 0) -- (4 + \sep, \h) node[anchor = south]{$3$};
	\draw[blue] (4, 0) -- (5 + \sep, \h) node[anchor = south]{$2$};
	\draw[blue] (5, 0) -- (6 + \sep, \h) node[anchor = south]{$1$};
	\draw[blue] (6, 0) -- (7 + \sep, \h) node[anchor = south]{$0$};
	\draw[blue] (7, 0) -- (8 + \sep, \h) node[anchor = south]{$1$};
    
	% 2. Draw the two special lines entirely in blue first
	\draw[blue] (8, 0) -- (9 + \sep, 1) -- (1, \h) node[anchor = south]{$2$};
	\draw[blue] (9+\sep, 0) -- (8, 1) -- (9 + \sep, \h) node[anchor = south]{$1$};

	% 3. Overdraw the bottom section (y=0 to y=1) in orange using a clip
	\begin{scope}
		\clip (-1, 0) rectangle (11, 1); 
		\draw[orange] (8, 0) -- (9 + \sep, 1) -- (1, \h);
		\draw[orange] (9+\sep, 0) -- (8, 1) -- (9 + \sep, \h);
	\end{scope}
    
	% Brown rectangles
	\draw[brown] (-0.5, \h) rectangle (1.5 , \h + 1);
	\draw[brown] (1.5 + \sep, \h) rectangle (9.5 + \sep, \h + 1);
\end{braid}
=
 \begin{braid}
	\def\h{5};
	\def\sep{0.7};
	\def\a{15+\sep};
	\def\dashpart{0.9};
	
	\draw (0, 0) -- (0, \h) node[anchor = south]{$0$};
	\draw (1, 0) -- (2 + \sep, \h) node[anchor = south]{$1$};
	\draw (2, 0) -- (3 + \sep, \h) node[anchor = south]{$2$};
    \draw (3, 0) -- (4 + \sep, \h) node[anchor = south]{$3$};
    \draw (4, 0) -- (5 + \sep, \h) node[anchor = south]{$2$};
    \draw (5, 0) -- (6 + \sep, \h) node[anchor = south]{$1$};
    \draw (6, 0) -- (7 + \sep, \h) node[anchor = south]{$0$};
    \draw (7, 0) -- (8 + \sep, \h) node[anchor = south]{$1$};
    \draw (8+\sep, 0)  -- (1, \h) coordinate[pos=0.1] (B) node[anchor = south]{$1$};
    \draw (9+\sep, 0) -- (9+\sep, \h) node[anchor = south]{$2$};
    
	\draw[brown] (-0.5, \h) rectangle (1.5 , \h + 1);
    \draw[brown] (1.5 + \sep, \h) rectangle (9.5 + \sep, \h + 1);
    \greendot(B);
\end{braid}
=
 \begin{braid}
	\def\h{5};
	\def\sep{0.7};
	\def\a{15+\sep};
	\def\dashpart{0.9};
	
	\draw (0, 0) -- (0, \h) node[anchor = south]{$0$};
	\draw (1, 0) -- (2 + \sep, \h) node[anchor = south]{$1$};
	\draw (2, 0) -- (3 + \sep, \h) node[anchor = south]{$2$};
    \draw (3, 0) -- (4 + \sep, \h) node[anchor = south]{$3$};
    \draw (4, 0) -- (5 + \sep, \h) node[anchor = south]{$2$};
    \draw (5, 0) -- (6 + \sep, \h) node[anchor = south]{$1$};
    \draw (6, 0) -- (7 + \sep, \h) node[anchor = south]{$0$};
    \draw (7, 0) -- (1, \h) node[anchor = south]{$1$};
    \draw (8+\sep, 0)  -- (8+\sep, \h) node[anchor = south]{$1$};
    \draw (9+\sep, 0) -- (9+\sep, \h) node[anchor = south]{$2$};
    
	\draw[brown] (-0.5, \h) rectangle (1.5 , \h + 1);
    \draw[brown] (1.5 + \sep, \h) rectangle (9.5 + \sep, \h + 1);
\end{braid}
\neq 0.
\]
It follows that there is no non-zero homomorphism from $S^\mu$ to $S^\lambda$. In \cref{S:ShortLeg}, we saw that whenever there is a candidate $\mu$ with the correct residue sequence, there is a homomorphism unless the Garnir relations prevent it. However, this no longer holds in the short arm case, as we have just seen.
\end{eg}
The last two examples show that care is needed when enumerating the homomorphisms to a hook $\lambda$ with a short arm. On the one hand, there are more candidate source partitions $\mu$. On the other hand, some plausible partitions do not produce a homomorphism. A full description of the partitions $\mu$ such that $\Hom_{R_n}(S^\mu,S^\lambda) \neq \{0\}$ in the spirit of \cref{T:leg} would therefore be significantly harder to achieve, but we do not expect it to be outside the scope of the methods used in this text.

\bibliographystyle{amsalpha}
\bibliography{master}

\end{document}